\let\ORIlabel\label
\let\ORIrefstepcounter\refstepcounter
\AddToHook{package/hyperref/before}{%
  \let\label\ORIlabel
  \let\refstepcounter\ORIrefstepcounter
}

\documentclass[onefignum,onetabnum]{siamonline171218}

\usepackage{mathpazo,mathtools}

\usepackage{lipsum}
\usepackage{amsfonts}
\usepackage{graphicx}
\usepackage{epstopdf}
\usepackage{algorithmic}
\usepackage{amssymb}
\ifpdf
  \DeclareGraphicsExtensions{.eps,.pdf,.png,.jpg}
\else
  \DeclareGraphicsExtensions{.eps}
\fi
\newcommand{\dd}{\mathrm{d}}
\newcommand{\ee}{\mathrm{e}}
\newcommand{\ii}{\mathrm{i}}
\newcommand{\phe}{\varphi}

\usepackage{enumitem}
\setlist[enumerate]{leftmargin=.5in}
\setlist[itemize]{leftmargin=.5in}

\newsiamremark{remark}{Remark}
\newsiamremark{hypothesis}{Hypothesis}
\crefname{hypothesis}{Hypothesis}{Hypotheses}
\newsiamthm{claim}{Claim}
\newsiamthm{example}{Example}
\newsiamthm{siamprop}{Proposition}

\headers{Applications of an Explicit Formula for the Benjamin-Ono Equation}{L. Gassot, P. G\'erard, and P. D. Miller}

\title{An Explicit Formula for the Solution of the Benjamin-Ono Equation and its Applications
}

\author{Louise Gassot\thanks{CNRS and Department of Mathematics, University of Rennes, Rennes, France
	(\email{louise.gassot@cnrs.fr}). L.\@ G.\@ was partially supported by the France 2030 framework program, the Centre Henri Lebesgue ANR-11-LABX-0020-01, and the ANR project HEAD--ANR-24-CE40-3260.}
 \and
 Patrick G\'erard\thanks{Laboratoire de Math\'ematiques d'Orsay, Universit\'e Paris Saclay, Paris, France (\email{patrick.gerard@universite-paris-saclay.fr}). P.\@ G.\@  was partially supported by the French Agence Nationale de la Recherche under the ANR project ISAAC--ANR-23--CE40-0015.}
\and Peter D. Miller\thanks{Department of Mathematics, University of Michigan, Ann Arbor, MI 
  (\email{millerpd@umich.edu}).  P.\@ D.\@ M.\@ was partially supported by the National Science Foundation under grant DMS-2508694.}
}

\usepackage{subcaption}

\usepackage{amsopn}

\makeatletter
\newcommand*{\addFileDependency}[1]{
  \typeout{(#1)}
  \@addtofilelist{#1}
  \IfFileExists{#1}{}{\typeout{No file #1.}}
}
\makeatother

\newcommand*{\myexternaldocument}[1]{%
    \externaldocument{#1}%
    \addFileDependency{#1.tex}%
    \addFileDependency{#1.aux}%
}

\ifpdf
\hypersetup{
  pdftitle={An Explicit Formula for the Solution of the Benjamin-Ono Equation and its Applications},
  pdfauthor={Louise Gassot and Patrick Gérard and Peter D. Miller}
}
\fi

\myexternaldocument{ex_supplement}

\usepackage{tikz}
\usetikzlibrary{arrows.meta}
\usetikzlibrary{decorations.pathmorphing}
\usetikzlibrary{decorations.markings}
\usetikzlibrary{patterns}
\usetikzlibrary{automata}
\usetikzlibrary{positioning}
\usepackage{tikz-cd}

\tikzset{->-/.style={decoration={markings,mark=at position #1 with {\arrow[thick]{>}}},postaction={decorate}}}
	
\tikzset{-<-/.style={decoration={markings,mark=at position #1 with{\arrow[thick]{<}}},postaction={decorate}}}

\newcommand{\uC}{C}

\newcommand{\MA}{\mathbf{A}}

\newcommand{\ME}{\mathbf{E}}

\newcommand{\Mq}{\mathbf{q}}
\newcommand{\Md}{\mathbf{d}}

\newcommand{\Mb}{\mathbf{b}}

\newcommand{\My}{\mathbf{y}}

\newcommand{\R}{\mathbb{R}}
\newcommand{\C}{\mathbb{C}}

\newcommand{\longleftarroww}[1] {\mathop{\longleftarrow}\limits^{#1}}

\title{An Explicit Formula for the Solution of the Benjamin-Ono Equation and its Applications}

\begin{document}

\maketitle

\begin{abstract}
We summarize recent developments in the theory of the Benjamin-Ono equation, a well-known long-wave asymptotic model for internal waves in deep water, focusing on results obtained by means of an explicit formula for the solution of the Cauchy problem in terms of given initial data.  We give a full description of this formula, and provide simple examples, with the aim of bringing it into the realm of applied mathematics.  We also show how the formula simplifies further upon restriction to rational initial data.  We then show how the formula and its rational restriction explain in great detail, and in a strikingly simple fashion, the asymptotic behavior of solutions of the Benjamin-Ono equation in the small-dispersion and long-time limits.  In the setting of long-time asymptotics, we prove some new results related to soliton resolution yielding pointwise convergence with a concrete decay rate under the assumption of rational initial data. The first half of the paper is a stand-alone general survey and user's guide, while the appendices constituting the second half are intended for readers who want to explore the topics in greater depth, and it includes the proofs of our new results.
\end{abstract}



\section{Introduction}
The Benjamin-Ono equation
\begin{equation}
\begin{aligned}\partial_t u + \partial_x(u^2)&=\varepsilon\partial_x|D|u,\quad \widehat{|D|u}(\diamond)=|\!\diamond\!|\cdot\widehat{u}(\diamond)\\
&=-\frac{\varepsilon}{\pi}\mathrm{P.V.}\int_\mathbb{R}\frac{\partial_y^2u\,\dd y}{y-x}
\end{aligned}
\label{eq:BO}
\end{equation}
is an asymptotic model for internal waves in a one-dimensional fluid channel, and is one of the well-known soliton equations.
Here, $\varepsilon>0$ is a parameter measuring the strength of dispersion relative to nonlinearity.  It can be scaled out of the equation by making the independent variable substitutions $t=\varepsilon\bar{t}$ and $x=\varepsilon\bar{x}$.  It is important however that such a substitution does not leave the initial data invariant; if we are given fixed (independent of $\varepsilon$) Cauchy data 
\begin{equation}
    u(0,x):=u_0(x),\quad x\in\mathbb{R}, \quad \text{$u_0$ real valued and decaying as $x\to\pm\infty$},
    \label{eq:IC}
\end{equation}
then the qualitative behavior of the corresponding solution $u(t,x)$ of \eqref{eq:BO} depends on $\varepsilon$ in a nontrivial way that we will describe below.

The asymptotic limit leading to \eqref{eq:BO} is that of small-amplitude vertical motions of the sharp interface between two fluid layers of different density (the \emph{pycnocline}, also called the \emph{thermocline} or \emph{halocline}, in the cases of density-stratified fluids with layers of different temperature or salinity) and disturbances of long wavelength, with the lower, denser, fluid layer assumed to be infinitely deep.  This derivation first appeared in a paper of Benjamin \cite{Benjamin67}, and its solitary wave solutions (solitons) and periodic traveling wave solutions were first studied there and also in the paper of Ono \cite{Ono75}.  A more comprehensive derivation that also leads under more general assumptions to other models such as the intermediate long-wave equation was given in \cite{ChoiC96}.  It appears that the first  fully rigorous mathematical derivation of the Benjamin-Ono equation from the stratified fluid equations first appeared in the paper of Paulsen \cite{Paulsen24}.  A comparison of predictions of the Benjamin-Ono equation and of other long-wave models such as the Korteweg-de Vries equation with physical experiments was given in \cite{KoopB81}.  A particularly dramatic phenomenon thought to be modeled by the Benjamin-Ono equation is the ``morning glory cloud'' that has been identified as an atmospheric internal bore wave generated by flows over the Cape York peninsula and appearing on southern spring mornings near Burketown on the Gulf of Carpentaria, Australia \cite{ChristieMC81}. The use of the (forced) Benjamin-Ono equation as a model for morning glory clouds is discussed in \cite{PorterS02}.

There are many properties of the Benjamin-Ono equation \eqref{eq:BO} and its Cauchy problem with initial data \eqref{eq:IC} that have been first established using  methods of analysis applicable to whole families of dispersive nonlinear equations.  A rigorous theory of its traveling wave solutions including both solitons and periodic waves was advanced by Amick and Toland \cite{AmickT91} in 1991.  The Amick-Toland theory proves that Benjamin's characterization in \cite{Benjamin67} of traveling waves for \eqref{eq:BO} was indeed both correct and complete.  In terms of functional-analytic properties of the Cauchy problem \eqref{eq:BO}--\eqref{eq:IC},  Saut \cite{Saut79} proved global existence and uniqueness to the Cauchy problem in $H^2(\R)$. Local well-posedness for \eqref{eq:BO} was established in 1986 by I\'orio in \cite{Iorio86} on Sobolev spaces $H^s(\R)$ with index $s>\frac{3}{2}$. In the 1989 paper of Abdelouhab, Bona, Felland, and Saut \cite{AbdelouhabBFS89}, similar methodology was used to extend Saut's global wellposedness and I\'orio's local result to global well-posedness of \eqref{eq:BO} on $H^s(\R)$ for $s>\frac{3}{2}$.  Two years later, the proof of this result was shortened in a paper of Ponce \cite{Ponce91}, and the Sobolev index condition was lowered to include the endpoint $s=\frac{3}{2}$. This regularity threshold was imporevd in 2003, by Kenig and Koenig \cite{KenigK03} who revisited local well-posedness applying energy methods and smoothing, and by Koch and Tzvetkov \cite{KochT03} in an independent paper.  A breakthrough discovered by Tao in 2004 \cite{Tao04} was a gauge transformation that allowed the Koch-Tzvetkov approach to be used to extend global well-posedness to the Sobolev space $H^1(\R)$.  Tao's transformation was further used by Ionescu and Kenig \cite{IonescuK07} in 2007 to obtain a global well-posedness result in $H^s(\R)$ for all $s\ge 0$ in the context of Bourgain's spaces.  A new proof of this result with some generalization was later given by Molinet and Pilod \cite{MolinetP12} in 2012.  
Using the integrability structure of the equation, discussed further in the next paragraph, these results were later improved to negative regularity thresholds 
in the sense of the continuous extension of the flow map. 
The first result in this sense concerns the equation on the torus and  is due to the second author, Kappeler and Topalov \cite{GKT23} in 2023, using a Birkhoff map construction. A similar result on the line was obtained  by  Killip-Laurens-Vi\c{s}an \cite{KillipLV24} in 2024, who built on methods used in the 2021 PhD thesis of Talbut \cite{Talbut21} also at the $s=0$ level and applied the method of commuting flows to prove the well-posedness of the Cauchy problem in $H^s(\R)$ for $s>-\frac{1}{2}$.  This is the optimal result for Sobolev spaces because $s=-\frac{1}{2}$ is the scaling-critical exponent for \eqref{eq:BO}.

As one of the famous soliton equations, the Benjamin-Ono equation comes with additional mathematical structure, the most important part of which is its representation in terms of the linear operators of a Lax pair found in 1979 independently by Bock and Kruskal \cite[Section 4]{BockK79} and by Nakamura \cite[Section 5]{Nakamura79}.  In 1983, an inverse-scattering formalism based on this Lax pair was proposed by Fokas and Ablowitz \cite{FokasA83}.  An approach based on formal power series was developed in 1985 by Anderson and Taflin \cite{AndersonT85}.  Coifman and Wickerhauser \cite{CoifmanW90} improved this method in 1990 working with convergent Neumann series that gave a solution of the Cauchy problem \eqref{eq:BO}--\eqref{eq:IC} for small initial data.  The Fokas-Ablowitz theory was revisited in 1998 by Kaup and Matsuno \cite{KaupM98}, who filled in several gaps and found new relations on the scattering data.  In 2016, Miller and Wetzel \cite{MillerW16a} explicitly computed all of the Fokas-Ablowitz scattering data under the assumption that $u_0$ is a rational function in $L^1(\R)$.  In two papers shortly thereafter \cite{Wu16a,Wu17}, Wu finally made the direct scattering problem in the Fokas-Ablowitz theory fully rigorous for initial data in weighted $L^2(\R)$ spaces.  To date, however, the inverse problem (which takes the form of a scalar nonlocal Riemann-Hilbert problem in the Fokas-Ablowitz theory) has not been fully justified for large data.  A different mathematical approach to soliton equations is Hirota's bilinear method, which has been applied to the Benjamin-Ono equation by Matsuno \cite{Matsuno84}.  

The book of Klein and Saut \cite{KleinS21nonlinear} contains an extensive survey of results on the Benjamin-Ono equation of the sort briefly described above, i.e, both by general PDE methods and by inverse-scattering methods. By contrast, this paper is a survey of recent results obtained for the initial-value problem \eqref{eq:BO}--\eqref{eq:IC} for the Benjamin-Ono equation with decaying initial data $u_0$ on $\R$ by a novel and completely different approach, namely using an explicit formula for its solution that was recently discovered by one of the authors in \cite{Gerard22}.  This formula is based fundamentally on the Lax pair structure of the equation, but it bypasses all manipulations behind inverse-scattering transforms, and in particular it sidesteps the question of solvability of the Fokas-Ablowitz nonlocal Riemann-Hilbert problem. The explicit formula will be described in Section~\ref{sec:explicit}.  Then, in Section~\ref{sec:small-dispersion} we will describe the results of using the explicit formula to analyze the behavior of the solution $u(t,x)$ of \eqref{eq:BO}--\eqref{eq:IC} in the small-dispersion limit $\varepsilon\downarrow 0$.  Next, we turn to the large-time limit $t\to\pm\infty$ for fixed dispersion $\varepsilon>0$ and fixed initial data $u_0$.  In Section~\ref{sec:large-time} we present a soliton-resolution result for a general class of real-valued initial data in Sobolev spaces and we formulate a new version of this result including pointwise rates of decay for error terms assuming the stronger condition (but dense in Sobolev spaces) that the initial data is rational.  Some longer proofs and auxiliary calculations can be found in the Appendices.  

The research described in this paper is very recent, and new results are still coming.  Hence this work is meant to describe only the current state of the art.

\section{Explicit solution of the Cauchy problem}
\label{sec:explicit}
A remarkable feature of the Benjamin-Ono equation is the existence of an explicit formula for the solution of the Cauchy problem \eqref{eq:BO}--\eqref{eq:IC} first derived in~\cite{Gerard22}.   In this paper, we focus entirely on the Benjamin-Ono equation posed for $x\in\mathbb{R}$ with decaying boundary conditions at $x=\pm\infty$, but there is also an explicit formula applicable to the case of periodic boundary conditions on an interval (also first derived in \cite{Gerard22}).
 Only a handful of soliton equations are known to be solvable in such an explicit way; other examples include the cubic Szeg\H{o} equation \cite{GerardG15,GerardP24}, the generalized half-wave maps equation \cite{GerardL25, GerardL25HWMtorus}, and the equation known variously as the Calogero-Sutherland derivative nonlinear Schr\"odinger equation, the continuum Calogero-Moser equation, or the infinite-depth limit of the intermediate nonlinear Schr\"odinger equation \cite{Badreddine24gwp, KillipLV25}.

\subsection{Ingredients}
\subsubsection{The Hardy space \texorpdfstring{$L^2_+(\mathbb{R})$}{ } and the Cauchy-Szeg\H{o} projection \texorpdfstring{$\Pi:L^2(\R)\to L^2_+(\R)$}{}}
The Bock-Kruskal Lax pair \cite{BockK79} involves two different unknown functions that admit analytic continuation from the real line $x\in\R$ to the upper and lower half-planes respectively.  Hence, a natural function space on which to define a Lax operator is the Hardy space $L_+^2(\R)$ defined as follows:
\begin{equation}
    L^2_+(\R):=\{f\in L^2(\R)\mid \widehat{f}(\xi)=0, \forall \xi<0\},
    \label{eq:Hardy-Fourier}
\end{equation}
where our notation for the Fourier transform on $L^1(\R)$ is
\begin{equation}
    \hat{f}(\xi):=\int_\R f(x)\ee^{-\ii \xi x}\,\dd x,\quad f\in L^1(\R),\quad \xi\in\R.
\end{equation}
It is clear from Fourier inversion that any $f\in L^2_+(\R)$ admits analytic continuation to the upper half-plane $\mathbb{C}_+:=\{z\in\C:\mathrm{Im}(z)>0\}$.  Indeed, the following integral is absolutely convergent and differentiable with respect to $z$ on $\C_+$:
\begin{equation}
    f(z)=\frac{1}{2\pi}\int_\R \widehat{f}(\xi)\ee^{\ii z\xi}\,\dd\xi = \frac{1}{2\pi}\int_0^{+\infty}\widehat{f}(\xi)\ee^{\ii z\xi}\,\dd\xi,\quad \mathrm{Im}(z)>0,
    \label{eq:continuation}
\end{equation}
defining the claimed analytic continuation.  The function defined in \eqref{eq:continuation} has the property that
\begin{equation}
    \sup_{y>0}\int_\mathbb{R}|f(x+\ii y)|^2\,\dd x<\infty,
    \label{eq:horizontalnorms}
\end{equation}
which leads to an equivalent definition of $L_+^2(\R)$:
\begin{equation}
    L_+^2(\R)=\{f:\mathbb{C}_+\to\mathbb{C}\;\text{analytic} \mid  \text{condition \eqref{eq:horizontalnorms} holds}\}.
\end{equation}
The space $L_+^2(\R)$ is a closed subspace of $L^2(\R)$, and it is a Hilbert space with the standard inner product inherited from $L^2(\R)$:
\begin{equation}
    \langle f,g\rangle:=\int_\R f(x)g(x)^*\,\dd x=\frac{1}{2\pi}\int_0^{+\infty}\widehat{f}(\xi)\widehat{g}(\xi)^*\,\dd\xi = \frac{1}{2\pi}\langle\widehat{f},\widehat{g}\rangle,\quad f,g\in L_+^2(\R).
    \label{eq:inner-product-convention}
\end{equation}

Let $1_{\ge 0}$ denote the characteristic function of the positive half-line $\xi\ge 0$.  The Cauchy-Szeg\H{o} projection $\Pi:L^2(\R)\to L^2_+(\R)$ is defined on the Fourier side by
\begin{equation}
    \widehat{\Pi f}(\diamond):=1_{\ge 0}(\diamond)\widehat{f}(\diamond),\quad f\in L^2(\R),
    \label{eq:Pi-Fourier}
\end{equation}
and an equivalent definition can be given in terms of a singular integral:
\begin{equation}
    \Pi f(x):=\frac{1}{2\pi\ii}\lim_{\delta\downarrow 0}\int_\mathbb{R}\frac{f(y)\,\dd y}{y-(x+\ii\delta)},\quad f\in L^2(\R),
\label{eq:Pi-singular-integral}
\end{equation}
with the limit existing for almost every $x\in\R$.  It is clear from \eqref{eq:Hardy-Fourier} and \eqref{eq:Pi-Fourier} that $\Pi f\in L^2_+(\R)$ and that $\Pi^2=\Pi\circ\Pi=\Pi$ on $L^2(\R)$.  Also, it is easy to see that $\Pi$ is a bounded self-adjoint linear operator on $L^2(\R)$.  Finally, we note that for real-valued functions $f\in L^2(\R)$, the Fourier transform is fully determined from its positive-frequency part, which means that $\Pi$ is invertible for such $f$.  Indeed, the complementary projection coincides with complex conjugation of $\Pi f$:  $(\mathrm{Id}-\Pi)f(x)=(\Pi f(x))^*$ for almost every $x\in\mathbb{R}$.  Hence 
\begin{equation}
    f(x)=2\mathrm{Re}(\Pi f(x)),\quad f\in L^2(\R), \quad\text{$f$ real valued}.
\label{eq:Pif-real}
\end{equation}

\subsubsection{The linear operator \texorpdfstring{$\mathcal{X}^*$}{X*} and the functional \texorpdfstring{$I_+$}{I+} on \texorpdfstring{$L^2_+(\mathbb{R})$}{L2+}}
\label{sec:Xstar-and-Iplus}
Let $\mathcal{X}$ denote the linear operator on $L^2_+(\R)$ of multiplication by the independent variable $x$, with domain
\begin{equation}
\begin{aligned}
\mathrm{dom}(\mathcal{X})
	&=\{h\in L^2_+(\R) \mid \diamond\, h(\diamond)\in L^2(\R)\}
	\\
	&=\{h\in L^2_+(\R) \mid \widehat{h}\in H^1(0,+\infty),\; \widehat{h}(0)=0\}.
\end{aligned}
\label{eq:domain-of-X}
\end{equation}
Note that actually $\diamond\,h(\diamond) \in L^2_+(\R)\subset L^2(\R)$ for $h\in \mathrm{dom}(\mathcal{X})$.

\begin{definition}[The linear operator $\mathcal{X}^*$]
We define $\mathcal{X}^*$ as the adjoint of $\mathcal{X}$ in $L^2_+(\R)$. 
\label{def:Xstar}
\end{definition}
While it is a familiar fact that being a multiplication by a real-valued function, $\mathcal{X}$ is self-adjoint on $L^2(\R)$, this is no longer the case on the smaller Hilbert space $L_+^2(\R)$.  To better understand this point and also obtain a formula for the action of $\mathcal{X}^*$, 
let us recall some elementary properties of $\mathcal{X}^*$. By definition, the domain of the operator $\mathcal{X}^*$ is the set of functions $f\in L^2_+(\R)$ such that there is $C>0$ for which $|\langle f, \mathcal{X} h\rangle|\leq C\|h\|_{L^2}$ holds for every $h\in\mathrm{dom}(\mathcal{X})$. By density, one can establish such a bound by testing only against functions $h\in L^2_+(\R)$ such that $\widehat{h}\in C^\infty_0(\R)$. Writing the inner product on the Fourier side and then integrating by parts, one can see that
\begin{equation}
\langle  f, \mathcal{X} h\rangle=\frac{1}{2\pi}\langle\widehat{f},\widehat{\mathcal{X}h}\rangle
	=\frac{1}{2\pi}\left\langle \widehat{f},\ii\frac{\dd \widehat{h}}{\dd \xi} \right\rangle.
\end{equation} 
The Riesz theorem therefore implies that if $f\in \mathrm{dom}(\mathcal{X}^*)$, then $\widehat{f}$ admits a weak derivative in $L^2(0,+\infty)$. Conversely, we can check that $\mathcal{X}^* f$ is characterized by
\begin{equation}
\mathrm{dom}(\mathcal{X}^*)
	=\{f\in L^2_+(\R) \mid \widehat{f}
    \in H^1(0,+\infty)\},
\label{eq:domain-of-X^*}
\end{equation}
a domain larger than $\mathrm{dom}(\mathcal{X})$ (cf., \eqref{eq:domain-of-X}) as it is missing the condition $\widehat{f}(0)=0$,
and the formula
\begin{equation}
\widehat{\mathcal{X}^* f}(\xi)=\ii\frac{\dd\widehat{f}}{\dd \xi}(\xi),\quad \xi>0.
\label{eq:X-star-hat}
\end{equation}
Taking the inverse Fourier transform, and integrating by parts gives
\begin{equation}
\mathcal{X}^* f(x)=\frac{1}{2\pi }\int_0^{\infty}\ii\frac{\dd\widehat{f}}{\dd \xi}(\xi)\ee^{\ii x \xi}\,\dd \xi
=xf(x)+\frac{1}{2\pi\ii }I_+[f],
\end{equation}
where $I_+$ denotes a linear functional defined on $\mathrm{dom}(I_+)=\{f\in L_+^2(\R)\vert \exists\delta>0,\widehat{f}\in C([0,\delta))\}$ by
\begin{equation}
I_+[f]:=\widehat{f}(0^+)=\lim_{\xi\downarrow 0}\widehat{f}(\xi).
\label{eq:Iplus}
\end{equation}
Note that $H^1(0,+\infty)$ is embedded in $C([0,+\infty))$ so that $I_+[f]$ makes sense for every $f\in\mathrm{dom}(\mathcal{X}^*)$, i.e., $\mathrm{dom}(\mathcal{X}^*)\subset\mathrm{dom}(I_+)$.
This gives an alternative characterization of $\mathrm{dom}(\mathcal{X}^*)$:
\begin{equation}
\mathrm{dom}(\mathcal{X}^*)=\{f\in L^2_+(\R)\mid \exists \mathfrak{a}=\mathfrak{a}[f]\in\C,\; \diamond\,f(\diamond)-\mathfrak{a}\in L^2(\R)\},
\label{eq:domain-Xstar-alternate}
\end{equation}
and of the action of $\mathcal{X}^*$:
\begin{equation}
\mathcal{X}^*f(x)=xf(x)-\mathfrak{a}[f],\quad \mathfrak{a}[f]=-\frac{1}{2\pi\ii}I_+[f].
\label{eq:Xstar-action-x}
\end{equation}

In a different approach, one can view the operator $\mathcal{X}$ as the infinitesimal generator of the Lax-Beurling semigroup~\cite{Lax59} on $L^2_+(\R)$ defined by
\begin{equation}
\mathcal{S}(\eta)f(x)=\ee^{\ii\eta x}f(x),\quad \eta\geq 0,
\end{equation}
or equivalently $\widehat{\mathcal{S}(\eta) f}(\xi)=\widehat{f}(\xi-\eta)$ for $\xi\in\R$ so that $\mathcal{S}(\eta)$ is a right-shift by $\eta$ in the Fourier variable. Note that the range of $\mathcal{S}(\eta)$ is contained in the space of functions whose Fourier transforms are supported in $[\eta,\infty)$, hence $\mathcal{S}(\eta)$ is not unitary. Then, we have
\begin{equation}
\ee^{-\ii \eta \mathcal{X}^*}=\mathcal{S}(\eta)^*
\end{equation}
with $\mathcal{S}(\eta)^*f=\Pi (\ee^{-\ii\eta \diamond}f(\diamond))$ and $\widehat{\mathcal{S}(\eta)^*f}(\xi)=\widehat{f}(\xi+\eta) 1_{\geq 0}(\xi)$. From the semigroup theory, we retrieve the property that $f\in L^2_+(\R)$ belongs to $\mathrm{dom}(\mathcal{X}^*)$ if and only if $\widehat{f}\in H^1(0,+\infty)$ just as in \eqref{eq:domain-of-X^*}, in which case 
\begin{equation}
\widehat{\mathcal{X}^* f}(\xi)=\ii\frac{\dd }{\dd \eta}\widehat{f}(\xi+\eta)|_{\eta=0}=\ii\frac{\dd \widehat{ f}}{\dd \xi}(\xi)
\end{equation}
in full agreement with \eqref{eq:X-star-hat}.

\subsubsection{The Lax operator \texorpdfstring{$\mathcal{L}_{u_0}$}{ } on \texorpdfstring{$L^2_+(\R)$}{L2+}}
The nonlocality of the Benjamin-Ono equation is captured at the level of the Lax operator through the action of a Toeplitz operator $\mathcal{T}_{u_0}$ on $L^2_+(\R)$ with symbol $u_0\in L^\infty(\R)$, defined by
\begin{equation}
    \mathcal{T}_{u_0}f = \Pi(u_0(\diamond)f(\diamond)),\quad f\in L^2_+(\R).
\end{equation}
The nonlocal nature of this operator is clear upon using the formula \eqref{eq:Pi-singular-integral}.
Then, the Lax operator $\mathcal{L}_{u_0}$ for the initial data $u_0\in L^\infty(\R)$ has domain 
\begin{equation}
\mathrm{dom}(\mathcal{L}_{u_0})=H^1_+(\R) := L_+^2(\R)\cap H^1(\R),    
\end{equation} 
and its action is given by
\begin{equation}
    \mathcal{L}_{u_0}f(x):=-\ii\varepsilon\frac{\dd f}{\dd x}(x)-\mathcal{T}_{u_0}f(x),\quad f\in \mathrm{dom}(\mathcal{L}_{u_0}).
\label{eq:Lax-operator}
\end{equation}
Here $\varepsilon>0$ is the dispersion coefficient in \eqref{eq:BO}.  Unlike $\mathcal{X}$,  the Lax operator $\mathcal{L}_{u_0}$ is indeed self-adjoint on the Hardy space $L_+^2(\R)$ with the indicated domain.

\subsection{The solution formula and recovery of the initial data}
Now we are ready to explain the explicit formula for the solution of the Cauchy problem \eqref{eq:BO}--\eqref{eq:IC} for suitable initial data $u_0$.
\begin{theorem}[Explicit formula, see~\cite{Gerard23}]
Let $u_0\in H^2(\mathbb{R})$ be given real-valued initial data, and let $u\in C(\mathbb{R},H^2(\mathbb{R}))$ be the solution to \eqref{eq:BO} with initial condition \eqref{eq:IC}. Then for every $t\in\mathbb{R}$, the analytic continuation of $x\mapsto \Pi u(t,x)$ as a function of $x\in\mathbb{R}$ to $z=x+\ii y\in \mathbb{C}_+$ satisfies 
\begin{equation}
    \Pi u(t,z)=\frac{1}{2\pi\ii}I_+\left[(\mathcal{X}^*-2t\mathcal{L}_{u_0}-z\,\mathrm{Id})^{-1}\Pi u_0\right],\quad z\in\mathbb{C}_+.
\label{eq:explicitformula}
\end{equation}
\label{thm:explicitformula}
\end{theorem}
Taking the limit $y\downarrow 0$ so that $z\to x\in\mathbb{R}$ and using \eqref{eq:Pif-real}, we recover the real-valued solution $u(t,x)$ of \eqref{eq:BO} with real initial data $u_0$ as
\begin{equation}
    u(t,x)=2\mathrm{Re}(\Pi u(t,x)),\quad (t,x)\in\mathbb{R}^2.
\end{equation}

Note that the computation required given $(t,z)$ is the inversion of the linear operator $\mathcal{X}^*-2t\mathcal{L}_{u_0}-z\mathrm{Id}$, which depends explicitly and linearly on $(t,z)$ as coefficients/parameters.  The relevant space is the Hardy space $L_+^2(\R)$, and after acting with the inverse operator on $\Pi u_0\in L_+^2(\R)$, the resulting function necessarily lies in $\mathrm{dom}(I_+)$, as shown at the end of Section 3 in~\cite{Gerard22}. Equivalently, one solves for $f\in L_+^2(\R)$ the equation $\mathcal{X}^*f-2t\mathcal{L}_{u_0}f-zf=\Pi u_0$ and then $f\in \mathrm{dom}(I_+)$.  One then simply computes the finite limit of its Fourier transform from the right at frequency $\xi=0$ to obtain a numerical value depending on the parameters $(t,z)$, namely the value of $\Pi u(t,z)$.    

The formula \eqref{eq:explicitformula} extends to initial data in $L^\infty(\R)\cap L^2(\R)$, see~\cite[Remark 7]{Gerard22}. For a further extension for initial data with lower regularity, $u_0\in L^2(\R)$, see~\cite{Chen25explicitBO} (see also~\cite{Chen2026explicitCCM} for initial data with nonvanishing conditions at infinity for the Calogero-Moser derivative NLS equation).  Neither of these extensions changes the formula \eqref{eq:explicitformula}, only the details of the interpretation of the various terms.  The analogue of Theorem~\ref{thm:explicitformula} applicable in the case of periodic boundary conditions (also first found in \cite{Gerard22}) has been recast as a numerical method \cite{BronsardCD26} for efficiently and accurately computing $u(t,x)$ at given $(t,x)\in\R^2$ without any need for time-marching from $t=0$, and this approach can also be used with large periods to approximate solutions of \eqref{eq:BO} on the line $\R$.  Finally, a generalization of Theorem~\ref{thm:explicitformula} for solving the $n$th equation in the Benjamin-Ono hierarchy with real-valued initial data $u_0\in H^s(\R)$ for $s>n+\frac{1}{2}$ was recently obtained by G\'erard and He in \cite{GerardH26}; one only needs to replace $2t\mathcal{L}_{u_0}$ with $(n+1)t\mathcal{L}_{u_0}^{n}$.

\begin{remark}[The explicit formula for $t=0$ and the Cauchy integral formula]
It is instructive to consider how the explicit formula \eqref{eq:explicitformula} returns the initial data $u(0,x)=u_0(x)$.  Setting $t=0$ in \eqref{eq:explicitformula} yields
\begin{equation}
    \Pi u(0,z)=\frac{1}{2\pi\ii}I_+\left[(\mathcal{X}^*-z\,\mathrm{Id})^{-1}\Pi u_0\right],
    \label{eq:Piu(0)}
\end{equation}
or in other words, using $y\in\mathbb{R}$ as the independent variable of functions $f\in L^2_+(\R)$,
\begin{equation}
    \Pi u(0,z)=\frac{1}{2\pi\ii}I_+[f],\quad \mathcal{X}^*f(y)-zf(y)=\Pi u_0(y),\quad z\in\mathbb{C}_+.
\end{equation}
Using \eqref{eq:Xstar-action-x} to express the 
action of $\mathcal{X}^*$ we get the equation $(y-z)f(y)-\mathfrak{a}[f]=\Pi u_0(y)$, where $\mathfrak{a}[f]=-I_+[f]/(2\pi\ii)$ is a constant depending linearly on $f$. 
Solving explicitly for $f(y)$, we obtain
\begin{equation}
    f(y)=\frac{\Pi u_0(y)+\mathfrak{a}[f]}{y-z}.
    \label{eq:f-t-zero}
\end{equation}
It remains to determine $\mathfrak{a}[f]$.  Since the parameter $z$ is in the upper half-plane $\mathbb{C}_+$ and $f\in L^2_+(\mathbb{R})$ must extend to an analytic function of its argument $y$ in the same half-plane, where $\Pi u_0$ is also analytic, it is necessary that the constant $\mathfrak{a}[f]$ takes the $z$-dependent value $\mathfrak{a}[f]=-\Pi u_0(z)$ to cancel the root of the denominator in \eqref{eq:f-t-zero}.  Thus 
\begin{equation}
\Pi u(0,z)=\frac{1}{2\pi\ii}I_+[f]=-\mathfrak{a}[f]=\Pi u_0(z)    
\end{equation}
holds for all $z\in\mathbb{C}_+$, which then implies that $u(0,\diamond)=u_0(\diamond)$ in $L^2(\mathbb{R})$. 

An interesting observation is that upon replacing $\Pi u(0,z)$ on the left-hand side of \eqref{eq:Piu(0)} with the equivalent expression $g(z):=\Pi u_0(z)$ and likewise replacing $\Pi u_0$ with $g\in L_+^2(\R)$ on the right-hand side yields the identity
\begin{equation}
    g(z)=\frac{1}{2\pi\ii}I_+\left[(\mathcal{X}^*-z\mathrm{Id})^{-1}g\right],\quad g\in L_+^2(\R), \quad z\in\mathbb{C}_+.
\label{eq:Cauchy-generalization}
\end{equation}
In fact, this is equivalent to the familiar Cauchy integral formula
\begin{equation}
    g(z)=\frac{1}{2\pi\ii}\int_\mathbb{R}\frac{g(x)\,\dd x}{x-z},\quad g\in L^2_+(\R),\quad z\in\mathbb{C}_+.
    \label{eq:Cauchy}
\end{equation}
Indeed, we may give a direct proof that the right-hand sides of \eqref{eq:Cauchy-generalization} and of \eqref{eq:Cauchy} agree.
By the definition of $I_+$, the right-hand side of \eqref{eq:Cauchy-generalization} reads
\begin{equation}
    \frac{1}{2\pi\ii}\lim_{\xi\downarrow 0}\hat{f}(\xi),\quad \text{where}\quad\mathcal{X}^*f-zf=g.
\end{equation}
Taking the Fourier transform of $\mathcal{X}^*f-zf=g$ using \eqref{eq:X-star-hat} gives the first-order differential equation
\begin{equation}
    \ii\frac{\dd \widehat{f}}{\dd\xi}(\xi)-z\widehat{f}(\xi)=\widehat{g}(\xi),\quad\xi>0.
\end{equation}
Using the integrating factor $\ee^{\ii z\xi}$, which is in $L^2(0,+\infty)$ as a function of $\xi$ because $z\in\C_+$, we obtain
\begin{equation}
    \ii\frac{\dd}{\dd\xi}\left(\ee^{\ii z\xi}\widehat{f}(\xi)\right)=\ee^{\ii z\xi}\widehat{g}(\xi),\quad \xi>0.
\end{equation}
For $z\in\C_+$ and $\widehat{g}\in L^2(0,+\infty)$, the right-hand side is in $L^1(0,+\infty)$ as a function of $\xi$, so we have
\begin{equation}
    \widehat{f}(\xi)=\ii\ee^{-\ii z\xi}\int_{\xi}^{+\infty}\ee^{\ii z\zeta}\widehat{g}(\zeta)\,\dd\zeta,\quad \xi>0.
\end{equation}
The right-hand side of \eqref{eq:Cauchy-generalization} therefore reads
\begin{equation}
    \frac{1}{2\pi\ii}\lim_{\xi\downarrow 0}\widehat{f}(\xi)=\frac{1}{2\pi}\int_0^{+\infty}\ee^{\ii z\zeta}\widehat{g}(\zeta)\,\dd\zeta = \frac{1}{2\pi} \langle \widehat{g},\ee^{-\ii z^*\diamond}\rangle.
\end{equation}
But $\ee^{-\ii z^*\xi}1_{\ge 0}(\xi)$ is the Fourier transform of $x\mapsto -[2\pi\ii(x-z^*)]^{-1}\in L_+^2(\R)$, so coming back from the Fourier side using \eqref{eq:inner-product-convention},
\begin{equation}
    \frac{1}{2\pi\ii}\lim_{\xi\downarrow 0}\widehat{f}(\xi)=\left\langle g,-\frac{1}{2\pi\ii (\diamond-z^*)}\right\rangle=\frac{1}{2\pi\ii}\int_\R\frac{g(x)\,\dd x}{x-z},
\end{equation}
proving directly that \eqref{eq:Cauchy-generalization} is the same as \eqref{eq:Cauchy}.
\end{remark}

\subsection{The case of rational initial data}
\label{sec:rational}
A reasonable person might object to the word ``explicit'' being used to describe the solution formula \eqref{eq:explicitformula}.  After all, implementing it apparently requires the inversion of a linear operator including both a differential term and a nonlocal Toeplitz term, the symbol of which is a general initial condition function $u_0$.  The operator itself is explicit, but it is not clear how to compute the inverse, or the degree to which it can be done explicitly.  In this section, we will demonstrate that for any $u_0$ in a dense subspace  of $L^2(\R,\R)$ (unfortunately not invariant under the flow of \eqref{eq:BO}), the formula \eqref{eq:explicitformula} indeed reduces to a completely explicit expression, namely a ratio of finite determinants with entries that are contour integrals with simple integrands.

\subsubsection{Implementation of Theorem~\ref{thm:explicitformula} for rational $u_0$}
The following reduction first appeared in \cite{BlackstoneGGM24a}.  Consider the dense subspace of $L^2(\R,\R)$ consisting of rational functions $u_0:\R\to\R$ with simple poles in complex-conjugate pairs off the real line; each such $u_0$ can be written in the form
\begin{equation}
    u_0(x)=\Pi u_0(x)+\Pi u_0(x)^*,\quad \Pi u_0(x)^*=\sum_{n=1}^N\frac{c_n}{x-p_n}
\label{eq:PN-rational-IC}
\end{equation}
with distinct poles $p_1,\dots,p_N\in \mathbb{C}_+$ and arbitrary nonzero complex coefficients $c_1,\dots,c_N$.
To express the solution $u(t,x)$ of the Cauchy problem \eqref{eq:BO}--\eqref{eq:IC} for such a rational initial condition $u_0$ and $t>0$, we first define a system of $N+1$ oriented contours $C_0,\dots,C_N$ in the complex plane punctured at the poles of $u_0$ as follows.  
\begin{itemize}
    \item The contour $C_0$ passes from $z=\infty\ee^{3\pi\ii/4}$ to $z=\infty\ee^{-\pi\ii/4}$ with the poles $\{p_1,\dots,p_N\}$ on its left and the poles $\{p_1^*,\dots,p_N^*\}$ on its right (left/right means relative to the orientation of $C_0$). 
    \item For $j=1,\dots,N$, 
    \begin{itemize}
        \item if $-\ii c_j/\varepsilon$ is not a positive integer, then $C_j$ passes from $z=\infty\ee^{3\pi\ii/4}$ to $z=\infty\ee^{3\pi\ii/4}$ with the poles $\{p_1,\dots,p_j\}$ on its left and all remaining poles on its right, and we say that $j$ is a \emph{nonexceptional index};
        \item if $-\ii c_j/\varepsilon$ is a positive integer, then $C_j$ is a contour in $\C_+$ originating from $z=\infty\ee^{3\pi\ii/4}$ and tending in the forward direction to $z=p_j$, and we say that $j$ is an \emph{exceptional index}.
    \end{itemize}
\end{itemize}
See Figure~\ref{fig:contours}.
\begin{figure}[h]
\begin{center}
    \begin{tikzpicture}[>=Stealth, scale=0.75]

  \draw[gray, thin, ->] (-8,0) -- (8,0);
  \draw[gray, thin, ->] (0,-2.5) -- (0,2.5);
\draw[red, thick, decoration={markings, mark=at position 0.92 with {\arrow[red,thick]{Stealth}}},
  postaction={decorate}
  ]
  (-8,2.5) to[out=-45, in=180, looseness=0] (-5.5,0) -- (5.5,0) to[out=0, in=135, looseness=0] (8,-2.5);

  \draw[red, thick, decoration={markings, mark=at position 0.89 with {\arrow[red,thick]{Stealth}}},
  postaction={decorate}
  ]
  (-6,2.5) to[out=-45, in=180, looseness=0] (-4.6,1.2) -- (-2.7,1.2) to[out=0, in=-45, looseness=0] (-4,2.5);
  \draw[red, thick, decoration={markings, mark=at position 0.9 with {\arrow[red,thick]{Stealth}}},
  postaction={decorate}
  ]
  (-6.5,2.5) to[out=-45, in=180, looseness=0] (-4.9,0.9) -- (1,0.9) to[out=0, in=-45, looseness=0] (-0.6,2.5);

  \draw[
  red, thick, decoration={markings, mark=at position 0.92 with {\arrow[red,thick]{Stealth}}},
  postaction={decorate}
  ]
  (-7,2.5) to[out=-45, in=180, looseness=0] (-5.1,0.6) -- (1,0.6) to[out=0, looseness=0] (3,1.1); 

\draw[red, thick, decoration={markings, mark=at position 0.94 with {\arrow[red,thick]{Stealth}}},
  postaction={decorate}
  ]
  (-7.5,2.5) to[out=-45, in=180, looseness=0] (-5.3,0.3) -- (7,0.3) to[out=0, in=-45, looseness=0] (4.8,2.5);

  \fill (-4.5,1.8) circle (2.5pt);
  \node[right] at (-4.5,1.8) {$p_1$};

  \fill (-4.5,-1.8) circle (2.5pt);
  \node[right] at (-4.5,-1.8) {$p_1^*$};

  \fill (-2,1.5) circle (2.5pt);
  \node[right] at (-2,1.5) {$p_2$};

  \fill (-2,-1.5) circle (2.5pt);
  \node[right] at (-2,-1.5) {$p_2^*$};

  \fill (3,1.1) circle (2.5pt);
  \node[right] at (3,1.1) {$p_3$};

  \fill (3,-1.1) circle (2.5pt);
  \node[right] at (3,-1.1) {$p_3^*$};

  \fill (4,1.7) circle (2.5pt);
  \node[right] at (4,1.7) {$p_4$};

  \fill (4,-1.7) circle (2.5pt);
  \node[right] at (4,-1.7) {$p_4^*$};

  \draw[gray] node at (7,2) {$\boxed{z}$};

  \node[above,right] at (6.8,-1.2) {$\color{red}C_0$};

  \node[above,right] at (-3.4,2) {$\color{red}C_1$};

  \node[below,left] at (1.2,1.8) {$\color{red}C_2$};

  \node[above,right] at (1.5,1.3) {$\color{red}C_3$};

  \node[above,right] at (5.7,1.8) {$\color{red}C_4$};

\end{tikzpicture}
\end{center}
\caption{A schematic diagram of the contours for $N=4$ with $n=3$ being an exceptional index.}
\label{fig:contours}
\end{figure}
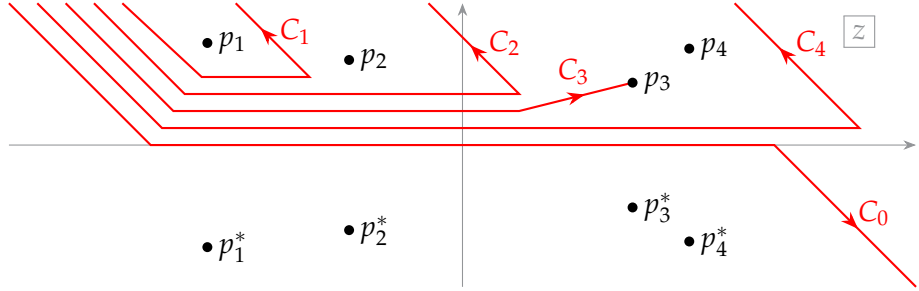

Then we define an exponential factor $\ee^{-\ii h(z)/\varepsilon}$ on each contour $C_0,\dots,C_N$ by 
\begin{equation}
    \ee^{-\ii h(z)/\varepsilon}=\ee^{-\ii h(z;t,x)/\varepsilon}:=\ee^{-\ii(z-x)^2/(4\varepsilon t)}\prod_{n=1}^N(z-p_n)^{-\ii c_n/\varepsilon}(z-p_n^*)^{-\ii c_n^*/\varepsilon}
    \label{eq:PN-exph}
\end{equation}
where a particular branch of the product is chosen for large $|z|$ where all of the contours originate with $\arg(z)=3\pi/4$, and the value along any given contour is determined by analytic continuation.\footnote{A standard way to more concretely define $\ee^{-\ii h(z)/\varepsilon}$ is to first introduce a system of branch cuts in $\C$ emanating from the poles and then to fix the value of $\ee^{-\ii h(z)/\varepsilon}$ at a single point in the complement of the cuts.  The system of contours $C_0,\dots,C_N$ is then specified consistently with the above rules in the cut complex plane.  This latter approach is the one taken in \cite{BlackstoneGGM24a,BlackstoneGGM26,BlackstoneMM26}.}  Suppose that  $\varepsilon t>0$.  Then two square matrices of dimension $N+1$ are  defined with elements being contour integrals involving $\ee^{-\ii h(z)/\varepsilon}$ as follows:
\begin{equation}\label{eq:Amatrix}
    \mathbf{A}(t,x):=\begin{bmatrix}\displaystyle
    \int_{C_0}\ee^{-\ii h(z)/\varepsilon}u_0(z)\,\dd z & \displaystyle\int_{C_0}\frac{\ee^{-\ii h(z)/\varepsilon}\,\dd z}{z-p_1} & \cdots&\displaystyle\int_{C_0}\frac{\ee^{-\ii h(z)/\varepsilon}\,\dd z}{z-p_N}\\
    \vdots & \vdots & \vdots & \vdots\\
     \displaystyle \int_{C_N}\ee^{-\ii h(z)/\varepsilon}u_0(z)\,\dd z & \displaystyle\int_{C_N}\frac{\ee^{-\ii h(z)/\varepsilon}\,\dd z}{z-p_1} & \cdots&\displaystyle\int_{C_N}\frac{\ee^{-\ii h(z)/\varepsilon}\,\dd z}{z-p_N}
     \end{bmatrix}
\end{equation}
\begin{equation}\label{eq:Bmatrix}
    \mathbf{B}(t,x):=\begin{bmatrix}\displaystyle
    \int_{C_0}\ee^{-\ii h(z)/\varepsilon}\,\dd z & \displaystyle\int_{C_0}\frac{\ee^{-\ii h(z)/\varepsilon}\,\dd z}{z-p_1} & \cdots&\displaystyle\int_{C_0}\frac{\ee^{-\ii h(z)/\varepsilon}\,\dd z}{z-p_N}\\
    \vdots & \vdots & \vdots & \vdots\\
    \displaystyle \int_{C_N}\ee^{-\ii h(z)/\varepsilon}\,\dd z & \displaystyle\int_{C_N}\frac{\ee^{-\ii h(z)/\varepsilon}\,\dd z}{z-p_1} & \cdots&\displaystyle\int_{C_N}\frac{\ee^{-\ii h(z)/\varepsilon}\,\dd z}{z-p_N}
     \end{bmatrix}
\end{equation}
For $k=2,\dots,N+1$, column $k$ is the same in $\mathbf{A}(t,x)$ and $\mathbf{B}(t,x)$ and involves integrands $\ee^{-\ii h(z)/\varepsilon}/(z-p_{k-1})$.  All of the integrals in row $j$ of either matrix are taken over the same contour $C_{j-1}$.  We also observe that as a consequence of the identity
\begin{equation}
    u_0(z)\prod_{n=1}^N(z-p_n)^{-\ii c_n/\varepsilon}(z-p_n^*)^{-\ii c_n^*/\varepsilon} = \ii\varepsilon\frac{\dd}{\dd z}\prod_{n=1}^N(z-p_n)^{-\ii c_n/\varepsilon}(z-p_n^*)^{-\ii c_n^*/\varepsilon},
\end{equation}
integration by parts shows that the first column elements of $\mathbf{A}(t,x)$ can be equivalently written in the form
\begin{equation}
A_{j1}(t,x)=\int_{C_{j-1}}\ee^{-\ii h(z)/\varepsilon}u_0(z)\,\dd z = -\frac{1}{2t}\int_{C_{j-1}}\ee^{-\ii h(z)/\varepsilon}(z-x)\,\dd z,\quad j=1,\dots,N+1.
\label{eq:Aj1-alternate}
\end{equation}

The following result gives a formula for the solution of the Benjamin-Ono initial-value problem with rational data that is indeed truly explicit.
\begin{theorem}[Explicit formula for rational initial data, see~\cite{BlackstoneGGM24a}]\label{thm:inversion-formula-rational}
Fix $\varepsilon>0$ and let $u_0$ be of the form \eqref{eq:PN-rational-IC}.  Then for every $x\in\mathbb{R}$ and $t>0$, the solution $u$ of the Cauchy problem \eqref{eq:BO}--\eqref{eq:IC} satisfies $u(t,x)=2\mathrm{Re}(\Pi u(t,x))$ with
\begin{equation}
    \Pi u(t,x)=\frac{\det(\mathbf{A}(t,x))}{\det(\mathbf{B}(t,x))}.
    \label{eq:Pi-u-fraction}
\end{equation}
An equivalent formula is
\begin{equation}
\Pi u(t,x)=\ii\varepsilon\frac{\partial}{\partial x}\log\left(\det(\mathbf{B}(t,x))\prod_{n=1}^N\ee^{\ii (p_n-x)^2/(4\varepsilon t)}\right).
    \label{eq:Pi-u-derivative}
\end{equation}
To solve the Cauchy problem for $t<0$, we may use the symmetry $u(-t,x)=u(t,-x)$, for which we have the initial data $u_0(-x)$ that is also of the form \eqref{eq:PN-rational-IC} with pole and residue data explicitly related to that of $u_0(x)$.
\end{theorem}
The idea behind the proof of this result in \cite{BlackstoneGGM24a} is to observe that for rational $u_0$ of the form \eqref{eq:PN-rational-IC}, the action of the Toeplitz term $\mathcal{T}_{u_0}$ in the Lax operator $\mathcal{L}_{u_0}$ on a function $f\in L^2_+(\R)$ can be expressed as a specific rational function involving $N$ constants, namely the values of the analytic extension of $f$ to $z\in\C_+$ at the poles $z=p_1,\dots,p_N$.  Moving this term to the right-hand side of the equation $\mathcal{X}^*f-2t\mathcal{L}_{u_0}f-zf=\Pi u_0$ along with the constant $\mathfrak{a}[f]=-I_+[f]/(2\pi\ii)$ entering from the action of $\mathcal{X}^*$ according to \eqref{eq:Xstar-action-x}, one sees that $f$ is necessarily the solution of a first-order linear differential equation with a forcing term involving $N+1$ unknown constants $\mathfrak{a}[f],f(p_1),\dots,f(p_N)$.  An integrating factor for this equation is precisely the exponential $\ee^{-\ii h(\diamond)/\varepsilon}$ given by \eqref{eq:PN-exph}, so $f$ can be found explicitly by an integration, enforcing only decay in the negative direction.  It remains to enforce that the resulting formula for $f$ exhibits decay also in the positive direction and has removable singularities at the points $p_1,\dots,p_N$.  This amounts to $N+1$ linear equations on $N+1$ unknowns.  Solving for the unknown $-\mathfrak{a}[f]=\Pi u(t,z)$ by Cramer's rule and letting $z\in\C_+$ approach the real line at a point $x\in\R$ yields the claimed formula \eqref{eq:Pi-u-fraction}.

\subsubsection{Examples of Theorem~\ref{thm:inversion-formula-rational}}
\label{sec:examples}
\paragraph{The simplest case}
Perhaps the simplest rational initial condition $u_0$ of the form \eqref{eq:PN-rational-IC} is
\begin{equation}
    u_0(x)=\frac{2ax-2b}{x^2+1},\quad a,b\in\R.
\label{eq:one-pole-rational}
\end{equation}
There is only one pole in $\C_+$, namely $p_1=\ii$, with residue coefficient $c_1=a+\ii b$.  We may take $C_0$ to be the straight-line contour parametrized by $z=\ee^{-\ii\pi/4}w$, $-\infty<w<+\infty$, and except in the exceptional cases that $c_1=\ii\varepsilon,2\ii\varepsilon,3\ii\varepsilon,\dots$, $C_1$ is just a loop beginning and ending at infinity in the direction $\arg(z)=3\pi/4$ that encircles $z=\ii$ but not $z=-\ii$, for instance the half-line parametrized by $z=-2+\ee^{-\ii\pi/4}w$ for $-\infty<w<0$ followed by the segment $z=w$ for $-2<w<2$ followed by the half-line $z=2+\ee^{3\pi\ii/4}w$ with $0<w<+\infty$.  In the exceptional cases, we  parametrize $C_1$ instead by the real half-line:  $z=\ii +\ee^{-\ii\pi/4}w$, $-\infty<w<0$.  See Figure~\ref{fig:OnePoleContours}.
\begin{figure}[h]
\begin{center}
    \begin{tikzpicture}[>=Stealth, scale=0.75]

  \draw[gray, thin, ->] (-2.5,0) -- (2.5,0);
  \draw[gray, thin, ->] (0,-2.5) -- (0,2.5);
  \draw[
  red, thick, decoration={markings, mark=at position 0.75 with {\arrow[red,thick]{Stealth}}},
  postaction={decorate}
  ]
  (-2.121,2.121) -- (2.121,-2.121);

  \fill (0,0.6) circle (2.5pt);
  \node[right] at (0,0.6) {$p_1=\ii$};

  \fill (0,-0.6) circle (2.5pt);
  \node[left] at (0,-0.6) {$p_1^*=-\ii$};

\node[above, left] at (-2.121+1.5,2.121+0.4) {$\infty\ee^{3\pi\ii/4}$};

\node[below, right] at (2.121-0.3,-2.121-0.1) {$\infty\ee^{-\ii\pi/4}$};

  \draw[gray] node at (2,2) {$\boxed{z}$};

  \node[above,right] at (1,-1) {$\color{red}C_0$};

 \def\shift{6.5}
  \draw[gray, thin, ->] ({-2.5+\shift},0) -- ({2.5+\shift},0);
  \draw[gray, thin, ->] ({0+\shift},-2.5) -- ({0+\shift},2.5);
  \draw[
  red, thick, decoration={markings, mark=at position 0.75 with {\arrow[red,thick]{Stealth}}},
  postaction={decorate}
  ]
  ({-2.121+\shift},{2.121-1.2}) -- ({-1.2+\shift},0) -- ({1.2+\shift},0) -- ({-2.121+1.2+\shift},{2.121});   \fill ({0+\shift},0.6) circle (2.5pt);
  \node[left] at ({0+\shift},0.6) {$p_1=\ii$};

  \fill ({0+\shift},-0.6) circle (2.5pt);
  \node[left] at ({0+\shift},-0.6) {$p_1^*=-\ii$};

\node[above, left] at (-2.121+1.5+\shift,2.121+0.4) {$\infty\ee^{3\pi\ii/4}$};

  \draw[gray] node at ({2+\shift},2) {$\boxed{z}$};

  \node[above, right] at ({\shift+0.4},0.9) {$\color{red}C_1$};

  \node[below, align=center] at ({\shift,-2.4}) {(non-exceptional case)};

\def\shift{13}

  \draw[gray, thin, ->] ({-2.5+\shift},0) -- ({2.5+\shift},0);
  \draw[gray, thin, ->] ({0+\shift},-2.5) -- ({0+\shift},2.5);

  \draw[
  red, thick, decoration={markings, mark=at position 0.7 with {\arrow[red,thick]{Stealth}}},
  postaction={decorate}
  ]
  ({-2.121+0.6+\shift},2.121) -- ({0+\shift},0.6);

  \fill ({0+\shift},0.6) circle (2.5pt);
  \node[right] at ({0+\shift},0.6) {$p_1=\ii$};

  \fill ({0+\shift},-0.6) circle (2.5pt);
  \node[left] at ({0+\shift},-0.6) {$p_1^*=-\ii$};

\node[above, left] at (-2.121+1.5+\shift,2.121+0.4) {$\infty\ee^{3\pi\ii/4}$};

  \draw[gray] node at ({2+\shift},2) {$\boxed{z}$};

  \node[above, right] at ({\shift-0.8},1.5) {$\color{red}C_1$};

  \node[below, align=center] at ({\shift},-2.4) {(exceptional cases)};

\end{tikzpicture}    
\end{center}
\caption{The integration contours $C_0$ (left) and $C_1$ (center and right) for $u_0(x)=(2ax-2b)/(x^2+1)$.}
\label{fig:OnePoleContours}
\end{figure}
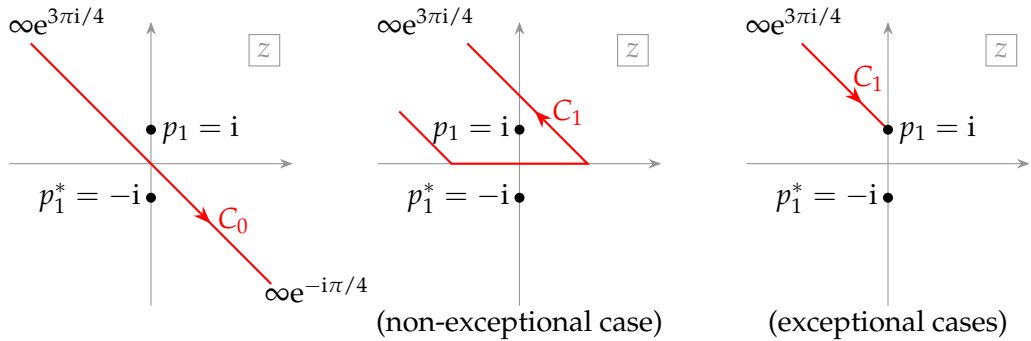
For these integration contours, a reasonable way to define the generally multivalued factors in $\ee^{-\ii h(z)/\varepsilon}$ is to choose the principal branch for $(z-p_1^*)^{-\ii c_1^*/\varepsilon} = (z+\ii)^{-\ii a/\varepsilon -b/\varepsilon}$, while we can take a straight-line branch cut from $z=\ii$ in the direction $\arg(z-\ii)=3\pi/4$ and define $(z-p_1)^{-\ii c_1/\varepsilon}=k(\ee^{\ii\pi/4}(z-\ii))^{-\ii a/\varepsilon + b/\varepsilon}$, for a suitable nonzero value of $k$ that actually plays no role since it cancels from the fraction of determinants in \eqref{eq:Pi-u-fraction}.  Note that there are no branch cuts at all for the exceptional cases in particular. Then, taking $k=1$, the elements of $\mathbf{A}(t,x)$ and $\mathbf{B}(t,x)$ consist of just six contour integrals:  for $j=1,2$,
\begin{equation}
\begin{gathered}
    A_{j1}(t,x):=2\ee^{\ii\pi/4}\int_{C_{j-1}}\ee^{-\ii (z-x)^2/(4\varepsilon t)}(az-b)(\ee^{\ii\pi/4}(z-\ii))^{-\ii a/\varepsilon + b/\varepsilon-1}(z+\ii)^{-\ii a/\varepsilon -b/\varepsilon-1}\,\dd z,\\
    B_{j1}(t,x):=\int_{C_{j-1}}\ee^{-\ii (z-x)^2/(4\varepsilon t)}(\ee^{\ii\pi/4}(z-\ii))^{-\ii a/\varepsilon + b/\varepsilon}(z+\ii)^{-\ii a/\varepsilon -b/\varepsilon}\,\dd z,\\
    A_{j2}(t,x)=B_{j2}(t,x):=\ee^{\ii\pi/4}\int_{C_{j-1}}\ee^{-\ii (z-x)^2/(4\varepsilon t)}(\ee^{\ii\pi/4}(z-\ii))^{-\ii a/\varepsilon + b/\varepsilon-1}(z+\ii)^{-\ii a/\varepsilon -b/\varepsilon}\,\dd z.
\end{gathered}
\label{eq:N=1-ContourIntegrals}
\end{equation}

\paragraph{Multisoliton solutions}
Another class of examples comes from allowing for arbitrary poles $p_1,\dots,p_N\in \C_+$ but assuming that all residue coefficients are of the form $c_n=-\ii\varepsilon m_n$ for positive integers $m_n>0$. Taking the simplest case of $m_n=1$ or $c_n=-\ii\varepsilon$ for all $n=1,\dots,N$, from \eqref{eq:PN-rational-IC} we have $u_0(x)=\Pi u_0(x) + \Pi u_0(x)^*$, where
\begin{equation}   \Pi u_0(x)=\ii \varepsilon\sum_{n=1}^N\frac{1}{x-p_n^*}=
\ii\varepsilon\frac{\partial}{\partial x}\log\left(\prod_{n=1}^N(x-p_n^*)\right).
    \label{eq:rational-soliton}
\end{equation}
It is known (see, for instance \cite{CLP1979,Sun2020}) that the $N$-soliton solutions have this form in which the poles $\{p_n\}_{n=1}^N$ in $\C_+$ evolve as functions of time solving the Calogero-Moser system of ordinary differential equations, which is also an integrable Hamiltonian system.  The quantities $\{p_n(t)\}_{n=1}^N$ can be expressed, in terms of the eigenvalues $\lambda_1<\cdots<\lambda_N<0$ of the Lax operator $\mathcal{L}_{u_0}$ and the complex phase constants $\{\gamma_n\}_{n=1}^N$ of the eigenfunctions, as the eigenvalues of the $N\times N$ matrix $\mathbf{P}(t)$ with elements 
\begin{equation}
    P_{jk}(t):=\begin{cases}
        -\ii\varepsilon (\lambda_j-\lambda_k)^{-1},&j\neq k\\
        -2\lambda_jt-\mathrm{Re}(\gamma_j)-\ii\varepsilon(2\lambda_j)^{-1},&j=k.
    \end{cases}
\end{equation}
Therefore, viewing the product in \eqref{eq:rational-soliton} as a determinant and denoting by $\mathbb{I}$ the $N\times N$ identity matrix,
\begin{equation}
    \Pi u(t,x)=\ii\varepsilon\frac{\partial}{\partial x}\log\left(\det(x\mathbb{I}-\mathbf{P}(t)^\dagger)\right).
\end{equation}
By elementary operations \cite[Section 1.2]{BlackstoneGM25}, this formula can be equivalently written as Matsuno's $N$-soliton formula \cite{Matsuno84}:
\begin{equation}
    \Pi u(t,x)=\ii\varepsilon\frac{\partial}{\partial x}\log\left(\det\left(\mathbb{I}-\frac{\ii}{\varepsilon}\mathbf{H}(t,x)\right)\right),
\label{eq:Matsuno}
\end{equation}
where $\mathbf{H}(t,x)$ is the Hermitian matrix with elements
\begin{equation}
    H_{jk}(t,x):=\begin{cases}\displaystyle -2\lambda_j(x+2\lambda_jt+\mathrm{Re}(\gamma_j)),&j=k,\\
    \displaystyle \frac{2\ii\varepsilon\sqrt{\lambda_j\lambda_k}}{\lambda_j-\lambda_k},&j\neq k.
    \end{cases}
    \label{eq:Matsuno-Matrix}
\end{equation}
So, by uniqueness, the solution $u(t,x)$ of \eqref{eq:BO} with initial data \eqref{eq:rational-soliton} is a multi-soliton solution.  It is possible to obtain Matsuno's representation \eqref{eq:Matsuno}--\eqref{eq:Matsuno-Matrix} from Theorem~\ref{thm:inversion-formula-rational} \emph{by direct calculations}.  For the reader's convenience, we give these calculations in Appendix~\ref{app:Nsoliton}.

Note that the discrete data $\{(\lambda_j,\mathrm{Re}(\gamma_j))\}_j$ has been characterized for $N$-soliton initial conditions \eqref{eq:rational-soliton}, and indeed more generally for any rational initial data of the form \eqref{eq:PN-rational-IC}, in \cite{MillerW16a}.  In particular, the eigenvalues are the zeros of an explicit Evans function that is a determinant of an $N\times N$ matrix of contour integrals depending on $\lambda$ (see Proposition~\ref{PN-prop-MW16} in Appendix~\ref{app:soliton-limit} below).  When the rational data has the form \eqref{eq:rational-soliton} there are exactly $N$ simple zeros of the Evans function, however more generally the number of eigenvalues bears no relation to the number $N$ of poles in $\C_+$.

\subsubsection{Remarks on Theorem~\ref{thm:inversion-formula-rational}}
\begin{remark}[Asymptotics as \texorpdfstring{$t\to0$}{t=0}]
\label{rem:continuity-at-t=0+}
The exponential factor $\ee^{-\ii h(z)/\varepsilon}$ appearing in the integrals defining the elements of the matrices $\mathbf{A}(t,x)$ (see \eqref{eq:Amatrix}) and $\mathbf{B}(t,x)$ (see \eqref{eq:Bmatrix}) is undefined at $t=0$, according to \eqref{eq:PN-exph}.  Of course by arguments independent of the rational assumption we have already shown how the explicit formula \eqref{eq:explicitformula} returns $\Pi u(0,z)=\Pi u_0(z)$ for all $z\in \C_+$ (and hence also in the limit $z\to x$ for almost every $x\in\R$) when $t=0$, but it is useful to see how the determinantal formula \eqref{eq:Pi-u-fraction} valid for rational data $u_0$ behaves in the limit $t\to 0$ to verify directly the continuity of $t\mapsto u(t,x)$ for each fixed $x\in\R$ at $t=0$.  

We assume that $t>0$ and compute the limit of the right-hand side of \eqref{eq:Pi-u-fraction} as $t\downarrow 0$ for fixed $x\in\R$ and $\varepsilon>0$.  
For this limit, it is best to isolate the factor $\ee^{-\ii (z-x)^2/(4\varepsilon t)}$ in $\ee^{-\ii h(z)/\varepsilon}$, since it is responsible for the concentration in the integrand as $t\downarrow 0$.  
The only critical point of the exponent function $-\ii (z-x)^2/(4\varepsilon t)$ is $z=x\in\mathbb{R}$, and it is simple.  To analyze the ratio of determinants of the matrices $\mathbf{A}(t,x)$ and $\mathbf{B}(t,x)$ in the limit $t\downarrow 0$, it is first useful to apply some row and column operations to the matrices that leave the determinants unchanged.  We first apply row operations on rows $j=2,\dots,N+1$ to replace the integration contour $C_{j-1}$ whenever $j-1$ is a non-exceptional index with a contour $\widetilde{C}_{j-1}$ that encircles only one pole $p_{j-1}$ in the counterclockwise sense and tends to $z=\infty\ee^{3\pi\ii/4}$ in both directions.  If $j-1$ is exceptional, then we just take $\widetilde{C}_{j-1}=C_{j-1}$.  Then, in the matrix $\mathbf{A}(t,x)$, we use column operations to replace the function $u_0(z)$ in the integrand of the first column with $\Pi u_0(z)$ (since $u_0(z)-\Pi u_0(z)$ is a linear combination of $(z-p_n)^{-1}$ for $n=1,\dots,N$).  Thus we obtain two modified matrices $\widetilde{\mathbf{A}}(t,x)$ and $\widetilde{\mathbf{B}}(t,x)$ with the same determinants as $\mathbf{A}(t,x)$ and $\mathbf{B}(t,x)$ respectively.

Then, using the method of steepest descent, it is not difficult to see that each of the $N+2$ distinct integrals appearing in any given row of both matrices $\widetilde{\mathbf{A}}(t,x)$ and $\widetilde{\mathbf{B}}(t,x)$, all of which correspond to the same contour of integration, will be dominated in the limit $t\downarrow 0$ by contributions from the neighborhood of exactly one point.  For the first row, corresponding to the contour $C_0$, this is the simple critical point $z=x$.  For row $j$, $j=2,\dots,N+1$, corresponding to the contour $\widetilde{C}_{j-1}$, this is the point $z=p_{j-1}$.  Although the details of the form of the leading term are different depending on whether the dominating point is a critical point or a pole $z=p_{j-1}$, and in the latter case whether the integrand is integrable at the pole or not, these details produce nonzero factors that are common to every element of the same row of both $\widetilde{\mathbf{A}}(t,x)$ and $\widetilde{\mathbf{B}}(t,x)$.  Hence they will cancel from the ratio of determinants in \eqref{eq:Pi-u-fraction} for $\Pi u(t,x)$.  The relative sizes of the renormalized integrals in any given row are then determined by the order of the integrand at the dominating point.  Thus, for some nonzero common factor $\alpha_1(t)\neq 0$, the first rows of $\widetilde{\mathbf{A}}(t,x)$ and of $\widetilde{\mathbf{B}}(t,x)$ can be written as
\begin{equation}
\begin{split}
\left(\widetilde{A}_{11},\dots,\widetilde{A}_{1,N+1}\right) &= \alpha_1(t)\left[\left(\Pi u_0(x),\frac{1}{x-p_1},\dots,\frac{1}{x-p_N}\right)+o(1)\right]\\
\left(\widetilde{B}_{11},\dots,\widetilde{B}_{1,N+1}\right)&=\alpha_1(t)\left[\left(1,\frac{1}{x-p_1},\dots,\frac{1}{x-p_N}\right)+o(1)\right]
\end{split}
\end{equation}
in the limit $t\downarrow 0$.
For row $j$, $j=2,\dots,N+1$, because $u_0(z)$ has been replaced by $\Pi u_0(z)$ which is analytic in the upper half-plane, the most singular element appears in column $j$ and therefore there exists some nonzero factor $\alpha_j(t)\neq 0$ such that
\begin{equation}
    \begin{split}
        \left(\widetilde{A}_{j1},\dots,\widetilde{A}_{j,N+1}\right)&=\alpha_j(t)\left[\left(0,\dots,0,1,0,\dots,0\right)+o(1)\right]\\
        \left(\widetilde{B}_{j1},\dots,\widetilde{B}_{j,N+1}\right)&=\alpha_j(t)\left[\left(0,\dots,0,1,0,\dots,0\right)+o(1)\right]
    \end{split}
\end{equation}
as $t\downarrow 0$, where the $1$ appears in column $j$ on the right-hand side in both cases.  Note that the $o(1)$ terms in the first column here are different expressions for $\widetilde{\mathbf{A}}(t,x)$ and for $\widetilde{\mathbf{B}}(t,x)$, but they are both relatively small compared to the identical diagonal elements.
Therefore, defining a $N$-component row vector by $\mathbf{v}^\top:=((x-p_1)^{-1},\dots,(x-p_N)^{-1})$, and
letting $\mathbf{0}$ denote the $N$-component column vector of zeros, in the limit $t\downarrow 0$,
\begin{equation}
u(t,x)= 2\mathrm{Re}\left(\frac{\det(\widetilde{\mathbf{A}}(t,x))}{\det(\widetilde{\mathbf{B}}(t,x))}\right) = 2\mathrm{Re}\left(\frac{\displaystyle \begin{vmatrix} \Pi u_0(x) & \mathbf{v}^\top\\
\mathbf{0} & \mathbb{I}\end{vmatrix}+o(1)}{\displaystyle\begin{vmatrix} 1 & \mathbf{v}^\top\\\mathbf{0} & \mathbb{I}\end{vmatrix}+o(1)}\right) = 2\mathrm{Re}\left(\Pi u_0(x)\right) + o(1).
\end{equation}
This verifies directly that the solution formula of Theorem~\ref{thm:inversion-formula-rational} is continuous from the right at $t=0$, for each $x\in\mathbb{R}$ and $\varepsilon>0$.
\end{remark}

\begin{remark}[The solution \texorpdfstring{$u(t,x)$}{u(t,x)=0} in the complex \texorpdfstring{$x$}{x}-plane for rational \texorpdfstring{$u_0$}{u_0}]
When $\epsilon>0$ and $t>0$, all of the contour integrals in the elements of $\mathbf{A}(t,x)$ and $\mathbf{B}(t,x)$ are absolutely convergent for all $x\in\C$ and hence $\det(\mathbf{A}(t,x))$ and $\det(\mathbf{B}(t,x))$ are both entire functions of $x\in\C$. 
In particular, $x\mapsto \Pi u(t,x)$ is meromorphic for every $t>0$, and according to \eqref{eq:Pi-u-derivative} every pole (necessarily in the lower half-plane $\C_-$) is simple with residue that is a positive integer multiple of $\ii\varepsilon$.  Since for $x\in\R$ we have $u(t,x)=\Pi u(t,x)+\Pi u(t,x)^*=\Pi u(t,x)+\Pi u(t,x^*)^*$, it follows by Schwarz reflection that $x\mapsto u(t,x)$ is a Schwarz-symmetric meromorphic function for all $t>0$ with simple poles in complex-conjugate pairs, all having residues that are positive integer multiples of $\ii\varepsilon$, and having no finite accumulation point.  Naturally the same is true for $t=0$ except that the meromorphic function degenerates to the rational one $u_0$.  However for $t>0$, $x\mapsto u(t,x)$ is generally no longer rational, unless $\ii c_n/\varepsilon$ is a positive integer $m_n>0$ for every $n=1,\dots,N$ (in which case $u(t,x)$ is an exact $M$-soliton solution of \eqref{eq:BO} where $M=m_1+\cdots + m_N$ as discussed in Section~\ref{sec:examples}).  
In the general case that $p_n\in\C_+$ is a pole of $u_0$ with a non-quantized residue $c_n\not\in -\ii\varepsilon\mathbb{Z}_{>0}$, the behavior of $u(t,x)$ for $x\approx p_n$ is necessarily very complicated in the limit $t\downarrow 0$.  Note that the analysis carried out in Remark~\ref{rem:continuity-at-t=0+} is not valid for $x$ in the complex plane near a pole in $\C_-$ of $x\mapsto\Pi u(t,x)$.
\end{remark}

\section{Applications to small-dispersion asymptotic behavior for rational initial data}
\label{sec:small-dispersion}
Now we consider the implications of Theorem~\ref{thm:explicitformula} and Theorem~\ref{thm:inversion-formula-rational} for the initial-value problem \eqref{eq:BO}--\eqref{eq:IC} in the small-dispersion limit $\varepsilon\downarrow 0$ with fixed initial data $u_0$ independent of $\varepsilon$.  
\subsection{Weak small-dispersion limit}
First we consider convergence in the weak-$L^2$ topology.  Recall that $f_n$ converges weakly to $f$ in $L^2(\R)$ if for each $g\in L^2(\R)$, $\langle f_n,g\rangle \to \langle f,g\rangle$, which can be viewed as the convergence of local averages of $f_n$ to those of $f$.  The first result is the following.
\begin{theorem}[Weak small-dispersion limit, \cite{Gerard23}]
    Fix any $t\in\mathbb{R}$, and suppose that $u_0\in L^\infty(\R)\cap L^2(\R)$.  Let $u(t,x;\varepsilon)$ denote the solution of the Cauchy problem \eqref{eq:BO}--\eqref{eq:IC}.  Then 
    \begin{equation}
        \lim_{\varepsilon\to 0}u(t,\diamond;\varepsilon) = \bar{u}(t,\diamond)
    \end{equation}
    holds in the weak $L^2(\R)$ topology, where $\bar{u}(t,x)=2\mathrm{Re}(\Pi\bar{u}(t,x))$ and $x\mapsto \Pi\bar{u}(t,x)$ is the boundary value on $\R$ of the analytic function $\C_+\ni z\mapsto\Pi\bar{u}(t,z)$ given by
    \begin{equation}
        \Pi\bar{u}(t,z)=\frac{1}{2\pi\ii}I_+\left[(\mathcal{X}^*+2t\mathcal{T}_{u_0}-z\mathrm{Id})^{-1}\Pi u_0\right],\quad z\in\C_+.
        \label{eq:weak-limit}
    \end{equation}
    If moreover $u_0\in C^1(\R)$ with both $u_0(x)$ and $u_0'(x)$ tending to zero as $x\to\pm\infty$, then $\R\ni y\mapsto y+2tu_0(y)$ has only finitely-many critical values, and for any non-critical value $x\in\R$ the equation $y+2tu_0(y)=x$ has an odd number $2\ell(t,x)+1$ of simple real solutions $y_0(t,x)>y_1(t,x)>\cdots>y_{2\ell(t,x)}(t,x)$.  In this case, $\bar{u}(t,x)$ has another representation: 
    \begin{equation}
        \bar{u}(t,x)=\sum_{k=0}^{2\ell(t,x)}(-1)^k u_0(y_k(t,x)).
        \label{eq:alternating-sum}
    \end{equation}
    \label{thm:weak-small-dispersion}
\end{theorem}

A similar result was recently obtained in \cite{GerardH26} for the Benjamin-Ono hierarchy. 
Note that since the Lax operator $\mathcal{L}_{u_0}$ is defined by \eqref{eq:Lax-operator}, the weak limit $\bar{u}(t,x)$ is given by exactly the same explicit formula as is $u(t,x)$ itself, simply setting $\varepsilon=0$.  The specialized form \eqref{eq:alternating-sum} for $\bar{u}(t,x)$ is precisely the alternating sum of branches $u^\mathrm{B}_k(t,x):=u_0(y_k(t,x))$ of the generally multivalued solution of the inviscid Burgers equation $\partial_t u^\mathrm{B}+2u^\mathrm{B}\partial_xu^\mathrm{B}=0$ with initial data $u^\mathrm{B}(0,x)=u_0(x)$ obtained by the method of characteristics.  Indeed, $y+2tu_0(y)=x$ is just the equation satisfied by intercepts $y$ of all characteristic lines through the point $(t,x)\in\R^2$ along each of which the value of $u^\mathrm{B}$ is constant in time.  The formula \eqref{eq:alternating-sum} had been earlier established in \cite{MillerX10} as the weak $L^2$ limit of a family of multi-soliton solutions of \eqref{eq:BO} (a so-called soliton ensemble) converging to suitable $u_0$ (in particular being smooth and positive with a bell-shaped graph) at time $t=0$ in the strong $L^2$ sense.  The result in Theorem~\ref{thm:weak-small-dispersion} substantially improves that from \cite{MillerX10} by removing unnecessary conditions on $u_0$ (sometimes at the cost of requiring the somewhat less explicit limit formula \eqref{eq:weak-limit}), and by being applicable to the Cauchy problem rather than a soliton-ensemble approximation thereof.  The proof of Theorem~\ref{thm:weak-small-dispersion} is also far shorter and more direct than that in \cite{MillerX10}.  It is based on the explicit formula in Theorem~\ref{thm:explicitformula}; the fact that the Benjamin-Ono equation conserves the $L^2$-norm means that as $\varepsilon\to 0$, the $L^2$-norm of $x\mapsto u(t,x;\varepsilon)$ remains constant, equal to that of $u_0$ (independent of $\varepsilon$).  Therefore the family of functions $\{x\mapsto u(t,x;\varepsilon)\}_{\varepsilon>0}$ is compact in the weak $L^2$ topology so weak limits exist.  It only remains to argue that the weak limit is unique, is given by the formula \eqref{eq:weak-limit}, and that this reduces to \eqref{eq:alternating-sum} for suitable $u_0$. 

The method of proof shows that the weak limit satisfies $\|\bar{u}(t,\diamond)\|_{L^2(\R)}\le \|u_0\|_{L^2(\R)}$, but unlike classical solutions of the limiting inviscid Burgers equation, generally the inequality is strict for $t>0$ sufficiently large.  Indeed, in \cite[Section 3]{Gerard23}, it is shown that if $\ell(t,x)>0$, the function $\bar{u}(t,x)$ defined by \eqref{eq:alternating-sum} is not a solution of the limiting inviscid Burgers equation even though the functions $u_k^\mathrm{B}(t,x)=u_0(y_k(t,x))$ are so.  Another result established in \cite{Gerard23} is the maximum principle $\inf u_0(\diamond)\le \bar{u}(t,x)\le\sup u_0(\diamond)$.  By the method of characteristics, both inequalities are saturated for some $x\in\R$ provided that $t>0$ is sufficiently small that $\ell(t,x)=0$ for all $x\in\R$.  However, examples \cite[Section 5]{Gerard23} show that this is generally no longer the case once the inviscid Burgers solution becomes multi-valued for some $x\in\R$.  These results suggest that the dissipation of $L^2$-norm of the weak limit $\bar{u}(t,\diamond)$ over time is due to features of $u(t,\diamond;\varepsilon)$ such as wild oscillations of wavelength $o(1)$ as $\varepsilon\to 0$ that are not captured by local averages.

\begin{remark}[A simple example]
\label{rem:weak-one-pole}
The initial data $u_0(x)=(2ax-2b)/(x^2+1)$ satisfies the stronger hypotheses of Theorem~\ref{thm:weak-small-dispersion} under which the formula \eqref{eq:alternating-sum} applies.  The characteristic equation $y+2tu_0(y)=x$ can be rearranged as $p_3(y;t,x):=(y-x)(y^2+1)+4t(ay-b)=0$ which is a cubic equation for $y$ and hence there are most $3$ real solutions, i.e., the maximum value of $\ell(t,x)$ is $\ell(t,x)=1$.  When $t=0$ one has a unique real solution $y_0(0,x)=x$ so $\ell(0,x)=0$. If $a<0$, then for large $t>0$, there are three real solutions given approximately by $y_0(t,x)\approx 2\sqrt{-a}$, $y_1(t,x)\approx -b/a$, and $y_2(t,x)\approx -2\sqrt{-a}$ (all independent of $x$ at leading order).  Eliminating $y$ between $p_3(y;t,x)=0$ and $p_3'(y;t,x)=0$ gives the discriminant locus where there exist non-simple roots $y$; adjoining the additional condition $p_3''(y;t,x)=0$ (i.e., $6y-2x=0$) then determines specific points on the locus where all three roots of the cubic coincide, i.e., $p_3(y;t,x)$ is a perfect cube of a linear factor.  These are points of gradient catastrophe for the solution $u^\mathrm{B}(t,x)$ of the inviscid Burgers equation with initial data $u_0$, denoted $(t_\mathrm{c},x_\mathrm{c})$.  In the special case $a=0$, the discriminant locus is defined by the relation $108b^2t^2+4bx(x^2+9)t+(x^2+1)^2=0$, which is a quadratic equation for $t$ with solutions $t=(-x^3-9x\pm(x^2-3)^{3/2})/(54b)$.  These values are real for $|x|>\sqrt{3}$, with two branches bifurcating from each of two gradient catastrophe points $(t_\mathrm{c},x_\mathrm{c})=(\mp 2\sqrt{3}/(9b),\pm\sqrt{3})$.  In between each pair of bifurcating branches, one has $\ell(t,x)=1$ and otherwise $\ell(t,x)=0$.  When $a\neq 0$ there are three catastrophe points.  See Figure~\ref{fig:N=1}.
\end{remark}
\begin{figure}[h]
\begin{center}
\includegraphics[width=0.45\linewidth]{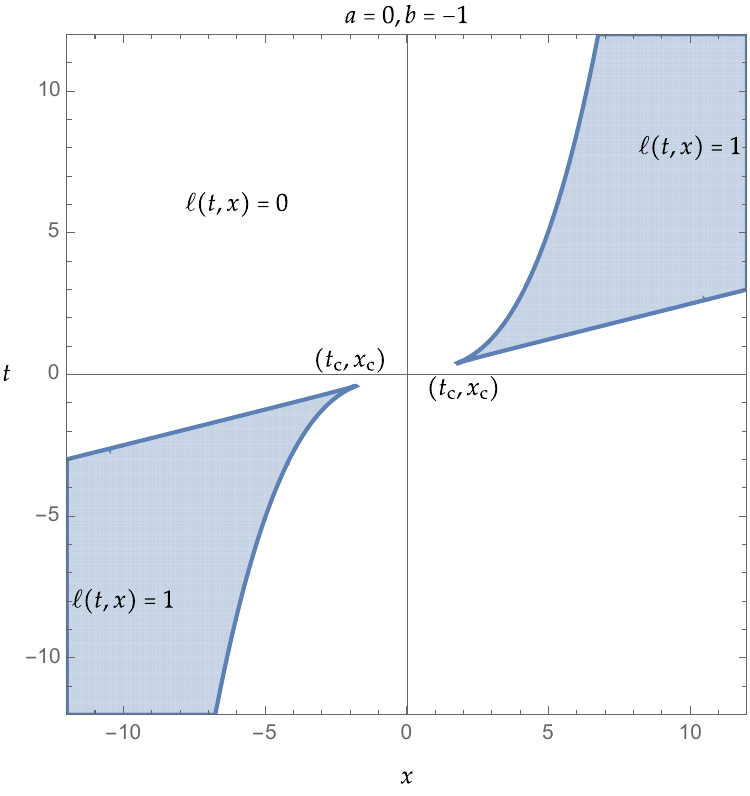}\hfill\includegraphics[width=0.45\linewidth]{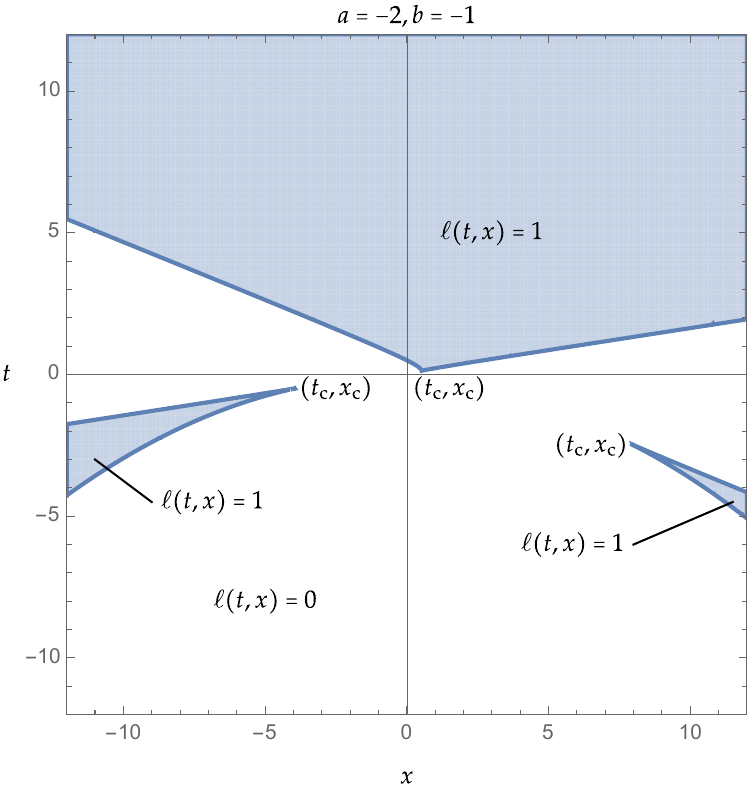}
\end{center}
\caption{Regions of the $(t,x)$-plane and piecewise-constant values of $\ell(t,x)$ for initial data $u_0(x)=2(ax-b)/(x^2+1)$ with $(a,b)=(0,-1)$ (left) and $(a,b)=(-2,-1)$ (right).  The boundary curves constitute the so-called caustic locus $\mathcal{C}$ for the multi-valued inviscid Burgers solution.}
\label{fig:N=1}
\end{figure}

In the periodic case, the corresponding analogue of the explicit formula \eqref{eq:explicitformula} was used by M{\ae}hlen \cite{Maehlen26zero} to improve earlier results~\cite{Gassot21zero, Gassot23zero}, and it was also shown that for sufficiently regular $u_0$ the weak limit $\bar{u}(t,x)$ is also of the form~\eqref{eq:alternating-sum}. See also~\cite{Badreddine24zero} for an identification of the weak small-dispersion limit for the Calogero-Moser derivative NLS equation by means of an explicit formula for its solution. 

It is natural to compare Theorem~\ref{thm:weak-small-dispersion} with its analogue for the Cauchy problem of the Korteweg-de Vries equation:
\begin{equation}
    \partial_t u+\partial_x(u^2)+\varepsilon^2\partial_x^3u=0,\quad u(0,x)=u_0(x).
\label{eq:KdV}
\end{equation}
For this problem, and for a soliton-ensemble sequence of initial data $u_0$,  Lax and Levermore \cite{LL} proved the existence of a small-$\varepsilon$ limit $\bar{u}^\mathrm{KdV}(t,x)$ in the same weak-$L^2$ topology as in Theorem~\ref{thm:weak-small-dispersion}.  The analogue of \eqref{eq:weak-limit} is a formula for $\bar{u}^\mathrm{KdV}$ in terms of the solution of a constrained convex optimization problem for measures parametrized by $(t,x)\in\R^2$.  Under additional regularity assumptions on $u_0$ similar to those in Theorem~\ref{thm:weak-small-dispersion}, the extremal measure has a density with a finite number of intervals on which the constraints are active, with an odd number of $(t,x)$-dependent endpoints that play roles analogous to $u_0(y_0(t,x)),\dots,u_0(y_{2\ell}(t,x))$ in \eqref{eq:alternating-sum}.  However, determining these endpoints is much more difficult than solving for intercepts of characteristic lines, and the domains in the $(t,x)$-plane on which the number of endpoints is constant are not the same as those for which $\ell(t,x)$ is constant.

\subsection{Strong small-dispersion asymptotics}
\label{sec:rational-small-dispersion}
It is natural to expect that the behavior of $u(t,x)$ in the small-dispersion limit $\varepsilon\downarrow 0$ should have something to do with the inviscid Burgers equation, since this is exactly what one obtains from the Benjamin-Ono equation upon setting $\varepsilon=0$.  However, as is well known, the solution of $\partial_t u^\mathrm{B}+2u^\mathrm{B}\partial_xu^\mathrm{B}=0$ for nonzero initial data vanishing at $x=\pm\infty$ is not classical for all time $t\in\mathbb{R}$, with the generic type of singularity being a finite-time gradient catastrophe occurring at a point $(t,x)=(t_\mathrm{c},x_\mathrm{c})$.  Since the neglected term in \eqref{eq:BO} when $\varepsilon=0$ is dispersive rather than diffusive, the solution of Benjamin-Ono for small $\varepsilon>0$ is expected to become highly oscillatory in the neighborhood of the catastrophe point $(t_\mathrm{c},x_\mathrm{c})$; in physical language we then have a dispersive shock wave --- a spreading oscillatory wave packet --- rather than a viscous shock modeled by a weak solution of the inviscid Burgers approximation with a jump discontinuity.

A model for a dispersive shock wave is a slowly modulated wavetrain, i.e., a periodic traveling wave belonging to a family of such solutions with continuous parameters that instead of being constants are functions of $(t,x)$ that vary on scales much longer than than the period/wavelength of the wave.  In 1965, it was discovered by Whitham \cite{Whitham65a,Whitham65b} using an averaging approach applied either to local conservation laws or to a Lagrangian density that for such a structure to be formally consistent with a given nonlinear dispersive equation to leading order in the small ratio of scales, it is necessary that the wave parameters satisfy a first-order system of quasilinear partial differential equations called the Whitham equations or the modulation equations.  

For integrable dispersive equations such as the Korteweg-de Vries equation \eqref{eq:KdV}, it is known that there exist families of exact solutions that are restrictions to a linear flow, via a traveling wave variable $\mathbf{k}x-\mathbf{w}t$ with wavenumber vector $\mathbf{k}\in\R^\ell$ and frequency vector $\mathbf{w}\in\R^\ell$, of a smooth function defined on the $\ell$-torus $\mathbb{T}^\ell$.  Such solutions are called \emph{multiphase waves}, and they are generalizations of periodic traveling waves, for which $\ell=1$.  The $\ell$-torus that appears is the real part of the Jacobian variety of a Riemann surface of genus $\ell$.  Whitham's modulation theory was extended to the multiphase wave solutions of the Korteweg-de Vries equation by Flaschka, Forest, and McLaughlin \cite{FlaschkaFM80}.  This basic approach has since been extended to numerous other systems having Lax pairs involving differential operators.  In particular, the Flaschka-Forest-McLaughlin theory shows that regardless of the value of $\ell$, the modulation equations always admit a complete set of $2\ell+1$ Riemann invariant variables, which provides a clear explanation for what Whitham himself regarded as a miracle for the case $\ell=1$ considered in his original paper \cite{Whitham65a}.

For the Benjamin-Ono equation \eqref{eq:BO}, a family of periodic traveling wave solutions was found in the original 1967 paper of Benjamin \cite{Benjamin67}, and the 1991 paper of Amick-Toland \cite{AmickT91} showed that Benjamin's solutions were the only such waves.  These solutions were generalized in 1979 to the multiphase case of $\ell>1$ by Satsuma-Ishimori \cite{SatsumaI79} and Case \cite{Case79}.  Unlike the Korteweg-de Vries case for which the multiphase waves are written in terms of the highly-transcendental Riemann theta function of the underlying genus-$\ell$ Riemann surface, the multiphase solutions of Benjamin-Ono are rational-trigonometric, i.e., rational functions of elementary oscillatory exponential factors.  In 1991, the multiphase wave solutions of Benjamin-Ono were obtained in a new way by Dobrokhotov-Krichever \cite{DobrokhotovK91}, who also were able to use their construction to develop the Benjamin-Ono analogue of multiphase modulation theory.  As was the case for the Korteweg-de Vries equation, the $\ell$-phase modulation equations found by Dobrokhotov-Krichever constitute a quasilinear system in $2\ell+1$ dependent variables, and this system admits Riemann invariants for all $\ell=1,2,3,\dots$.  What is more surprising is that for Benjamin-Ono, the modulation system written in the Riemann invariant variables completely decouples into $2\ell+1$ independent copies of the inviscid Burgers equation.

In 1974, Gurevich and Pitaevski\v{i} \cite{GurevichP74} used a matching formalism in the Korteweg-de Vries setting to connect an expanding dispersive shock wave described by the $\ell=1$ Whitham theory with appropriate solutions of the inviscid Burgers approximation on the left and right.  This gives a reasonable formal approximation to the solution of a Cauchy problem in the small-dispersion limit.  This formal approach has been adapted to the Benjamin-Ono problem also.  Like Gurevich-Pitaevski\v{i}, Matsuno \cite{Matsuno98} considered step-like initial data and constructed a global approximation of the solution in the small-$\varepsilon$ limit that develops as a single $\ell=1$ dispersive shock wave.  Subsequently, Jorge, Minzoni, and Smyth \cite{JorgeMS99} analyzed  more general initial data leading to a global approximate solution involving also multiphase waves for $\ell=2$.  

Making results such as these mathematically rigorous requires passing from the weak convergence setting of Theorem~\ref{thm:weak-small-dispersion} to a stronger topology.  The issue is that there can be no limit in the strong $L^2(\R)$ sense due to the presence of the oscillations forming a dispersive shock wave, which have a characteristic wavelength proportional to $\varepsilon$ and a nonvanishing amplitude.  So, to have strong convergence it is first necessary to construct a model for the solution that contains wild oscillations, and then to consider the difference between the true solution of the Cauchy problem and the model.  Then a strong limit may exist, and if the model is accurate, the limit should be zero.

For the Korteweg-de Vries equation \eqref{eq:KdV} in the small-dispersion limit, strong asymptotics have been obtained by Deift, Venakides, and Zhou \cite{DeiftVZ97,DeiftVZ98} using Riemann-Hilbert techniques.  An interesting and compelling aspect of these works is that the asymptotic method employed does not require any ansatz for the dispersive shock wave such as that assumed in Whitham modulation theory; rather the correct leading term (model) is extracted systematically and directly from the Riemann-Hilbert representation of the solution of the Cauchy problem.  Since the formulation of the solution of the Benjamin-Ono Cauchy problem \eqref{eq:BO}--\eqref{eq:IC} in terms of a suitable Riemann-Hilbert problem remains unjustified, and anyway it is not clear how the Deift-Venakides-Zhou methodology might be extended to apply to the Fokas-Ablowitz inverse problem, a new approach is needed to establish strong small-dispersion asymptotics for the Benjamin-Ono equation.  That approach is the use of the explicit formula in the setting of rational initial data, Theorem~\ref{thm:inversion-formula-rational}.

To formulate the result, we first specify the relevant oscillatory model denoted $u^\mathrm{ZD}(t,x;\varepsilon)$ for the solution of the Cauchy problem \eqref{eq:BO}--\eqref{eq:IC}.  It is well-defined for all $\varepsilon\neq 0$ whenever (i) the stronger conditions of Theorem~\ref{thm:weak-small-dispersion} hold, namely that $u_0\in L^\infty(\R)\cap L^2(\R)\cap C^1(\R)$ with $u_0$ and $u_0'$ tending to zero as $x\to\pm\infty$, and (ii) $(t,x)\in\R^2$ is such that $\ell(t,x)$ has a definite value.  For simplicity, here we only define $u^\mathrm{ZD}(t,x;\varepsilon)$ for rational $u_0$ of the form \eqref{eq:PN-rational-IC} and in the cases $\ell(t,x)=0,1$; the more general definition and some of its properties are described in Appendix~\ref{app:uZDgeneral}. 

If $\ell(t,x)=0$, then there is only one simple real solution $y=y_0(t,x)$ of the characteristic equation $y+2tu_0(y)=x$ and hence only one branch $u^\mathrm{B}_0(t,x):=u_0(y_0(t,x))$ of the inviscid Burgers solution; in this case we set 
\begin{equation}
    u^\mathrm{ZD}(t,x;\varepsilon):=u_0^\mathrm{B}(t,x),\quad\ell(t,x)=0.
\label{eq:uZD-ell-0}
\end{equation}
If $\ell(t,x)=1$, then there are three real simple solutions $y$ of the characteristic equation, $y_0(t,x)>y_1(t,x)>y_2(t,x)$, and three corresponding branches $u_j^\mathrm{B}(t,x)=u_0(y_j(t,x))$ with $u_0^\mathrm{B}(t,x)<u_1^\mathrm{B}(t,x)<u_2^\mathrm{B}(t,x)$ of the multivalued inviscid Burgers solution. The function $h(\diamond;t,x)$ appearing in the exponent in \eqref{eq:PN-exph} can be written explicitly for $u_0$ of the form \eqref{eq:PN-rational-IC} as
\begin{equation}
    h(y;t,x):=\frac{(y-x)^2}{4t}+\sum_{n=0}^N\left[c_n\log(y-p_n)+c_n^*\log(y-p_n^*)\right],
    \label{eq:h-rational}
\end{equation}
which is real valued for $y\in\R$.  We use \eqref{eq:h-rational} to define a real phase variable $\theta_1(t,x)$ by
\begin{equation}
    \theta_1(t,x):=h(y_1(t,x))-h(y_2(t,x)).
\end{equation}
We next introduce the auxiliary function
\begin{equation}
    \Psi(y;t,x):=\frac{\pi}{2}+\sum_{n=1}^N\arg(y-p_n)-\sum_{m=1}^{N-1}\arg(y-z_m(t,x)),
    \label{eq:Psi-ell=1}
\end{equation}
where $y=z_m(t,x)$, $m=1,\dots,N-1$, denote the complex solutions in $\C_+$ of the characteristic equation $y+2tu_0(y)=x$ (counted with multiplicity).
Using $\Psi(y;t,x)$, we define a phase correction $\varphi_1(t,x)$ by
\begin{equation}
    \varphi_1(t,x):=\frac{\pi}{2}+\Psi(y_1(t,x);t,x)-\Psi(y_2(t,x);t,x).
\end{equation}
Finally, when $\ell(t,x)=1$ we set
\begin{equation}
    u^\mathrm{ZD}(t,x;\varepsilon):=u_0^\mathrm{B}(t,x)+(u_2^\mathrm{B}(t,x)-u_1^\mathrm{B}(t,x))U\left(\frac{\theta_1(t,x)}{\varepsilon}+\varphi_1(t,x);r(t,x)\right),\quad \ell(t,x)=1,
\label{eq:uZDell1}
\end{equation}
where
\begin{equation}
    0<r(t,x):=\sqrt{\frac{u_1^\mathrm{B}(t,x)-u_0^\mathrm{B}(t,x)}{u_2^\mathrm{B}(t,x)-u_0^\mathrm{B}(t,x)}}<1\quad\text{and}\quad U(\phi;r):=\frac{1-r^2}{1+r^2-2r\cos(\phi)}.
\label{eq:r-and-periodic-profile}
\end{equation}
It is not difficult to check that if for some $t\in\R$, $\ell(t,x)=1$ on a maximal interval $x\in (a,b)$, then $u_1^\mathrm{B}(t,x)-u_0^\mathrm{B}(t,x)\downarrow 0$ as $x\downarrow a$ while $u_2^\mathrm{B}(t,x)-u_1^\mathrm{B}(t,x)\downarrow 0$ as $x\uparrow b$; it follows that $u^\mathrm{ZD}(t;x,\varepsilon)$ is always continuous at $x=a$, but generally has a jump discontinuity at $x=b$.  One can also check that if $U_j:=u_j^\mathrm{B}(t,x)$ are all constants for $j=0,1,2$, then $\theta_1(t,x)=(U_1-U_2)x-(U_1^2-U_2^2)t + \theta_1^0$ for some constant $\theta_1^0$ and $\varphi_1(t,x)$ is also a constant.  Thus \eqref{eq:uZDell1}--\eqref{eq:r-and-periodic-profile} reduces to an exact periodic traveling wave solution of \eqref{eq:BO} of the form found by Benjamin \cite{Benjamin67} and studied by Amick-Toland \cite{AmickT91}.

The results are then the following.
\begin{theorem}[Strong $L^2$ small-dispersion asymptotics, \protect{\cite[Corollary 1.6]{BlackstoneGGM26}}]
Suppose that $u_0\in L^2(\mathbb{R})$ is a rational function with simple poles of the form \eqref{eq:PN-rational-IC}, and let $u(t,x)=u(t,x;\varepsilon)$ denote the solution of the Benjamin-Ono Cauchy problem \eqref{eq:BO}--\eqref{eq:IC} with initial condition $u_0$.  If $t\ge 0$ is such that $\ell(t,x)\le 1$ for generic $x\in\mathbb{R}$, then 
\begin{equation}
\lim_{\varepsilon\to 0}\|u(t,\diamond;\varepsilon)-u^\mathrm{ZD}(t,\diamond;\varepsilon)\|_{L^2(\mathbb{R})}=0.
\end{equation}
\label{thm:strong-L2}
\end{theorem}
This statement is precisely of the form described above; after subtracting off a suitable oscillatory model, one recovers convergence (to zero) in the strong $L^2(\R)$ topology.

\begin{theorem}[Strong pointwise and locally-uniform small-dispersion asymptotics, \protect{\cite[Theorem 1.5]{BlackstoneGGM26}}]
 Under the same conditions on $u_0$, if $(t,x)\in\mathbb{R}^2$ is a point such the characteristic equation $y+2tu_0(y)=x$ has only simple real solutions, then in particular $\ell(t,x)$ is a well-defined nonnegative integer and
\begin{equation}
u(t,x;\varepsilon)=u^\mathrm{ZD}(t,x;\varepsilon)+O(\varepsilon)
\end{equation}
holds pointwise and also locally uniformly in the small-dispersion limit $\varepsilon\to 0$.
\label{thm:strong-ZD-asymptotics}
\end{theorem}
The proof of this theorem in \cite{BlackstoneGGM26} is based on the asymptotic expansion of the contour integral elements of the matrices $\mathbf{A}(t,x)=\mathbf{A}(t,x;\varepsilon)$ and $\mathbf{B}(t,x)=\mathbf{B}(t,x;\varepsilon)$ defined in \eqref{eq:Amatrix}--\eqref{eq:Bmatrix} using the method of steepest descent, which in turn requires knowledge of the (generally complex) critical points of the exponent function $z\mapsto h(z;t,x)$.  \emph{The key observation, which points toward the appearance of the multivalued solution of the inviscid Burgers equation in the small-dispersion asymptotics, is that $z\in\C$ is a critical point of $h(\diamond;t,x)$ if and only if it is a solution of the characteristic equation $z+2tu_0(z)=x$.} If integrals over two different contours are dominated by contributions from the same critical point of $h$, then the corresponding two rows of each matrix will have the same asymptotic expansion to all orders, which makes the asymptotic behavior of the fraction in \eqref{eq:Pi-u-fraction} appear to be indeterminate.  In reality, this means that the difference between the two rows in question lies beyond all orders in $\varepsilon$.  To avoid having to find the exponentially small terms, the first step is to show that suitable contours can be selected so that the integrals in each row are dominated by distinct complex critical points or pairs of consecutive simple real critical points.  A recent argument \cite{Cheng26+} shows that a version of this step applies when one or more complex critical points have higher multiplicity.  With this done, one can apply the standard saddle point approach to obtain the leading term as $\varepsilon\to 0$ of $\det(\mathbf{A}(t,x))/\det(\mathbf{B}(t,x))$. Then one observes that using Schur complements (or the Woodbury formula) reduces the dimension of the limiting determinants from $(N+1)\times (N+1)$ to $\ell(t,x)\times\ell(t,x)$, where $2\ell(t,x)+1$ is the number of real critical points of $h$.  After taking twice the real part, it remains to identify the resulting leading term with the expression $u^\mathrm{ZD}(t,x;\varepsilon)$ as defined above and in greater generality in Appendix~\ref{app:uZDgeneral}.  As was the case in the Korteweg-de Vries theory \cite{DeiftVZ97,DeiftVZ98}, the formula for the leading term follows directly from the asymptotic analysis of the solution representation (here the explicit formula \eqref{eq:Pi-u-fraction} rather than a Riemann-Hilbert problem) rather than being provided a priori as an ansatz.  

Accounting for the preservation of the $L^2$-norm of solutions of Benjamin-Ono, the locally uniform convergence result of Theorem~\ref{thm:strong-ZD-asymptotics} together with the condition $\ell(t,x)\le 1$ yields the mean-square convergence result in Theorem~\ref{thm:strong-L2} as a corollary.
The reason for the restriction to $\ell(t,x)\le 1$ is that in the proof given in \cite{BlackstoneGGM26}, one of the terms that needs to be estimated involves an average of the (generally multiphase) approximation.  This is straightforward to calculate in the zero-phase or single-phase cases, but in the multiphase case obtaining a similar result requires avoiding resonances that can change the values of averages.

A numerical comparison between $u(t,x;\varepsilon)$ for $u_0(x)=(2ax-2b)/(x^2+1)$ from Section~\ref{sec:examples} computed by numerical evaluation of the contour integrals in \eqref{eq:N=1-ContourIntegrals} and $u^\mathrm{ZD}(t,x;\varepsilon)$ calculated directly from \eqref{eq:uZD-ell-0}--\eqref{eq:r-and-periodic-profile} using Remark~\ref{rem:weak-one-pole} is shown in Figure~\ref{fig:uZDvsExact}.  
\begin{figure}[h]
\begin{center}
    \includegraphics[width=0.49\linewidth]{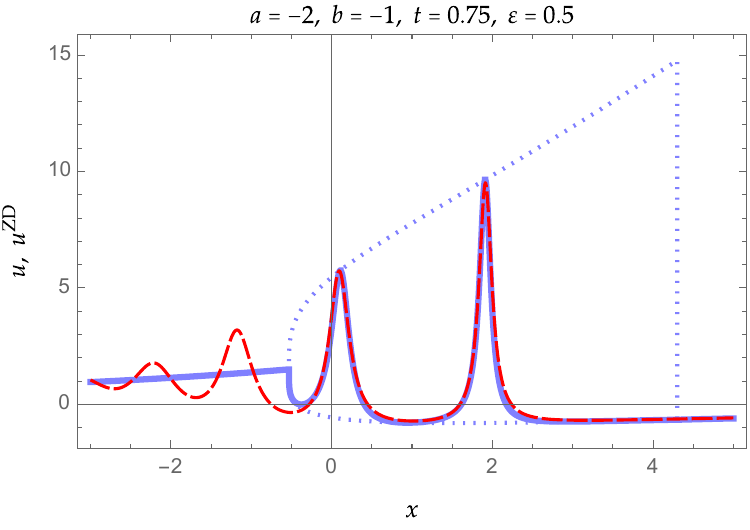}\hfill\includegraphics[width=0.49\linewidth]{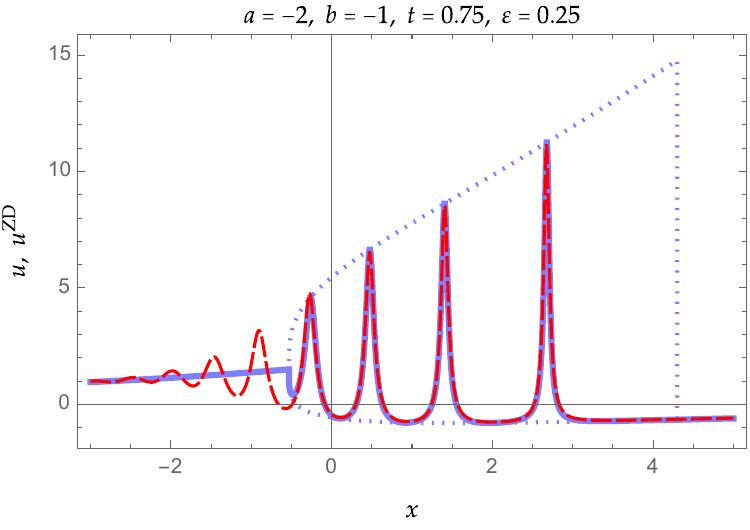}
\end{center}
\caption{Comparison between $u^\mathrm{ZD}(t,x;\varepsilon)$ (thick blue) and $u(t,x;\varepsilon)$ (red, dashed) for $u_0(x)=(2ax-2b)/(x^2+1)$ with $a=-2$ and $b=-1$ for $t=0.75$.  The upper and lower envelopes of $u^\mathrm{ZD}(t,x,\varepsilon)$ (independent of $\varepsilon$ and obtained by replacing $\cos(\phi)$ in \eqref{eq:r-and-periodic-profile} $1$ and $-1$ respectively) are plotted with dotted blue lines.  Left:  $\varepsilon=0.5$.  Right:  $\varepsilon=0.25$.}
\label{fig:uZDvsExact}
\end{figure}
One can clearly see that the approximation is least accurate near the point on the left where the upper and lower envelopes of the oscillatory dispersive shock wave meet, and although the jump discontinuity of $u^\mathrm{ZD}(t,x;\varepsilon)$ corresponding to the obvious jump of the upper envelope is so small as to be invisible in these plots, it is nonzero and therefore there is also a loss of accuracy there.  Improving the approximation near these points, which lie on the caustic locus for the fixed value of $t$ in question, is the subject of the next section.

\subsection{Small-dispersion asymptotics near the caustic locus and universality}
\label{sec:rational-universality}
The closed set $\mathcal{C}\subset\mathbb{R}^2$ of points in the real $(t,x)$-plane for which $x$ is a critical value of the mapping $\R\ni y\mapsto y+2tu_0(y)$ (equivalently, for which there exists a real critical point of $y\mapsto h(y;t,x)$ of higher multiplicity) is called the \emph{caustic locus} for the multivalued solution of the inviscid Burgers equation.  For each $\varepsilon>0$, the model function $u^\mathrm{ZD}(t,x;\varepsilon)$ of Theorem~\ref{thm:strong-ZD-asymptotics} is defined for $(t,x)\in\R^2\setminus\mathcal{C}$, and Theorem~\ref{thm:strong-ZD-asymptotics} is only applicable on this domain.  While this theorem asserts locally uniform accuracy of $u^\mathrm{ZD}(t,x;\varepsilon)$ for each $(t,x)\in\R\setminus\mathcal{C}$,
there is no assertion that accuracy persists upon allowing $(t,x)$ to approach a point of $\mathcal{C}$.  Although any lack of uniformity near $\mathcal{C}$ is not relevant for the $L^2$ convergence result in Theorem~\ref{thm:strong-L2}, still one may ask how the solution of the Cauchy problem \eqref{eq:BO}--\eqref{eq:IC} behaves when $\varepsilon$ is small and also $(t,x)$ is close to the caustic locus $\mathcal{C}$.  

Although there may be several gradient catastrophe points for the inviscid Burgers solution for given $u_0$, we now focus on the first-occurring such point for $t>0$.  Let $t_\mathrm{c}\in\R$ be defined by
\begin{equation}
    t_\mathrm{c}:=\min \{t>0\mid (t,x)\in\mathcal{C}\}.
\end{equation}
Assuming the minimum is finite (which is the same as the assumption that the inviscid Burgers solution is not globally classical for positive $t$), the simplest and most general situation is that there is a unique critical value $x=x_\mathrm{c}$ such that $(t_\mathrm{c},x_\mathrm{c})\in\mathcal{C}$.  The latter point is the earliest point of gradient catastrophe of the inviscid Burgers solution with initial data $u_0$.  At this point one has a unique solution $y=y_\mathrm{c}\in\R$ of $y+2t_\mathrm{c}u_0(y)=x_\mathrm{c}$ that also satisfies $1+2t_\mathrm{c}u_0'(y)=0$ and $2t_\mathrm{c}u_0''(y)=0$, while generically it also holds that $2t_\mathrm{c}u_0'''(y)> 0$.  In this situation, for small $t-t_\mathrm{c}>0$, there is a unique maximal open interval $x_\mathrm{h}(t)<x<x_\mathrm{s}(t)$ in which $\ell(t,x)=1$ while $\ell(t,x)=0$ both for $x<x_\mathrm{h}(t)$ and for $x>x_\mathrm{s}(t)$.  The interval $x_\mathrm{h}(t)<x<x_\mathrm{s}(t)$ contains the highly oscillatory dispersive shock wave.  We call the curve $x=x_\mathrm{h}(t)$ (respectively, $x=x_\mathrm{s}(t)$) the \emph{harmonic edge}  (respectively, the \emph{soliton edge}) of the Benjamin-Ono dispersive shock wave.  Referring to Figure~\ref{fig:uZDvsExact}, $x=x_\mathrm{h}(t)$ is the point at which the upper and lower envelopes of $u^\mathrm{ZD}(t,x;\varepsilon)$ come together, while $x=x_\mathrm{s}(t)>x_\mathrm{h}(t)$ is the point where the upper envelope has its jump discontinuity.

The basic description given above also holds for other equations supporting dispersive shock wave formation when the dispersion is small. For general weakly dispersive perturbations of the inviscid Burgers equation, the space-time point from which the first dispersive shock wave is generated is still the same gradient catastrophe point $(t_\mathrm{c},x_\mathrm{c})$.  However, in general the curves $x=x_\mathrm{h}(t)$ and $x=x_\mathrm{s}(t)$ emanating from the catastrophe point in the direction of increasing time $t$ bounding the dispersive shock wave are not related to the multivalued solution of the inviscid Burgers equation for initial data $u_0$.  Instead, they are defined as a degeneration locus for a solution of the $\ell=1$ Whitham modulation equations, a $3\times 3$ first-order hyperbolic system that, unlike in the Benjamin-Ono case, cannot be uncoupled even after being reduced to Riemann-invariant form. 
Nonetheless, for a large class of dispersive Hamiltonian perturbations of the inviscid Burgers equation that includes the Korteweg-de Vries equation \eqref{eq:KdV} as a special case, Dubrovin \cite{Dubrovin08} conjectured a remarkable universality property of the solution for general initial data in the small-dispersion limit with variables $(t,x)$ localized in a neighborhood of a generic gradient catastrophe point.  Specifically, the claim is that in scaled variables, the solution of the perturbed problem is asymptotic to a particular solution of the second equation in the Painlev\'e-I hierarchy.  In the special case of the Korteweg-de Vries equation \eqref{eq:KdV}, Claeys and Grava \cite{ClaeysG09} proved Dubrovin's conjecture (i.e., the universality of the limiting profile with respect to initial data for this specific equation was established) using Riemann-Hilbert methods.  The same authors subsequently analyzed the Korteweg-de Vries solution near the harmonic edge $x=x_\mathrm{h}(t)$ \cite{ClaeysG10a} and near the soliton edge $x_\mathrm{s}(t)$ \cite{ClaeysG10b}.  In both cases, it was proved that the solution exhibits universality with respect to the initial data $u_0$ with properties different at both edges but both being expressible in terms of known functions (the Hastings-McLeod solution of the homogeneous Painlev\'e-II equation modulating a sinusoidal oscillation for the harmonic edge, and a sum of $\mathrm{sech}^2$ soliton profiles with position shifts related to the norming constants of the Hermite polynomials for the soliton edge).  There is no conjecture in the literature asserting that the behavior near the curves $x=x_\mathrm{h}(t)$ or $x=x_\mathrm{s}(t)$ should be independent of the dispersive perturbation under consideration.

The following results are the analogues for Benjamin-Ono of those in \cite{ClaeysG09,ClaeysG10a,ClaeysG10b}.
\subsubsection{Universality near the gradient catastrophe point}
First, we give the Benjamin-Ono version of the result in \cite{ClaeysG09} for the Korteweg-de Vries equation.  Let $L$ be an ``elbow-shaped'' contour in the complex $\zeta$-plane originating from $\zeta=\infty\ee^{-7\pi\ii/8}$ and tending to infinity in the opposite direction toward $\zeta=\infty\ee^{-3\pi\ii/8}$, and denote by $\tau(T,X)$ the corresponding \emph{Pearcey integral}
\begin{equation}
    \tau(T,X):=\int_{L}\ee^{\ii(\zeta^4-T\zeta^2-X\zeta)}\,\dd\zeta.
    \label{eq:Pearcey-tau}
\end{equation}
This is an entire function of $(T,X)\in\C^2$ and it can be shown \cite{Cheng26+b} that $\tau(T,X)\neq 0$ for all $(T,X)\in\R^2$.
Now set
\begin{equation}
    U(T,X):=-2\mathrm{Im}\left(\frac{\partial}{\partial X}\log(\tau(T,X))\right),\quad (T,X)\in\R^2.     \label{eq:CatastropheU}
\end{equation}
\begin{theorem}[Universal small-dispersion asymptotics near the gradient catastrophe point, \protect{\cite[Theorem 1.7]{BlackstoneMM26}}]
Suppose that $u_0\in L^2(\R)$ is a rational function with simple poles of the form \eqref{eq:PN-rational-IC}, and let $u(t,x;\varepsilon)$ denote the solution of the Benjamin-Ono Cauchy problem \eqref{eq:BO}--\eqref{eq:IC} with initial condition $u_0$.  Assume that $(t_\mathrm{c},x_\mathrm{c})$ is a generic gradient catastrophe point.
Set $u_\mathrm{c}:=u_0(y_\mathrm{c})$ (the value of the inviscid Burgers solution at the catastrophe point for initial data $u_0$) and let $k:=(\frac{2}{3})^{1/4}t_\mathrm{c}u_0'''(y_\mathrm{c})^{1/4}>0$.  Defining localized coordinates by 
\begin{equation}
    T:=\frac{t-t_\mathrm{c}}{k^2\varepsilon^{1/2}}\quad\text{and}\quad X:=\frac{x-x_\mathrm{c}-2u_\mathrm{c}(t-t_\mathrm{c})}{k\varepsilon^{3/4}},
\end{equation}
the asymptotic formula
\begin{equation}
    u(t,x;\varepsilon)=u_\mathrm{c}+\frac{\varepsilon^{1/4}}{k}U(T,X)+O(\varepsilon^{1/2}),\quad\varepsilon\to 0
    \label{eq:catastrophe-asymp}
\end{equation}
holds uniformly for $(T,X)$ in compact subsets of $\R^2$. 
\label{thm:catastrophe}
\end{theorem}
The coordinate $X$ is fixed in a frame of reference moving with the characteristic velocity $2u_\mathrm{c}$ of the inviscid solution at the catastrophe point.  The form \eqref{eq:catastrophe-asymp} was conjectured in \cite{MasoeroRA15} without an explicit formula for $U(T,X)$, but instead giving a formal argument that $U(T,X)$ should satisfy a specific equation, which in our scaling reads
\begin{equation}
    X-2UT+4U^3-6U|D_X|U-3|D_X|U^2-4\partial_X^2U=0,
    \label{eq:U-eqn}
\end{equation}
where $|D_X|$ denotes the Fourier multiplier defined in \eqref{eq:BO} acting on functions of $X$.
Numerical evidence described in \cite{BlackstoneMM26} suggests that indeed the function defined by \eqref{eq:CatastropheU} satisfies \eqref{eq:U-eqn}.  The Benjamin-Ono equation does not fall within the universality class of equations considered by Dubrovin \cite{Dubrovin08}, and according to \cite{MasoeroRA15}, for Benjamin-Ono the equation \eqref{eq:U-eqn} should play the role played instead for equations of the Dubrovin class by the second equation in the Painlev\'e-I hierarchy.

To illustrate the key idea of universality, we present in Figure~\ref{fig:catastrophe} some examples of the convergence asserted in Theorem~\ref{thm:catastrophe} for different initial data taken from the two-parameter family $u_0(x)=2(ax-b)/(x^2+1)$ described in Section~\ref{sec:examples}.  
\begin{figure}[h]
\begin{center}
    \includegraphics[width=0.3\linewidth]{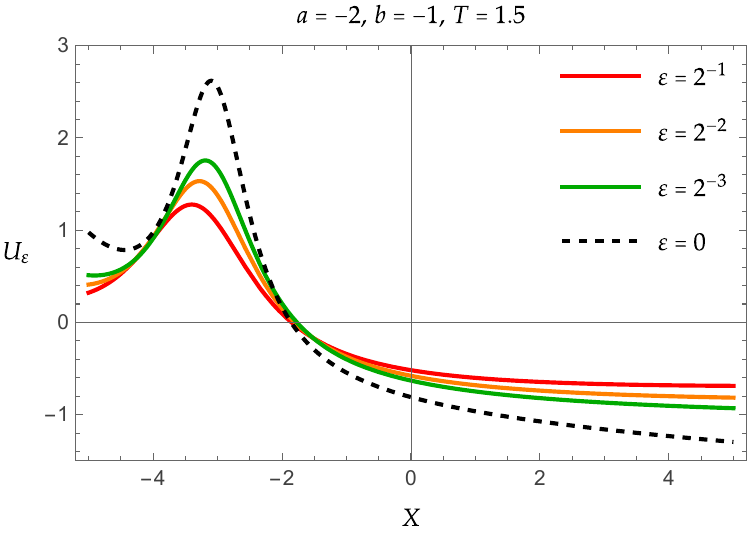}\hfill%
    \includegraphics[width=0.3\linewidth]{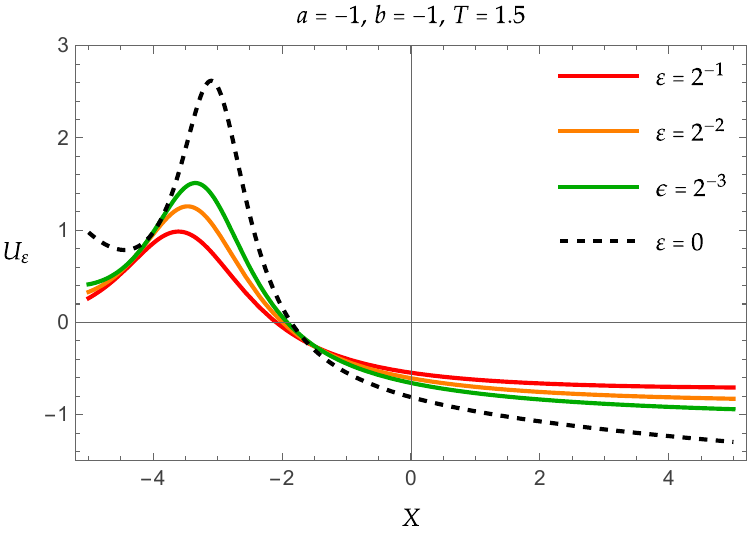}\hfill%
    \includegraphics[width=0.3\linewidth]{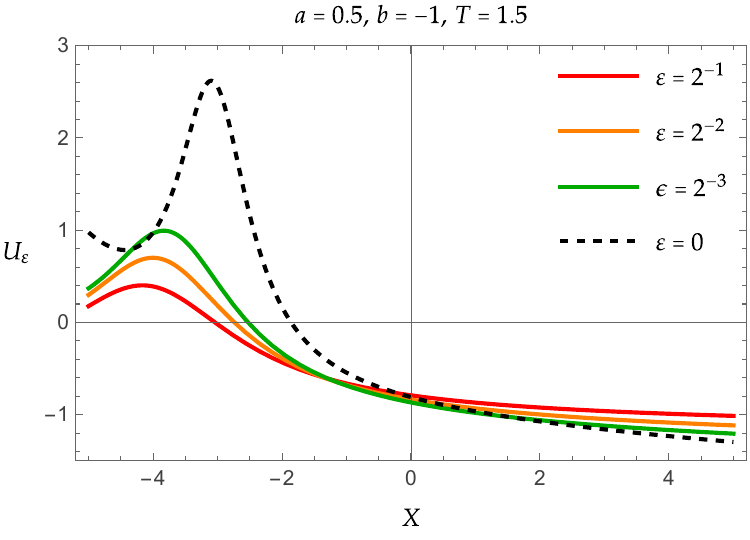}
\end{center}
\caption{Plots of $U_\varepsilon(T,X):=k\varepsilon^{-1/4}(u(t_\mathrm{c}+\varepsilon^{1/2}k^2T,x_\mathrm{c}+2\varepsilon^{1/2}k^2Tu_\mathrm{c}+\varepsilon^{3/4}kX;\varepsilon)-u_\mathrm{c})$ (colored curves) and $U_0(T,X):=U(T,X)$ (dashed black curve) for $T=1.5$ as functions of $X$.  Left:  $a=-2$, $b=-1$ (giving $t_\mathrm{c}=0.116$, $x_\mathrm{c}=0.467$, $u_\mathrm{c}=1.344$, and $k=0.235$).  Center:  $a=-1$, $b=-1$ (giving $t_\mathrm{c}=0.196$, $x_\mathrm{c}=0.804$, $u_\mathrm{c}=1.367$, and $k=0.344$).  Right:  $a=0.5$, $b=-1$ (giving $t_\mathrm{c}=0.473$, $x_\mathrm{c}=2.417$, $u_\mathrm{c}=1.701$, and $k=0.599$).  The convergence to the universal limit $U_0(T,X)$ is clear for each given case, but the rate does not appear to be uniform with respect to $a$.}
\label{fig:catastrophe}
\end{figure}
For $a=0$ and $b=-1$, such a study was presented in \cite[Figure 8]{BlackstoneMM26} using the fact that the exact solution $u(t,x;\varepsilon)$ for a discrete sequence $\varepsilon=\varepsilon_n$, $n=1,2,3,\dots$, $\varepsilon_n\to 0$, is an exact $n$-soliton solution for which Matsuno's formula \eqref{eq:Matsuno}--\eqref{eq:Matsuno-Matrix} applies.  This allows one to consider much smaller values of $\varepsilon$ than is numerically tractable using instead the evaluation of the contour integrals in \eqref{eq:N=1-ContourIntegrals}.  Nonetheless for $a\neq 0$ we computed these integrals for $\varepsilon=2^{-k}$ for $k=1,2,3$ to make the plots shown here.

\subsubsection{Universality near the harmonic edge}
Next, we consider the asymptotic behavior of $u(t,x;\varepsilon)$ for $(t,x)$ near the harmonic edge $x=x_\mathrm{h}(t)$.  For Benjamin-Ono, this curve is determined by the locally multivalued solution of the inviscid Burgers approximation as follows.  For $t-t_\mathrm{c}>0$ small enough, and assuming that $(t_\mathrm{c},x_\mathrm{c})$ is a generic gradient catastrophe point, the condition $x=x_\mathrm{h}(t)$ ensures that there exist two distinct real solutions $y$ of the characteristic equation $x=2tu_0(y)+y$, denoted $y=y_\mathrm{h}(t)$ (a double root with $u_0''(y_\mathrm{h}(t))>0$) and $y=y_2(t)$ (a simple root), with $y_\mathrm{h}(t)>y_2(t)$.  There are also $N-1$ complex solutions (counting with multiplicity) $y=z_{j,\mathrm{h}}(t)\in\mathbb{C}_+$, $j=2,\dots,N$, of the characteristic equation.  The corresponding branches of the inviscid Burgers solution are denoted $u_\mathrm{h}^\mathrm{B}(t):=u_0(y_\mathrm{h}(t))$ (really the value of the solution at the fold point where two smooth branches for $x>x_\mathrm{h}(t)$ meet with a vertical tangent) and $u_2^\mathrm{B}(t):=u_0(y_2(t))$ (the value of the branch that is smooth at $x=x_\mathrm{h}(t)$).  Note that $u_\mathrm{h}^\mathrm{B}(t)<u_2^\mathrm{B}(t)$.  
With this setup, 
we define a real phase by
\begin{equation}
    \Theta(t,x):=\left(u_2^\mathrm{B}(t)-u_\mathrm{h}^\mathrm{B}(t)\right)(x-x_\mathrm{h}(t)) + \left(u_2^\mathrm{B}(t)^2-u_\mathrm{h}^\mathrm{B}(t)^2\right)t-\int_{y_2(t)}^{y_\mathrm{h}(t)}u_0(y)\,\dd y,
\end{equation}
and a phase correction by 
\begin{multline}
    \varphi(t,X_\mathrm{h}):=\arg\left(\mathrm{Ai}(\ee^{\ii\pi/3}X_\mathrm{h})\right)-\frac{5\pi}{12} +\sum_{j=2}^N\left[\arg(z_{j,\mathrm{h}}(t)-y_\mathrm{h}(t))-\arg(z_{j,\mathrm{h}}(t)-y_2(t))\right] \\
    - \sum_{n=1}^N\left[\arg(y_\mathrm{h}(t)-p_n)-\arg(y_2(t)-p_n)\right],
\label{eq:harmonic-edge-phase-correction-rational}
\end{multline}
where $\mathrm{Ai}(\diamond)$ denotes the standard Airy function (see \cite[Chapter 9]{DLMF}).
We have the following result.
\begin{theorem}[Universal small-dispersion asymptotics near the harmonic edge, \protect{\cite[Theorem 1.5]{BlackstoneMM26}}]
Suppose that $u_0\in L^2(\R)$ is a rational function with simple poles of the form \eqref{eq:PN-rational-IC}, and let $u(t,x)=u(t,x;\varepsilon)$ denote the solution of the Benjamin-Ono Cauchy problem \eqref{eq:BO}--\eqref{eq:IC}.  Assume that $(t_\mathrm{c},x_\mathrm{c})$ is a generic gradient catastrophe point.
Set $k_\mathrm{h}(t):=2t[\frac{1}{2}u_0''(y_\mathrm{h}(t))]^{1/3}>0$, and define a localized spatial coordinate by 
\begin{equation}
    X_\mathrm{h}:=\frac{x-x_\mathrm{h}(t)}{k_\mathrm{h}(t)\varepsilon^{2/3}}.
\end{equation}
Then 
the asymptotic formula
\begin{equation}
    u(t,x;\varepsilon)=u_2^\mathrm{B}(t)+\varepsilon^{1/6}\sqrt{\frac{2(u_2^\mathrm{B}(t)-u_\mathrm{h}^\mathrm{B}(t))}{\pi k_\mathrm{h}(t)}}\frac{\cos(\varepsilon^{-1}\Theta(t,x)+\varphi(t,X_\mathrm{h}))}{|\mathrm{Ai}(\ee^{\ii\pi/3}X_\mathrm{h})|} + O(\varepsilon^{1/3})
    \label{eq:harmonic-edge-asymp}
\end{equation}
holds as $\varepsilon\to 0$ uniformly for $t-t_\mathrm{c}>0$ sufficiently small and $X_\mathrm{h}$ bounded.  
\label{thm:harmonic-edge}
\end{theorem}
We give an equivalent expression for the phase correction $\varphi(t,X_h)$ in equations \eqref{eq:g-h}--\eqref{eq:h-phase-correction} in  Appendix~\ref{app:uZDgeneral} that avoids any reference to the discrete data for the rational initial data $u_0$ and therefore applies to more general $u_0$. 
Theorem~\ref{thm:harmonic-edge} is an analogue for the Benjamin-Ono equation of a corresponding Korteweg-de Vries result obtained in \cite{ClaeysG10a}.  Both results display the dispersive solution as a correction to the continuous branch of the inviscid Burgers solution continued from the region to the left of the dispersive shock wave, and in both cases the correction is a modulated sinusoidal oscillation.  However in the Benjamin-Ono case the amplitude of the oscillatory correction is larger than in the Korteweg-de Vries case (proportional to $\varepsilon^{1/6}$ rather than $\varepsilon^{1/3}$).  Also, while the modulation of the oscillation is described in the Benjamin-Ono case by the classical special function $1/\mathrm{Ai}(\ee^{\ii\pi/3}X_\mathrm{h})$ (note that this expression is analytic for $X_\mathrm{h}\in\R$), it is instead described in the Korteweg-de Vries case by the Hastings-McLeod solution of the second Painlev\'e equation, a far more transcendental object.  In both the Benjamin-Ono and Korteweg-de Vries cases, the wavelength and frequency of the sinusoidal oscillation are proportional to $\varepsilon$, while the characteristic length scale of the modulation is proportional to $\varepsilon^{2/3}$.  For a more detailed description of the asymptotic formula \eqref{eq:harmonic-edge-asymp} including its relationship to formal linearization and group velocity propagation, and for further comparison with the Korteweg-de Vries case, we refer to \cite[Section 1.2]{BlackstoneMM26}.

\subsubsection{Universality near the soliton edge}
Finally, we describe the weakly-dispersive solution of Benjamin-Ono near the soliton edge $x=x_\mathrm{s}(t)$.  The curve $x=x_\mathrm{s}(t)$ is defined so that when $(t_\mathrm{c},x_\mathrm{c})$ is a generic gradient catastrophe point for the inviscid Burgers solution with initial data $u_0$, the conditions $x=x_\mathrm{s}(t)$ and $t-t_\mathrm{c}>0$ sufficiently small imply the existence of two distinct real solutions $y=y_\mathrm{s}(t)$ (a double root with $u_0''(y_\mathrm{s}(t))<0$) and $y_0(t)$ (a simple root) of the characteristic equation $x=2tu_0(y)+y$ such that $y_\mathrm{s}(t)<y_0(t)$.  The corresponding branches of the multivalued inviscid Burgers solution are $u^\mathrm{B}_\mathrm{s}(t):=u_0(y_\mathrm{s}(t))$ (the coalescence of two smooth branches for $x<x_\mathrm{s}(t)$ over $x=x_\mathrm{s}(t)$ with a vertical tangent) and a smooth branch $u^\mathrm{B}_0(t)$ representing the unique inviscid Burgers solution continued from the right to $x=x_\mathrm{s}(t)$.  We have $u_\mathrm{s}^\mathrm{B}(t)>u_0^\mathrm{B}(t)$.  There are also $N-1$ (counting with multiplicity) complex solutions $y=z_{j,\mathrm{s}}(t)\in\C_+$, $j=2,\dots,N$, of the characteristic equation.

To formulate our result, we need to recall the Lorentzian profile of the Benjamin-Ono soliton with velocity $v>0$:
\begin{equation}
    L(\xi;v):=\frac{2v}{v^2\xi^2+1},\quad\xi\in\R,\quad v>0,
\label{eq:LorentzianProfile}
\end{equation}
such that the Benjamin-Ono equation \eqref{eq:BO} has the exact solution
\begin{equation}
    u(t,x;\varepsilon)=b+L(\varepsilon^{-1}(x-x_0-(v+2b)(t-t_0);v),\quad t_0,x_0,b\in\R,\quad v>0
    \label{eq:BO-soliton-on-background}
\end{equation}
representing a soliton propagating on a constant background $b\in\R$.  The correct velocity $v$ will be given by 
\begin{equation}
    v(t):=2\left(u_\mathrm{s}^\mathrm{B}(t)-u_0^\mathrm{B}(t)\right)>0,
    \label{eq:modulated-soliton-speed}
\end{equation}
and we define a time-dependent position shift by
\begin{equation}
    \lambda(t):=\mathrm{Im}\left(\sum_{j=2}^N\frac{1}{y_\mathrm{s}-z_{j,\mathrm{s}}(t)}-\sum_{n=1}^N\frac{1}{y_\mathrm{s}(t)-p_n}\right).
    \label{eq:soliton-edge-lambda-rational}
\end{equation}
Then the result is the following.
\begin{theorem}[Universal small-dispersion asymptotics near the soliton edge, \protect{\cite[Corollary 1.4]{BlackstoneMM26}}]
Suppose that $u_0\in L^2(\R)$ is a rational function with simple poles of the form \eqref{eq:PN-rational-IC}, and let $u(t,x)=u(t,x;\varepsilon)$ denote the solution of the Benjamin-Ono Cauchy problem \eqref{eq:BO}--\eqref{eq:IC}.  Assume that $(t_\mathrm{c},x_\mathrm{c})$ is a generic gradient catastrophe point.
Set $k_\mathrm{s}(t):=2t[-\frac{1}{2}u_0''(y_\mathrm{s}(t))]^{1/3}>0$, define a localized spatial coordinate by
\begin{equation}
    X_\mathrm{s}:=\frac{x-x_\mathrm{s}(t)}{k_\mathrm{s}(t)\varepsilon^{2/3}},
\end{equation}
and denote by $a_1>a_2>a_3>\cdots$ the zeros of the Airy function $\mathrm{Ai}(\diamond)$.  
Then the asymptotic formula
\begin{equation}
    u(t,x;\varepsilon)=u_0^\mathrm{B}(t)+\sum_{n=1}^\infty L\left(\varepsilon^{-1/3}k_\mathrm{s}(t)(X_\mathrm{s}-a_n)+2t\lambda(t);v(t)\right) + O(\varepsilon^{1/3})
    \label{eq:modulated-solitons}
\end{equation}
holds as $\varepsilon\to 0$ uniformly for $t-t_\mathrm{c}>0$ sufficiently small and $X_\mathrm{s}$ bounded.  Moreover, for each $n=1,2,3,\dots$ and fixed $t_0$ with $t_0-t_\mathrm{c}>0$ sufficiently small
\begin{equation}
    u(t,x;\varepsilon)=u_0^\mathrm{B}(t_0) + L\left(\varepsilon^{-1}\left(x-x_0^{(n)}-2u_\mathrm{s}^\mathrm{B}(t_0)(t-t_0)\right);v(t_0)\right)+O(\varepsilon^{1/3})
    \label{eq:unmodulated-soliton}
\end{equation}
holds as $\varepsilon\to 0$ uniformly for $x-x_0^{(n)}=o(\varepsilon^{2/3})$ and $t-t_0=o(\varepsilon^{2/3})$, where an offset is defined by $x_0^{(n)}:=x_\mathrm{s}(t_0)+\varepsilon^{2/3}k_\mathrm{s}(t_0)a_n-2\varepsilon t_0\lambda(t_0)$.
\label{thm:soliton-train}
\end{theorem}
In Appendix~\ref{app:uZDgeneral}, equations \eqref{eq:g-s}--\eqref{eq:lambda}, we give an alternate formula for $\lambda(t)$ that only refers to the values of $u_0$ on the real line and hence applies to more general initial data.
The formula \eqref{eq:modulated-solitons} displays the dispersive solution as a (not small) correction to the background inviscid Burgers solution given by an infinite sum (convergent) of Lorentzian pulses with temporally modulated velocities and position shifts.  The second formula \eqref{eq:unmodulated-soliton} is a further localization of the first formula around a fixed time $t_0$ and a chosen pulse with index $n$.  In this version of the formula, the leading term is an exact solution of the form \eqref{eq:BO-soliton-on-background} with $b=u_0^\mathrm{B}(t_0)$, $v=v(t_0)$ given by \eqref{eq:modulated-soliton-speed}, and $x_0=x_0^{(n)}$.

Theorem~\ref{thm:soliton-train} is the Benjamin-Ono analogue of the corresponding result from \cite{ClaeysG10b} for the Korteweg-de Vries equation \eqref{eq:KdV}.  Both results describe the dispersive solution near the soliton edge as a train of identical exact soliton solutions of the underlying equation. The solitons are boosted in velocity by propagating on the locally-constant background of the inviscid Burgers solution continued to the soliton edge $x=x_\mathrm{s}(t)$ from the right of the dispersive shock wave.  In both cases, the solitons have width proportional to $\varepsilon$.  

The key differences between the Benjamin-Ono and Korteweg-de Vries results lie in the spatial displacements of the solitons from the soliton edge.  In the Korteweg-de Vries case, the soliton spacing is a constant multiple of $\varepsilon\ln(\varepsilon^{-1})$ with a correction proportional to $\varepsilon$ that depends on the index $n$ counting the solitons going leftward from the edge.  This correction involves the leading coefficient of the $n$th normalized Hermite polynomial, and it grows rapidly as $n$ increases.   On the other hand, it has the opposite sign to the leading term, so the spacing between consecutive solitons decreases with $n$.  This effect is significant because it is what allows for a formal matching between the description of the dispersive solution near the soliton edge and the approximation valid in the bulk of the dispersive shock wave, which is a modulated cnoidal periodic wave of wavelength proportional to $\varepsilon$.  Indeed, at a formal level as the index $n$ grows, the spacing will eventually become reduced from $O(\varepsilon\ln(\varepsilon^{-1}))$ to $O(\varepsilon)$ and the individual solitons will merge into a periodic traveling wave with the latter wavelength.  

On the other hand, in the Benjamin-Ono case the soliton spacing is much larger, proportional to  $\varepsilon^{2/3}$, and a similar kind of ``compressive'' effect of the soliton train moving into the dispersive shock wave from the right is achieved instead through the fact that the displacement of the $n$th soliton from the edge is proportional to the $n$th zero $a_n$ of the Airy function; as is well known, the sequence $\{a_n\}_{n=1}^\infty$ is decreasing and unbounded, but $a_n-a_{n+1}$ is asymptotically given by (see \cite[Eqns.\@ 9.9.6 and 9.9.18]{DLMF}) $(3n/(2\pi^2))^{-1/3}(1+O(n^{-1}))$ as $n\to\infty$ so the spacing of the Airy zeros tends to zero.  See \cite[Section 1.1]{BlackstoneMM26} for more details about the asymptotic behavior of solutions of the Benjamin-Ono and Korteweg-de Vries equations in the small-dispersion limit near their respective (different, for the same $u_0$) soliton edge curves.

The proofs given in \cite{BlackstoneMM26} of Theorems~\ref{thm:catastrophe}, \ref{thm:harmonic-edge}, and \ref{thm:soliton-train} again rely on the representation of the Benjamin-Ono solution $u(t,x;\varepsilon)$ in terms of determinants with contour integral entries as in Theorem~\ref{thm:inversion-formula-rational}.  Since for $(t,x)$ in the caustic locus $\mathcal{C}\subset\R^2$ there are real critical points of the exponent $y\mapsto h(y;t,x)$ (i.e., real solutions $y$ of the critical point equation $y+2tu_0(y)=x$) that are not simple, the standard steepest-descent method based on simple saddle points that applies to prove Theorem~\ref{thm:strong-ZD-asymptotics} no longer suffices.  Therefore it is necessary for $(t,x)$ close to $\mathcal{C}$ to first find suitable normal forms for $h$ that unfold the coalescence of critical points.  Near the gradient catastrophe point $(t_\mathrm{c},x_\mathrm{c})$ this normal form is a quartic polynomial, and the corresponding contour integrals are Pearcey integrals similar to $\tau(T,X)$ defined in \eqref{eq:Pearcey-tau}.  However near the harmonic and soliton edges the normal form is instead a cubic polynomial, and the corresponding contour integrals are Airy functions, which manifest themselves differently in the expansion of $u(t,x;\varepsilon)$ near the two edges.

\begin{remark}
The formulations of Theorems~\ref{thm:catastrophe}, \ref{thm:harmonic-edge}, and \ref{thm:soliton-train} in \cite{BlackstoneMM26} included an extra condition, that all complex solutions of the characteristic equation are simple.  However, this condition is not actually required for the results to hold \cite{Cheng26+}. 
\label{rem:non-simple-complex-critical-points}    
\end{remark}

\section{Applications to large-time asymptotic behavior}
\label{sec:large-time}
Now we turn our attention to the asymptotic behavior of the solution of the Cauchy problem \eqref{eq:BO}--\eqref{eq:IC} in the large-$t$ regime for fixed dispersion $\varepsilon>0$.  Without loss of generality, we therefore set $\varepsilon=1$ in this section.  

\subsection{General soliton resolution result}
Another way to write the Benjamin-Ono soliton solution given by \eqref{eq:LorentzianProfile}--\eqref{eq:BO-soliton-on-background} with zero background $b=0$ and taking $x_0=t_0=0$ for $\varepsilon=1$ is
\begin{equation}\label{eq:solitonBO}
u(t,x)=S_\gamma(x-v_\gamma t),
\quad
S_\gamma(y):=\frac{2\mathrm{Im}(\gamma)}{|y+\gamma|^2},
\quad
v_\gamma:=\frac{1}{\mathrm{Im}(\gamma)},
\end{equation}
where $\gamma\in \C_+$ is a complex number with a positive imaginary part.
\begin{theorem}[Soliton resolution for general initial data~\cite{GassotGM26}]\label{thm:solitonres}
Let $u_0\in H^2(\R)$ be a real--valued function such that $\diamond \,u_0(\diamond)\in H^1(\R)$. Assume moreover that there exist $\alpha_0\in \R$ and $w_0\in L^2(\R)$ such that
  \begin{equation}\label{hypu}
 \forall x\in \R,\quad  x^2u_0(x)=\alpha_0+w_0(x).
 \end{equation}
 Let $u(t,x)$ denote the solution of the Cauchy problem \eqref{eq:BO}--\eqref{eq:IC} with $\varepsilon=1$.  Then there exist a nonnegative integer $M$, points $\gamma_1,\dots,\gamma_M$ in $\C_+$ with $\mathrm{Im}(\gamma_1)<\dots <\mathrm{Im}(\gamma_M)$, and  real-valued functions $u_\infty ^\pm $ in $H^2(\R) $ such that
\begin{equation}\label{eq:solitonres}
 u(t,x)-\sum_{j=1}^M S_{\gamma_j}(x-v_{\gamma_j}t) -\ee^{t\partial_x|D|}u_\infty^\pm  \longrightarrow 0 
 \end{equation}
in $H^2(\R) $ as $t\to \pm\infty$.
Here, $\ee^{t\partial_x|D|}u_\infty^\pm$ denotes the solution of the linearized Benjamin-Ono equation $u_t=\partial_x|D|u$ with initial data $u(0,x)=u_\infty^\pm(x)$.
\end{theorem}

The assumptions on $u_0$ in Theorem \ref{thm:solitonres} can be slightly relaxed if one is only interested in $L^2$ convergence in \eqref{eq:solitonres}. The boundedness and decay conditions on $u_0$ are crucial. The specific assumptions on the initial data in the statement of the theorem are designed so that these conditions are preserved by the Benjamin-Ono dynamics.

In fact, it is possible to describe precisely the soliton parameters $\gamma_1,\dots,\gamma_M$ and the radiation profiles $u_\infty^\pm$ from the spectral data of the Lax operator $\mathcal{L}_{u_0}$ defined in~\eqref{eq:Lax-operator}. To this end, we denote by $\lambda_1<\cdots< \lambda_M<0$ the eigenvalues of the Lax operator $\mathcal{L}_{u_0}$ on $L^2_+(\R)$, and by $\phe_1,\dots, \phe_M$ the corresponding eigenfunctions, normalized in $L^2$. The fact that under the hypotheses on $u_0$ in force the eigenvalues are simple and negative and that their number $M$ is finite has been established in \cite[Theorem 1.2]{Wu16a}.
For every $\lambda>0$, we also denote by $m_\pm=m_\pm(x;\lambda)$ the unique solution~\cite[Theorem 26]{BlackstoneGGM24a} in the space $L^\infty(\R)$  of the equation
\begin{equation}\label{def:m}
 (\mathcal{L}_{u_0}-\lambda\,\mathrm{Id})m_\pm=0\ ,\quad \lim_{x\to \pm\infty}{\ee^{-\ii\lambda x}m_\pm(x;\lambda)}=1.
 \end{equation}
For every $f\in L^2_+(\R)$, two \emph{distorted Fourier transforms}  $\widetilde f^\pm $ of $f$ are defined as the extension to $L_+^2(\mathbb{R})$ from $L_+^2(\mathbb{R})\cap L_+^1(\mathbb{R})$ of
\begin{equation}\label{distortedFT}
 \widetilde f^{\pm} (\lambda)=\int_\R f(x) m_\pm(x;\lambda)^*\,\dd x=\langle f,m_\pm(\diamond;\lambda)\rangle,
 \quad \lambda >0.
 \end{equation}
The following distorted Plancherel identity was proven in~\cite[Theorem 26]{BlackstoneGGM24a}:
\begin{equation}\label{Planchereldist}
\sum_{j=1}^M \vert \langle f,\phe_j\rangle \vert ^2+\int_0^\infty \vert \widetilde f^\pm(\lambda)\vert ^2 \frac{\dd\lambda}{2\pi} =\int_\R |f(x)|^2\,\dd x,\quad f\in L_+^2(\R).
\end{equation}
Recall now the notation $\mathcal{X}^*$ and $I_+[\diamond]$ from Section~\ref{sec:Xstar-and-Iplus}.
\begin{theorem}[Soliton resolution and spectral data~\cite{GassotGM26}]\label{thm:main-spectral}
The parameters in Theorem~\ref{thm:solitonres} are given as follows.
\begin{itemize}
    \item For every $j\in \{1,\dots, M\}$, we have
$\gamma_j=-\langle \mathcal{X}^*\phe_j,\phe_j\rangle$.
\item The real-valued functions $u_\infty^\pm \in L^2(\R)$ satisfy
$ \widehat {u_\infty^\pm} (\lambda )=\widetilde{\Pi u_0}^\mp(\lambda)$ for every $\lambda>0$.
\end{itemize}
\end{theorem}
\begin{remark}[The functions $\lambda\mapsto \widetilde{\Pi u_0}^\mp(\lambda)$ and the Fokas-Ablowitz theory]
\label{rem:beta-of-lambda}
    In the Fokas-Ablowitz inverse-scattering formalism for the Benjamin-Ono equation \cite{FokasA83} (see also Kaup-Matsuno \cite{KaupM98} and Wu \cite{Wu17}) a \emph{reflection coefficient} $\beta(\lambda)$ is defined for $\lambda>0$, and one can check that in our notation, $\beta(\lambda)=\ii\widetilde{\Pi u_0}^+(\lambda)$.  The other distorted Fourier transform of $\Pi u_0$, $\alpha(\lambda):=\widetilde{\Pi u_0}^-(\lambda)$ has the same modulus as does $\beta(\lambda)$:  $|\alpha(\lambda)|=|\beta(\lambda)|$, for all $\lambda>0$.  See \cite[Remark 15]{BlackstoneGGM24a} for further information.
\end{remark}
Note that Wu's identities~\cite[Lemma 2.5 and (2.23)]{Wu16a} (see also \cite{BadreddineKV26orbital} for $u$ less  regular) imply 
\begin{equation}\label{eq:Wu} |\langle \phe_j, \Pi u_0\rangle |^2=-2\pi \lambda_j\ ,\ \lambda_j I_+[\phe_j]=-\langle \phe_j,\Pi u_0\rangle,\quad j=1,\dots,M.
\end{equation}
Consequently, expressing $\langle \mathcal{X}^*\phe_j,\phe_j\rangle$ on the Fourier side and integrating by parts, the reciprocals of the soliton velocities satisfy
\begin{equation}\label{Imp}
\mathrm{Im}(\gamma_j)=\frac{\vert I_+[\phe_j]\vert ^2}{4\pi}=-\frac{1}{2\lambda_j}>0,\quad j=1,\dots,M.
\end{equation}
Because the imaginary parts of each complex phase constant $\gamma_j$ is determined from the corresponding Lax eigenvalue, the notation $\gamma_j$ is used in some other papers to denote the independent real-valued quantity written here as $\mathrm{Re}(\gamma_j)$.

The proof of Theorem~\ref{thm:solitonres} is based on two weak convergence results in the limit $t\to+\infty$ (the case of $t\to -\infty$ can be deduced from the opposite limit using the fact that $u(-t,-x)$ solves the Benjamin-Ono equation \eqref{eq:BO} whenever $u(t,x)$ does).
For the soliton part, we show that for every $j\in \{1,\dots, M\}$ and for every $z\in \C_+$, we have
\begin{equation}\label{soliton} \Pi u(t,z-2t\lambda_j)\longrightarrow \frac{\ii}{z-\langle \mathcal{X}^*\phe_j,\phe_j\rangle },\quad t\to+\infty. 
\end{equation}
For the radiation part, we show that for every $z\in \C_+$ with $|z|$ small enough, we have
\begin{equation}\label{scat} (2t)^{1/2}\mathrm{e}^{\ii t\lambda^2}\Pi u(t,z-2t\lambda)
	\longrightarrow \frac{\mathrm{e}^{\ii\pi/4}}{\sqrt{2\pi}}\mathrm{e}^{\ii\lambda z}\widetilde{\Pi u_0}^-({\lambda}), \quad t\to+\infty
\end{equation}
in the sense of weak convergence in $L^2(0,+\infty )$ in the $\lambda$ variable.
Strong convergence of the remainder term is then obtained thanks to an orthogonality property of the objects at hand as $t\to+\infty$ (they travel at different speeds), combined with the distorted Plancherel identity~\eqref{Planchereldist} applied to $\Pi u_0$.  Then the stationary-phase convergence \eqref{eq:stationary-phase} below can be applied to the limiting object in \eqref{scat} yielding the term $\ee^{t\partial_x|D|}u^+_\infty$ in \eqref{eq:solitonres}. Indeed, taking $z\to x_0\in\R$ and replacing $\lambda$ with $(x_0-x)/(2t)$ in \eqref{scat} with $\widehat{\Pi u^+_\infty}=\widetilde{\Pi u_0}^-$ gives the approximation
\begin{equation}
    \Pi u(t,x)=\frac{\ee^{\ii\pi/4}}{\sqrt{4\pi t}}\ee^{-\ii x_0^2/(4t)}\ee^{\ii x_0^2/(4t)}\widehat{\Pi u_\infty^+}\left(\frac{x_0-x}{2t}\right) + o(t^{-1/2}),\quad t\to+\infty.
\end{equation}

A generalization of the soliton resolution result to the case of an infinite-soliton ($M=\infty$) was studied in~\cite{GassotG26}. A related result of asymptotic stability of multisolitons under low-regularity perturbations was proven in~\cite{BadreddineKV26asymptotic}. 

\subsection{Pointwise asymptotics for rational initial data in the limit \texorpdfstring{$t\to\pm\infty$}{t to infinity}}
\subsubsection{Large \texorpdfstring{$t$}{t} with \texorpdfstring{$|x|>\delta|t|$}{|x|>delta |t|}}
In the case of rational initial data, it is possible to establish pointwise and locally uniform decay rates on the expression tending to zero as $t\to\pm\infty$ in \eqref{eq:solitonres} in Theorem~\ref{thm:solitonres}. The analysis is based on the simplified explicit formula for rational initial data~\cite{BlackstoneGGM24a} reproduced in Theorem~\ref{thm:inversion-formula-rational}. In particular, we have the following new result.

\begin{theorem}[Pointwise soliton resolution for rational initial data]\label{thm:soliton-rational} Let $u_0\in L^2(\R)$ with $\diamond\, u_0(\diamond)\in L^2(\R)$ be a rational initial condition with simple poles, of the form \eqref{eq:PN-rational-IC}, and let $u(t,x)$ denote the solution of the Cauchy problem \eqref{eq:BO}--\eqref{eq:IC} for $\varepsilon=1$. Let $\gamma_j$, $1\leq j\leq M$ and $u_\infty^\pm\in L^2(\R)$ be as in Theorem~\ref{thm:main-spectral}. Then for every $x_0\in\R$, we have as $t\to+\infty$
\begin{equation}\label{eq:soliton-rational}
    u(t,x_0+v_{\gamma_j} t)=S_{\gamma_j}(x_0)+O(t^{-1}),
    \quad j=1,\dots,N,
\end{equation}
and for every $v>0$ such that $v\not\in\{v_{\gamma_1},\dots,v_{\gamma_M}\}$,
\begin{equation}\label{eq:soliton-rational2}
    u(t,x_0+v t)=O(t^{-1}).
\end{equation}
Moreover, for every $\lambda>0$,  and $x_0\in\R$, we have 
\begin{equation}\label{eq:radiation-rational}
    \Pi u(t,x_0-2\lambda t)
    =\ee^{\ii\pi/4-\ii t \lambda^2+\ii\lambda x_0}\frac{\widetilde{\Pi u_0}^-(\lambda)}{\sqrt{4\pi t}}
    +O(t^{-1}).
\end{equation}
The $O(t^{-1})$ error terms are uniform for $x_0$ in compact subsets of $\R$.
\end{theorem}

Notice that, for initial data of the form \eqref{eq:PN-rational-IC},
the condition $\diamond u_0(\diamond)\in L^2(\R)$ is equivalent to 
\begin{equation}
    \mathrm{Re}(c_1+\cdots+c_N)=0,
    \label{eq:sum-of-residues-zero}
\end{equation}
which guarantees in particular that $u_0(x)=O(x^{-2})$ as $x\to\pm\infty$. 

Theorem \ref{thm:soliton-rational} implies in particular weak convergence in $L^2$ of the left hand sides of \eqref{eq:soliton-rational} and \eqref{eq:radiation-rational}, which again can be  converted into strong convergence using the same arguments as Theorem \ref{thm:solitonres}. On the other hand, the proof of Theorem \ref{thm:soliton-rational} is quite different in character from that of Theorem~\ref{thm:solitonres} in \cite{GassotGM26}, being based on the asymptotic expansion for large $t>0$ of contour integrals appearing in the representation of $u(t,x)$ given by Theorem~\ref{thm:inversion-formula-rational}.  We prove \eqref{eq:soliton-rational}--\eqref{eq:soliton-rational2} in Appendix~\ref{app:soliton-limit}, and we prove \eqref{eq:radiation-rational} in Appendix~\ref{app:radiation}.

Note that for every $f\in L^2(\R)$, we can write 
\begin{equation}
\ee^{-\ii t\partial_x^2}f(x)-\ee^{-\ii x^2/4t}\frac{\ee^{ \ii \pi/4}}{\sqrt{4\pi t}}\widehat f\left (\frac{-x}{2t}\right )
    =o(1),\quad t\to+\infty
\label{eq:stationary-phase}
\end{equation}
in $L^2(\R)$ and, under additional mild hypotheses, also uniformly for $x=x_0-2\lambda t$ where $x_0$ lies in a compact set. Note that for $f\in L_+^2(\R)$, we have $\ee^{-\ii t\partial_x^2}f=\ee^{t\partial_x|D|}f$.  One could make the estimate more quantitative by specifying a rate of decay for the error term, under stronger decay assumptions on $f$, or equivalently, regularity assumptions on $\widehat{f}$. Hence the right-hand side of~\eqref{eq:radiation-rational} should be compared to the radiation term $\ee^{t \partial_x|D|}u_\infty^+(x_0-2\lambda t)$.
However, the link between the behavior of the generalized eigenfunctions $m_\pm (\diamond ;\lambda)$ given by \eqref{def:m} as $\lambda \to 0$ and this type of estimate is still unclear to us.

\subsubsection{Universality revisited:  large \texorpdfstring{$t$}{t} with \texorpdfstring{$x=O(\sqrt{t})$}{x=O(sqrt(t))}}
Recall from Section~\ref{sec:rational} the description of the contours $C_0,\dots,C_N$ for rational initial data $u_0$ of the form \eqref{eq:PN-rational-IC}, as well as the notions of exceptional and nonexceptional indices.  Without loss of generality (by reordering the indices), we assume that if there exists any nonexceptional index, then $N$ is nonexceptional.  Moreover, we assume the additional condition \eqref{eq:sum-of-residues-zero}.
which implies the existence of an antiderivative of $u_0$ considered as a function of the complex variable $z$ and denoted $L(z)$, that is analytic in a neighborhood of $z=\infty$ and satisfies $L(\infty)=0$.  We define a quantity $E_n$ for $n=1,\dots,N$ as follows:
\begin{equation}
    E_n:=\begin{cases}\ee^{2\pi(c_1+\cdots+c_n)},& \text{$n$ nonexceptional}\\
    0,&\text{$n$ exceptional}.
    \end{cases}
\label{eq:En-def}
\end{equation}
Then, we split each of the contours $C_1,\dots,C_N$ at an arbitrary break point so that $C_j=C_j^-\sqcup C_j^+$ for $j=1,\dots,N$, where $C_j^-$ is the initial arc originating from $z=\infty$, while $C_j^+$ terminates at $z=\infty$ if $j$ is a nonexceptional index and at $p_j$ in the exceptional case.  A matrix $\mathbf{J}$ of dimension $N\times N$ and a vector $\mathbf{d}\in\C^N$ are then defined as follows:
\begin{equation}
    J_{mn}:=\int_{C_m^-}\frac{\ee^{-\ii L(z)}-1}{z-p_n}\,\dd z +\int_{C_m^+}\frac{\ee^{-\ii L(z)}-E_m}{z-p_n}\,\dd z +2\pi\ii E_m\delta_{m\ge n},\quad 1\le m,n\le N
\end{equation}
and
\begin{equation}
    \mathbf{d}=(1-E_1,1-E_2,\dots,1-E_N)^\top.
    \label{def:d-vector}
\end{equation}
Then, let 
\begin{equation}
\Delta:=\sum_{n=1}^N\det(\mathbf{J}\mathop{\longleftarrow}^n\mathbf{d})
\end{equation}
where the $n$th summand denotes the determinant of $\mathbf{J}$ replacing column $n$ with $\mathbf{d}$.  It is easy to check that $\Delta$ does not depend on the arbitrary choice of break point for each contour $C_j$.  It is proved in \cite[Theorem 1.1]{GassotGM26b} that rational initial data $u_0$ for which the condition $\Delta\neq 0$ holds are dense in $L^2$ and are \emph{generic} potentials in the sense defined in \cite{Wu17} (i.e., their reflection coefficients $\beta(\lambda)$ exhibit generic behavior as $\lambda\downarrow 0$).  With these notions, we have the following.
\begin{theorem}[Linearized self-similar universality, \protect{\cite[Theorems 1.3, 1.4]{GassotGM26b}}]
Let $u_0$ be a rational initial condition of the form \eqref{eq:PN-rational-IC} for which also the condition \eqref{eq:sum-of-residues-zero} holds and for which $\Delta\neq 0$.  If $u(t,x)$ denotes the corresponding solution of the Cauchy problem \eqref{eq:BO}--\eqref{eq:IC} with $\varepsilon=1$, then
\begin{equation}
    \lim_{t\to+\infty}t^{1/2}\ln(t)u(t,2t^{1/2}\xi)=U(\xi)
\label{eq:LSSU-limit}
\end{equation}
holds in the $L^\infty_\mathrm{loc}(\R)$ sense with respect to $\xi$, where $U(\xi)$ is the unique solution of 
\begin{equation}
    -\xi U(\xi)=\frac{1}{2}|D|U(\xi)
    \label{eq:LSSU-U-eqn}
\end{equation}
given explicitly by the convergent improper Fresnel-type integral
\begin{equation}
    U(\xi):=-4\int_\xi^{+\infty}\cos(w^2-\xi^2)\,\dd w.
\end{equation}
\label{thm:linearized-self-similar-universality}
\end{theorem}
This result asserts \emph{universality} of the large-$t$ asymptotic behavior because the limiting profile $U(\xi)$ is independent of $u_0$ within the class of data considered.  The terminology of \emph{linearized self-similarity} refers to the equation \eqref{eq:LSSU-U-eqn} satisfied by $U(\xi)$; this equation can be regarded as the linearization of another equation:
\begin{equation}
    -\xi U(\xi)+U(\xi)^2=\frac{1}{2}|D|U(\xi),
\label{eq:SSU-U-eqn}
\end{equation}
which characterizes the self-similar solutions of the Benjamin-Ono equation \eqref{eq:BO} with $\varepsilon=1$, i.e., exact solutions of the form $u(t,x)=t^{-1/2}U(\xi)$, $\xi=x/(2t^{1/2})$.  Observe that the existence of the limit in \eqref{eq:LSSU-limit} implies that $u(t,x)$ decays faster by a factor of $1/\ln(t)$ than the exact self-similar solutions.  This provides some heuristic for why the profile $U(\xi)$ might satisfy the linearized form of the exact self-similar profile equation \eqref{eq:SSU-U-eqn}.

For the Korteweg-de Vries equation, there is a similar notion of generic initial data, for which one again observes universal long-time asymptotics \cite{AblowitzS77}, in this case in the region $x=O(t^{1/3})$ and with a profile related to the Hastings-McLeod solution of the Painlev\'e-II equation.  However, in this situation, the rate of decay of the solution is commensurate with (nonlinear) self-similar solutions. 

\begin{remark}[Loss of universality and a Benjamin-Ono soliton gas]
\label{rem:soliton-gas}
    The condition \eqref{eq:sum-of-residues-zero} is essential for Theorem~\ref{thm:linearized-self-similar-universality}.  Indeed, for $u_0(x)=2(ax-b)/(x^2+1)$ with $a,b\in\R$ and $a\neq 0$, a different result entirely holds:
    letting
    \begin{equation}
    U(\xi;a):= -\frac{\dd}{\dd\xi}\mathrm{Im}\left(\log(\tau(\xi;a))\right),\quad \tau(\xi;a):=\int_{D}\ee^{-\ii (\zeta-\xi)^2}\zeta^{-2\ii a}\,\dd\zeta,
\end{equation}
where $D$ is the contour parametrized by $\zeta=1+\ee^{-\ii\pi/4}w$, $-\infty<w<+\infty$, we have
        \begin{equation}    \lim_{t\to+\infty}t^{1/2}u(t,2t^{1/2}\xi)=U(\xi;a),
        \label{eq:non-universal-limit}
\end{equation}
locally uniformly with respect to $\xi$.
This solution decays at the slower, self-similar rate of $t^{-1/2}$, and the limit in \eqref{eq:non-universal-limit} is not universal, depending on the real value of $a\in\R$.  In particular, when $a<0$ the limiting profile can exhibit very large peaks.  See Figure~\ref{fig:non-universal}.
\begin{figure}[h]
\begin{center}
    \includegraphics[width=0.3\linewidth]{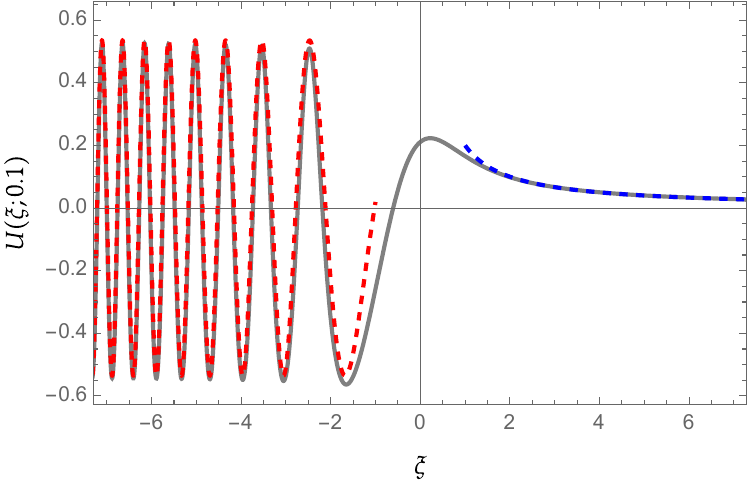}\hfill%
    \includegraphics[width=0.3\linewidth]{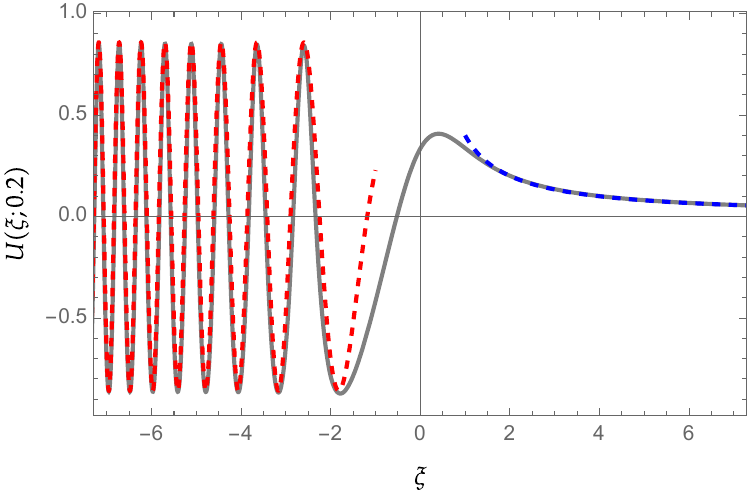}\hfill%
    \includegraphics[width=0.3\linewidth]{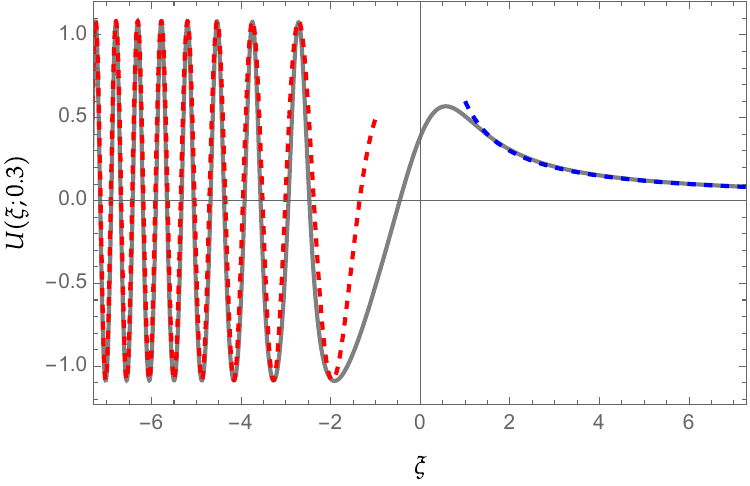}\\
    \includegraphics[width=0.3\linewidth]{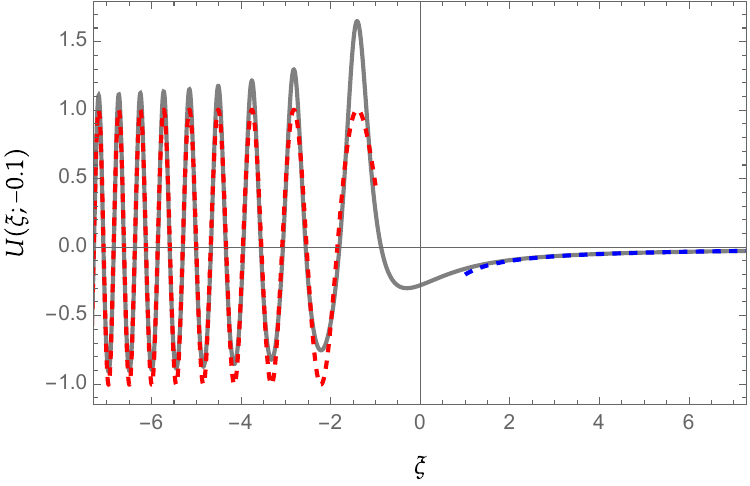}\hfill%
    \includegraphics[width=0.3\linewidth]{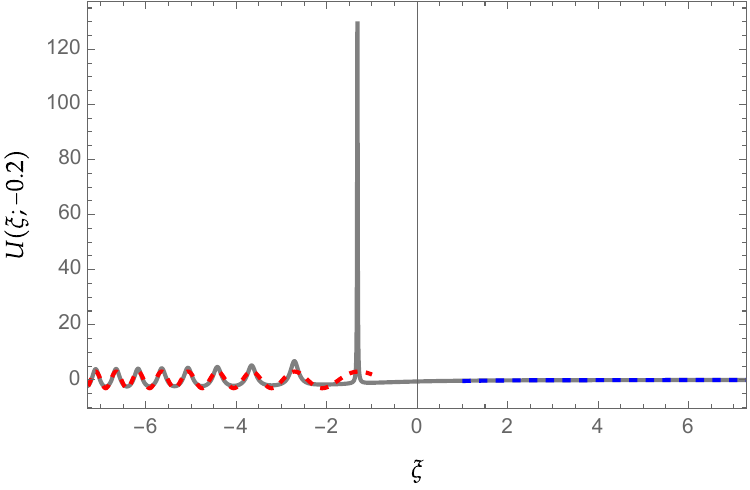}\hfill%
    \includegraphics[width=0.3\linewidth]{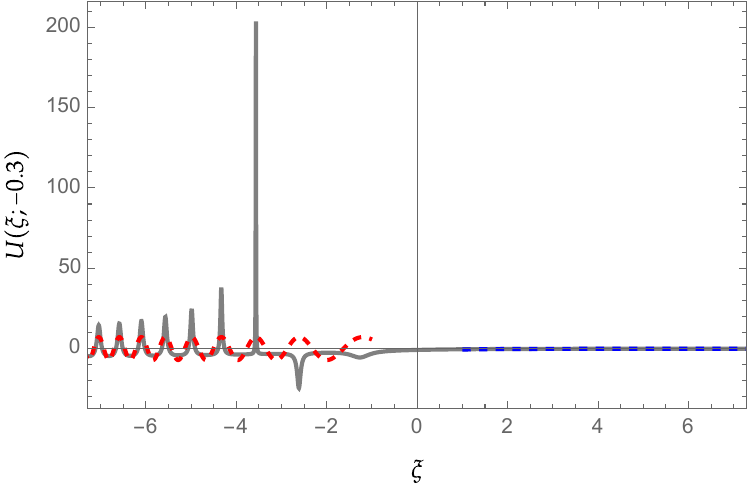}    
\end{center}
\caption{The function $U(\xi;a)$ from Remark~\ref{rem:soliton-gas} (gray) along with asymptotes for large negative (dashed red) and positive (dashed blue) $\xi$.  Top row, left-to-right:  $a=0.1$, $a=0.2$, $a=0.3$.  Bottom row, left-to-right:  $a=-0.1$, $a=-0.2$, $a=-0.3$.}
\label{fig:non-universal}
\end{figure}
Also note that $U(\xi;0)$ vanishes identically, which is consistent with the faster rate of decay in Theorem~\ref{thm:linearized-self-similar-universality}, which indeed applies for generic $b\in\R$ if $a=0$.
On the spectral side, a distinguishing characteristic of this initial condition is that for $a\neq 0$ there are infinitely many Lax eigenvalues with $\lambda=0$ as an accumulation point.  Thus the corresponding solution can be viewed as a deterministic soliton gas composed predominantly of small-amplitude slowly moving solitons.
These preliminary results will be presented in more detail in a subsequent publication.
\end{remark}

\appendix
\section{Obtaining multisoliton solutions from Theorem~\ref{thm:inversion-formula-rational}}
\label{app:Nsoliton}
In this Appendix, we give the details of the derivation of Matsuno's formula \eqref{eq:Matsuno}--\eqref{eq:Matsuno-Matrix} from Theorem~\ref{thm:inversion-formula-rational} applied to the $N$-soliton initial condition \eqref{eq:rational-soliton} as discussed in Section~\ref{sec:examples}.

To write the formula \eqref{eq:Pi-u-derivative} in the situation that $c_n=-\ii$ for all $n=1,\dots,N$, first observe that the exponential factor $\ee^{-\ii h(z;t,x)/\varepsilon}$ defined in \eqref{eq:PN-exph} is given by
\begin{equation}
\ee^{-\ii h(z;t,x)/\varepsilon}
	=\ee^{-\ii (z-x)^2/(4 \varepsilon t)}R(z),
\quad
R(z):=\prod_{n=1}^N\frac{z-p_n^*}{z-p_n},
\label{eq:meromorphic-exp-factor}
\end{equation}
which is single-valued meromorphic, with simple poles (respectively, zeros) at $z=p_n$ (respectively at $z=p_n^*$), $n=1,\dots,N$.  Moreover, none of the indices are exceptional.  Therefore, by row operations, $\det(\mathbf{B}(t,x))=\det(\mathbf{C}(t,x))$ where $C_{jk}(t,x)=B_{jk}(t,x)$ for rows $j=1,2$ while $C_{jk}(t,x)=B_{jk}(t,x)-B_{j-1,k}(t,x)$ for $3\le j\le N+1$. The elements of $\mathbf{C}(t,x)$ have the same form as those of $\mathbf{B}(t,x)$, except that each integration contour $C_n$ for $n=1,\dots,N$ is replaced with $L_n$, a positively-oriented loop enclosing only one pole, $z=p_n$.  
Equation~\eqref{eq:Pi-u-derivative} then becomes
\begin{equation}
\Pi u(t,x)=\ii\varepsilon\frac{\partial}{\partial x}\log\left(\det(\mathbf{C}(t,x))\prod_{n=1}^N\ee^{\ii (p_n-x)^2/(4 \varepsilon t)}\right).
    \label{eq:Pi-u-derivative2}
\end{equation}
The rational function defined as 
\begin{equation}
R_n(z):=(z-p_n)R(z),\quad n=1,\dots,N,
\label{eq:Rn-define}
\end{equation}
has a removable singularity at $z=p_n$, and $R_n(p_n)\neq 0$ because $\mathrm{Im}(p_n)>0$ and $R_n(z)$ vanishes only at the points $p_1^*,\dots,p_N^*$. With the help of this function, we can write the partial fraction expansions
\begin{equation}
R(z)=1+\sum_{n=1}^N\frac{R_n(p_n)}{z-p_n},
	\label{eq:R}
\end{equation}
\begin{equation}
\frac{R(z)}{z-p_k}=\frac{R_k(p_k)}{(z-p_k)^2}+\frac{R_k'(p_k)}{z-p_k}+\mathop{\sum_{n=1}^N}_{n\neq k}\frac{R_n(p_n)}{(p_n-p_k)(z-p_n)}.
\label{eq:Rk}
\end{equation}
We will next use these expansions to compute $\det(\mathbf{C}(t,x))$.

\subsection{Simplification of rows $2,\dots,N+1$ of $\mathbf{C}(t,x)$}
  Since $\ee^{-\ii h(z;t,x)/\varepsilon}$ is meromorphic, all of the integrals in rows $j=2,\dots,N+1$ of the matrix $\mathbf{C}(t,x)$ can be evaluated explicitly by residues: for $j=1,\dots,N$, 
  \begin{equation}
    C_{j+1,1}(t,x)
    =\int_{L_{j}}\ee^{-\ii (z-x)^2/(4 \varepsilon t)}R(z)\,\dd z
    =2\pi\ii\mathop{\mathrm{Res}}_{z=p_{j}}\left[\ee^{-\ii (z-x)^2/(4 \varepsilon t)}R(z)\right],
\label{eq:C-first-column-residues}
\end{equation}
and, for $j,k=1,\dots,N$,
\begin{equation}
        C_{j+1,k+1}(t,x)=\int_{L_{j}}\frac{\ee^{-\ii (z-x)^2/(4 \varepsilon t)}}{z-p_{k}}R(z)\,\dd z
    =2\pi\ii\mathop{\mathrm{Res}}_{z=p_{j}}\left[\frac{\ee^{-\ii (z-x)^2/(4 \varepsilon t)}}{z-p_{k}}R(z)\right].
    \label{eq:C-other-columns-residues}
\end{equation}
Using the partial fractions expansions \eqref{eq:R}--\eqref{eq:Rk} gives
\begin{equation}
    C_{j+1,1}(t,x)
    =2\pi\ii\ee^{-\ii (p_{j}-x)^2/(4 \varepsilon t)}R_j(p_j),
    \quad j=1,\dots,N,
\label{eq:C-first-column-residues2}
\end{equation}
and for $j,k=1,\dots,N$,
\begin{equation}
        C_{j+1,k+1}(t,x)=
\begin{cases}\displaystyle 
        2\pi\ii \ee^{-\ii (p_j-x)^2/(4 \varepsilon t)}\frac{R_j(p_j)}{p_j-p_{k}}, & j\neq k,
\\ \displaystyle 
	2\pi \ii \ee^{-\ii (p_j-x)^2/(4 \varepsilon t)} \left(R_j'(p_j)-\frac{\ii}{2\varepsilon t}(p_j-x)R_j(p_j)\right), & j=k.
\end{cases}
    \label{eq:C-other-columns-residues2}
\end{equation}
Row $j+1$ of $\mathbf{C}$ can therefore be written as
\begin{multline}
\begin{bmatrix}
C_{j+1,1}(t,x) & C_{j+1,2}(t,x) & \dots & C_{j+1,N}(t,x)\end{bmatrix}
\\
	=\frac{\pi}{\varepsilon t} \ee^{-\ii (p_j-x)^2/(4 \varepsilon t)}R_j(p_j)
	\begin{bmatrix}
2\ii\varepsilon t &  (2t \mathbf{M}-x\mathbb{I}+\mathbf{D})_{j1} & \dots  &(2 t\mathbf{M}-x\mathbb{I}+\mathbf{D})_{jN}
\end{bmatrix}
\label{eq:C-other-rows}
\end{multline}
where $\mathbb{I}$ is the $N\times N$ identity matrix,  $\mathbf{D}$ is the diagonal matrix
\begin{equation}
\mathbf{D}:=\mathrm{diag}(p_1,\dots,p_N),
\label{eq:D-define}
\end{equation}
and the matrix $\mathbf{M}$ is defined as follows: for $j,k=1,\dots,N$,
\begin{equation}
M_{jk}=
\begin{cases}
\displaystyle 
\frac{\ii\varepsilon }{p_j-p_k}, &j\neq k,
\\
\displaystyle 
\ii\varepsilon \frac{R_j'(p_j)}{R_j(p_j)}, &j=k.
\end{cases}
\label{def:M}
\end{equation}

\subsection{Simplification of row $1$ of $\mathbf{C}(t,x)$}
Now recall that the elements of the first row of $\mathbf{C}(t,x)$ are contour integrals over $C_0$: 
\begin{equation}
    C_{11}(t,x)
    =\int_{C_0}\ee^{-\ii (z-x)^2/(4\varepsilon t)}R(z)\,\dd z,
\quad
     C_{1,k+1}(t,x)=\int_{C_0}\frac{\ee^{-\ii (z-x)^2/(4 \varepsilon t)}}{z-p_{k}}R(z)\,\dd z,\; k=1,\dots,N.
\end{equation}
These cannot be evaluated by residues.  However, the partial fraction expansions \eqref{eq:R}--\eqref{eq:Rk} show that these elements are linear combinations of the $x$-independent Gaussian integral
\begin{equation}
I(t)=\int_{C_0}\ee^{-\ii(z-x)^2/(4\varepsilon t)}\dd z
=2\ee^{-\pi\ii/4}\sqrt{\pi\varepsilon  t},
\end{equation}
the integrals 
\begin{equation}
J_j(t,x):=\int_{C_0}\ee^{-\ii(z-x)^2/(4\varepsilon t)}\frac{\dd z}{z-p_j},\quad j=1,\dots,N,
\end{equation}
and $N$ additional integrals 
\begin{equation}
\begin{aligned}
\int_{C_0}\ee^{-\ii(z-x)^2/(4\varepsilon t)}\frac{\dd z}{(z-p_j)^{2}}
	&=-\frac{\ii}{2\varepsilon t}\int_{C_0}(z-x)\ee^{-\ii(z-x)^2/(4\varepsilon t)}\frac{\dd z}{z-p_j}\\
	&=-\frac{\ii}{2\varepsilon t}I(t)
	-\frac{\ii(p_j-x)}{2\varepsilon t}J_j(t,x),\quad j=1,\dots,N,
\end{aligned}
\end{equation}
using integration by parts and writing $z-x=(z-p_j)+(p_j-x)$.
Therefore, from the expansion \eqref{eq:R} of $R(z)$,
\begin{equation}
\begin{aligned}
C_{11}(t,x)
	&=I(t)+\sum_{n=1}^N R_n(p_n) J_{n}(t,x)\\
	&=I(t)+\frac{1}{2\pi\ii}\sum_{n=1}^N J_{n}(t,x) \ee^{\ii(p_n-x)^2/(4\varepsilon t)} C_{n+1,1}(t,x),
\end{aligned}
\end{equation}
where we used also \eqref{eq:C-first-column-residues2},
and similarly combining the expansion \eqref{eq:Rk} of $R(z)/(z-p_k)$ with \eqref{eq:C-other-columns-residues2}, for $k=1,\dots,N$,
\begin{equation}
    \begin{aligned}
C_{1,k+1}(t,x)
	=&\frac{1}{2\ii\varepsilon t}R_k(p_k)I(t)
	+\left(R_k'(p_k)-\frac{\ii}{2\varepsilon t}(p_k-x)R_k(p_k)\right)J_{k}(t,x)
+\mathop{\sum_{n=1}^N}_{n\neq k}\frac{R_n(p_n)}{p_n-p_k}J_{n}(t,x)
	\\
	=&\frac{1}{2\ii\varepsilon t}R_k(p_k)I(t)
	+\frac{1}{2\pi\ii}\sum_{n=1}^NJ_{n}(t,x)\ee^{\ii(p_n-x)^2/(4\varepsilon t)} C_{n+1,k+1}(t,x).
\end{aligned}
\end{equation}
Therefore, we see that 
\begin{multline}
    \begin{bmatrix}C_{11}(t,x) & C_{12}(t,x) & \cdots & C_{1,N+1}(t,x)\end{bmatrix}=\frac{ I(t)}{2\ii \varepsilon t}\begin{bmatrix}2\ii\varepsilon t & R_1(p_1) &\cdots & R_N(p_N)\end{bmatrix} \\
    {}+\frac{1}{2\pi\ii}\sum_{n=1}^NJ_n(t,x)\ee^{\ii (p_n-x)^2/(4\varepsilon t)}\begin{bmatrix}C_{n+1,1}(t,x) & C_{n+1,2}(t,x) & \cdots & C_{n+1,N+1}\end{bmatrix}, 
    \label{eq:C-first-row}
\end{multline}
And clearly the terms on the second line make no contribution to $\det(\mathbf{C}(t,x))$.

\subsection{Simplification of $\det(\mathbf{C}(t,x))$}
Now we combine \eqref{eq:C-other-rows} with \eqref{eq:C-first-row} to calculate $\det(\mathbf{C}(t,x))$.
Let us denote 
\begin{equation}
\mathbf{r}:=[R_1(p_1),\dots,R_N(p_N)]^{\top}\quad\text{and}\quad \mathbf{1}:=[1,\dots,1]^{\top}.
\label{eq:r1-define}
\end{equation}
By multilinearity with respect to the rows and the first column, we have that
\begin{equation}
\det(\mathbf{C}(t,x))=\left(\frac{\pi}{\varepsilon t}\right)^N I(t)\left(\prod_{n=1}^N \ee^{-\ii (p_n-x)^2/(4 t)}R_n(p_n)\right) 
    \begin{vmatrix} 1 & \mathbf{r}^{\top}\\
    \mathbf{1} & 2t\mathbf{M}-x\mathbb{I}+\mathbf{D}
    \end{vmatrix}.
\end{equation}
Furthermore, Schur's formula implies
\begin{equation}
\begin{vmatrix} 1 & \mathbf{r}^{\top}\\
    \mathbf{1} & 2t\mathbf{M}-x\mathbb{I}+\mathbf{D}
    \end{vmatrix}
    	=\det(2t\mathbf{M}-x\mathbb{I}+\mathbf{D}-\mathbf{1}\mathbf{r}^{\top}).
\end{equation}
Therefore, returning to \eqref{eq:Pi-u-derivative2}, since the factor $(\pi/(\varepsilon t))^NI(t)R_1(p_1)\dots R_N(p_N)$ is independent of $x$, we get:
\begin{equation}
\Pi u(t,x)=\ii\varepsilon \frac{\partial}{\partial x}\log\det(2t\mathbf{M}-x\mathbb{I}+\mathbf{D}-\mathbf{1}\mathbf{r}^{\top}).
\label{eq:Pi-u-soliton-pre-Matsuno}
\end{equation}

This formula is reminiscent of Matsuno's $N$-soliton formula \cite{Matsuno84} (see \eqref{eq:Matsuno}--\eqref{eq:Matsuno-Matrix}).  But the matrices $\mathbf{M}$ and $\mathbf{D}-\mathbf{1r}^\top$ are constructed from the ``direct data'' $p_1,\dots,p_N\in\C_+$, whereas the corresponding matrices in Matsuno's formula are instead constructed from Lax eigenvalues and phase constants for $u_0$.    It remains to connect these two representations.

\subsection{Link with Matsuno's $N$-soliton formula}
To show that \eqref{eq:Pi-u-soliton-pre-Matsuno} in fact coincides with Matsuno's formula, we will make a link between $\mathbf{M}$ and the action of the Lax operator $\mathcal{L}_{u_0}$ on an adequate $N$-dimensional space, and between $\mathbf{D}-\mathbf{1}\mathbf{r}^{\top}$ and the action of the operator $\mathcal{X}^*$ on the same space.

First, we identify the eigenfunctions of the Lax operator $\mathcal{L}_{u_0}$ associated to $u_0=\Pi u_0 + (\Pi u_0)^*$ with $\Pi u_0$ given by \eqref{eq:rational-soliton}.  Recall the definition of $R_k(z)$ in \eqref{eq:meromorphic-exp-factor} and \eqref{eq:Rn-define}.
\begin{lemma}
Every eigenfunction $\varphi\in L_+^2(\R)$ of $\mathcal{L}_{u_0}$ belongs to $F:=\mathrm{span}\{f_k\}_{k=1}^N$, where 
\begin{equation}
f_k(z):=\frac{R_k(p_k)}{R_k(z)},
\quad z\in\C_+.
\label{eq:fk-define}
\end{equation}
Conversely, $F$ is spanned by eigenfunctions $\varphi$ of $\mathcal{L}_{u_0}$.
\end{lemma}
Note that the functions $f_k$ are linearly independent, and hence they form a basis of the $N$-dimensional space $F$.
\begin{proof}
Let $\varphi\in L_+^2(\R)$ be an eigenfunction of $\mathcal{L}_{u_0}$ associated to the eigenvalue $\lambda<0$. Then using the identity 
\begin{equation}
\Pi \left(\frac{\varphi}{\diamond-p}\right)(z)=\frac{\varphi(z)-\varphi(p)}{z-p},\quad p\in \C_+,    
\end{equation}
in $\mathcal{L}_{u_0}\varphi=\lambda\varphi$ leads to
\begin{equation}
\varepsilon \varphi'(z)-\ii(\lambda+u_0)\varphi(z)=-\varepsilon\sum_{n=1}^N\frac{\ii v_n}{z-p_n},
\quad v_n:=\varphi(p_n).
\end{equation}
Then using the integrating factor $\ee^{-\ii\lambda z/\varepsilon}R(z)$ gives
\begin{equation}\label{eq:phi-multisol}
(\ee^{-\ii\lambda z/\varepsilon}R(z)\varphi(z))'=-\ee^{-\ii\lambda z/\varepsilon}R(z)\sum_{n=1}^N\frac{v_n}{z-p_n}.
\end{equation}
Since the left-hand side is the derivative of a function meromorphic in $\C_+$, the residues of the right-hand side at all poles $p_k$, $k=1,\dots,N$, must vanish.  
To calculate those residues, first use \eqref{eq:Rk}, 
to obtain the partial fraction expansion
\begin{equation}
R(z)\sum_{n=1}^N\frac{v_n}{z-p_n}
	=\sum_{n=1}^N\frac{\alpha_n}{(z-p_n)^2}+\frac{\beta_n}{(z-p_n)},
    \label{eq:R-sum-partial-fractions}
\end{equation}
where 
\begin{equation}
\alpha_n
	:=R_n(p_n)v_n,
\quad
\beta_n
	:=R_n'(p_n)v_n+R_n(p_n)\mathop{\sum_{m=1}^N}_{m\neq n}\frac{v_m}{p_n-p_m}.
\label{eq:alphanbetan}
\end{equation}
Then, we must impose the conditions
\begin{equation}
0
	=\mathop{\mathrm{Res}}_{z=p_k}\left[-\ee^{-\ii\lambda z/\varepsilon}R(z)\sum_{n=1}^N\frac{v_n}{z-p_n}\right]=\ee^{-\ii\lambda p_k/\varepsilon}\left(\frac{\ii\lambda}{\varepsilon} \alpha_k-\beta_k\right),\quad k=1,\dots,N.
    \label{eq:zero-residues}
\end{equation}
For $\lambda<0$, we then integrate \eqref{eq:phi-multisol} to obtain
\begin{equation}
    R(z)\varphi(z)=\int_z^{+\ii\infty}\ee^{\ii\lambda(z-\zeta)/\varepsilon}R(\zeta)\sum_{n=1}^N\frac{v_n}{\zeta-p_n}\,\dd\zeta + K\ee^{\ii\lambda z/\varepsilon},
\end{equation}
where $K$ is an integration constant and the integral is independent of any path avoiding the points $p_1,\dots,p_N\in\C_+$ by \eqref{eq:zero-residues}.  Since $\varphi\in L_2^+(\R)$ and $R(z)\to 1$ as $z\to\infty$, by letting $z\to +\ii\infty$ we see that $K=0$.  The right-hand side is therefore a meromorphic function of $z$ with poles only at $\{p_1,\dots,p_N\}$.  By \eqref{eq:R-sum-partial-fractions}--\eqref{eq:alphanbetan}, the residue of $R(z)\varphi(z)$ at $p_n$ is $\alpha_n=R_n(p_n)v_n$.  One can also check by expanding the integral for large $z$ that $R(z)\varphi(z)\to 0$ as $z\to\infty$ in all directions so 
\begin{equation}
    R(z)\varphi(z)=\sum_{n=1}^N\frac{R_n(p_n)v_n}{z-p_n}\implies \varphi(z)=\sum_{n=1}^Nv_nf_n(z),\quad f_n(z):=\frac{R_n(p_n)}{(z-p_n)R(z)}=\frac{R_n(p_n)}{R_n(z)}.
    \label{eq:varphi-sum}
\end{equation}
Therefore any eigenfunction $\varphi\in L_+^2(\R)$ of $\mathcal{L}_{u_0}$ for $u_0=\Pi u_0 + (\Pi u_0)^*$ with $\Pi u_0$ given by \eqref{eq:rational-soliton} belongs to $F$.  On the other hand, it it is known \cite[Proposition 4.3]{Sun2020} that $\mathcal{L}_{u_0}$ has exactly $N$ simple negative eigenvalues $\lambda_1<\cdots<\lambda_N<0$ with linearly independent eigenfunctions $\varphi_j\in L^2_+(\R)$.  It follows that $F=\mathrm{span}\{\varphi_j\}_{j=1}^N$.
\end{proof}

\begin{lemma}
    $F$ is an invariant space for $\mathcal{L}_{u_0}$ and for the operator $\mathcal{X}^*$.  The matrix of $\mathcal{L}_{u_0}\vert_F$ in the basis $\{f_k\}_{k=1}^N$ is $-\mathbf{M}$.  The matrix of $\mathcal{X}^*\vert_F$ in the same basis is $\mathbf{D}-\mathbf{1r}^\top$.
\label{lem:L-and-Xstar-fs}
\end{lemma}
\begin{proof}
Since $F=\mathrm{span}\{\varphi_j\}_{j=1}^N$ where $\varphi_j\in L_+^2(\R)$ with $\varphi_j\neq 0$ satisfies $\mathcal{L}_{u_0}\varphi_j=\lambda_j\varphi_j$, 
then it is obvious that $F$ is invariant under $\mathcal{L}_{u_0}$.  Also, combining \eqref{eq:zero-residues} with \eqref{eq:alphanbetan} and comparing with \eqref{def:M} shows that the vector
    $\mathbf{v}^{(j)}:=[v_1^{(j)}, v_2^{(j)}, \cdots, v_N^{(j)}]^\top$, $v_n^{(j)}=\varphi_j(p_n)$, must satisfy the equation
\begin{equation}
\mathbf{M}\mathbf{v}^{(j)}=-\lambda_j \mathbf{v}^{(j)}.
	\label{eq:M-lax}
\end{equation}
But by \eqref{eq:varphi-sum}, $\mathbf{v}^{(j)}$ is the vector of coordinates of $\varphi_j$ in the basis $\{f_n\}_{n=1}^N$.  This proves that $-\mathbf{M}$ is the matrix of $\mathcal{L}_{u_0}\vert_F$ in that basis.

Since by \eqref{eq:fk-define} and \eqref{eq:Rn-define}, $f_n(\diamond)\in\mathrm{dom}(\mathcal{X}^*)$ with $\mathfrak{a}[f_n]=R_n(p_n)$ according to \eqref{eq:domain-Xstar-alternate}, applying the action in \eqref{eq:Xstar-action-x} gives
\begin{equation}
\begin{aligned}
    \mathcal{X}^*f_n(z)=zf_n(z)-\mathfrak{a}[f_n]
    &=p_nf_n(z) + \frac{R_n(p_n)}{R(z)}\left[1-R(z)\right]\\
    &=p_nf_n(z)-\frac{R_n(p_n)}{R(z)}\sum_{k=1}^N\frac{R_k(p_k)}{z-p_k}\\
    &=p_nf_n(z)-R_n(p_n)\sum_{k=1}^Nf_k(z)
    \end{aligned}
\end{equation}
using the expansion \eqref{eq:R} of $R(z)$ in partial fractions on the second line.  Therefore $F$ is also invariant under $\mathcal{X}^*$.  Since then
\begin{equation}
    \mathcal{X}^*\sum_{n=1}^Nv_nf_n(z) = \sum_{n=1}^Np_nv_nf_n(z) - \sum_{k=1}^N\left(\sum_{n=1}^NR_n(p_n)v_n\right)f_k(z),
\end{equation}
comparing with \eqref{eq:D-define} and \eqref{eq:r1-define} shows that the matrix of $\mathcal{X}^*\vert_F$ in the basis $\{f_n\}_{n=1}^N$ is $\mathbf{D}-\mathbf{1r}^\top$.
\end{proof}

Using Lemma~\ref{lem:L-and-Xstar-fs} in \eqref{eq:Pi-u-derivative2} gives
\begin{equation}
\Pi u(t,x)=\ii\varepsilon \frac{\partial}{\partial x}\log\left(\det(\mathrm{Mat}_{\{f_k\}_{k=1}^N}(\mathcal{X}^*-2t\mathcal{L}_{u_0}-x\,\mathrm{Id})\vert_F))\right),
\end{equation}
in which we recognize the linear operator appearing in the explicit formula \eqref{eq:explicitformula} of Theorem~\ref{thm:explicitformula}.
The well-known $N$-soliton formula of Matsuno \cite{Matsuno84} now arises simply by changing the basis of $F$ from $\{f_k\}_{k=1}^N$ to the Lax eigenfunctions $\{\varphi_j\}_{j=1}^N$, which obviously leaves the matrix determinant unchanged.  Therefore,
\begin{equation}
\Pi u(t,x)
	=\ii\varepsilon\frac{\partial}{\partial x}\log\left(\det(\mathrm{Mat}_{\{\varphi_k\}_{k=1}^N}(\mathcal{X}^*-2t\mathcal{L}_{u_0}-x\,\mathrm{Id})\vert_F))\right).
\end{equation}
 The matrix of $-2t\mathcal{L}_{u_0}\vert_F$ in the eigenfunction basis is clearly $\mathrm{diag}(-2t\lambda_1,\dots,-2t\lambda_N)$, and assuming that the eigenfunctions are $L^2$-normalized by $\|\varphi_k\|_{L^2(\R)}=1$, the matrix of $\mathcal{X}^*\vert_F$ in this basis was computed in \cite[Proposition 5.4]{Sun2020} as
\begin{equation}
\left(\mathrm{Mat}_{\{\varphi_k\}_{k=1}^N}\mathcal{X}^*\vert_F\right)_{jk}=\langle \mathcal{X}^*\varphi_k,\varphi_j\rangle
	=
	\begin{cases}\displaystyle 
	-\mathrm{Re}(\gamma_j)+\frac{\ii\varepsilon }{2\lambda_j} & \text{if}\quad j=k,
	\\\displaystyle 
	\frac{ \ii \varepsilon}{\lambda_k-\lambda_j}\sqrt{\frac{\lambda_k}{\lambda_j}} 
	&\text{if} \quad j\neq k,
	\end{cases}
\end{equation}
with $\gamma_j:=-\langle \mathcal{X}^*\varphi_j,\varphi_j\rangle$ being the complex phase constant associated to $\lambda_j$. 
Therefore
\begin{equation}
    \mathrm{Mat}_{\{\varphi_k\}_{k=1}^N}(\mathcal{X}^*-2t\mathcal{L}_{u_0}-x\,\mathrm{Id})\vert_F =\mathrm{diag}\left(\frac{\ii\varepsilon}{2\lambda_1},\dots,\frac{\ii\varepsilon}{2\lambda_N}\right)\left(\mathbb{I}-\frac{\ii}{\varepsilon}\mathbf{H}(t,x)\right),
    \label{eq:Matrix-eigenfunction-basis}
\end{equation}
where $\mathbf{H}(t,x)$ is the Hermitian matrix with elements defined in \eqref{eq:Matsuno-Matrix}.
Since the diagonal factor in \eqref{eq:Matrix-eigenfunction-basis} makes no contribution to $\Pi u(t,x)$, we obtain Matsuno's formula for the $N$-soliton solution in the form \eqref{eq:Matsuno}.

\section{Some formul\ae\ appearing in the small-dispersion theory}
\label{app:uZDgeneral}
The formul\ae\ given in this Appendix will show that the model function $u^\mathrm{ZD}(t,x;\varepsilon)$ of Theorems~\ref{thm:strong-L2} and \ref{thm:strong-ZD-asymptotics}, and also the leading terms in the universality theorems near the harmonic edge (Theorem~\ref{thm:harmonic-edge}) and soliton edge (Theorem~\ref{thm:soliton-train}), make sense without any rational assumption.  Indeed, we only require $u_0\in C^1(\R)\cap L^2(\R)$ with $u_0(x)$ and $u_0'(x)$ both tending to zero as $|x|\to\infty$ to define these approximations.
\subsection{The oscillatory model function $u^\mathrm{ZD}(t,x;\varepsilon)$ in general}
Recall then the solutions $y_0(t,x)>y_1(t,x)>\cdots>y_{2\ell}(t,x)$ of $y+2tu_0(y)=x$ for all but finitely many critical values $x\in\R$ (and $\ell=\ell(t,x)=0,1,2,\dots$), and the corresponding branches $u_k^\mathrm{B}(t,x):=u_0(y_k(t,x))$ of the multivalued solution of the inviscid Burgers equation with initial data $u_0$.  
Then, the function 
\begin{equation}
g(y)=g(y;t,x):=\frac{y+2tu_0(y)-x}{\displaystyle\prod_{k=0}^{2\ell(t,x)}(y+2tu_k^\mathrm{B}(t,x)-x)}
\label{eq:g-of-y}
\end{equation}
is real-valued and positive for $y\in\mathbb{R}$, with asymptotic behavior $g(y)=y^{-2\ell(t,x)}(1+O(y^{-1}))$ as $y\to\pm\infty$.  
Consequently, the following function is well-defined:
\begin{equation}
\Psi(y)=\Psi(y;t,x):=\frac{1}{2\pi}\int_0^{+\infty}\ln\left(\frac{g(y-s)}{g(y+s)}\right)\,\frac{\dd s}{s},\quad y\in\mathbb{R}.  
\end{equation}
In \cite[Lemma 6.11]{BlackstoneGGM26} it is shown that for rational data $u_0$ of the form \eqref{eq:PN-rational-IC}, $\Psi(y)$ can be expressed in terms of the poles $\{p_n\}_{n=1}^N$ and complex solutions $y=z_m(t,x)\in\C_+$, $m=1,\dots,N-\ell$ (counted with multiplicity) of the critical point equation $y+2tu_0(y)=x$:
\begin{equation}
    \Psi(y)=\frac{\pi}{2}\ell+\sum_{n=1}^N\arg(y-p_n)-\sum_{m=1}^{N-\ell}\arg(y-z_m(t,x)).
\end{equation}
This matches the form written in the $\ell=1$ case in \eqref{eq:Psi-ell=1}.

Recall that if $\ell(t,x)=0$, then the model function $u^\mathrm{ZD}(t,x;\varepsilon):=u_0^\mathrm{B}(t,x)$ is non-oscillatory.  We now assume that $\ell(t,x)$ is an arbitrary positive integer and define 
$u^\mathrm{ZD}(t,x;\varepsilon)$ as follows.  
First, we define the Lax-Oleinik objective function
\begin{equation}
h(y)=h(y;t,x):=\frac{1}{4t}(y-x)^2 + \int_0^y u_0(s)\,\dd s,\quad y\in\mathbb{R},
\label{eq:Lax-Oleinik}
\end{equation}
and we observe that in the special case of rational data $u_0$ of the form \eqref{eq:PN-rational-IC}, $h(y;t,x)$ coincides with the form written in \eqref{eq:h-rational}.
We then define $\ell(t,x)$ nonlinear phases by evaluating $h$ at consecutive pairs of critical points:
\begin{equation}
\theta_j(t,x):=h(y_{2j-1}(t,x))-h(y_{2j}(t,x)),\quad j=1,\dots,\ell(t,x)
\end{equation}
and $\ell(t,x)$ phase corrections by similarly evaluating $\Psi$:
\begin{equation}
\varphi_j(t,x):=\frac{\pi}{2}+\Psi(y_{2j-1}(t,x))-\Psi(y_{2j}(t,x)),\quad j=1,\dots,\ell(t,x).
\end{equation}
Then introduce $\ell(t,x)$ complex quantities $\xi_j(t,x)$ by setting
\begin{equation}
\xi_j(t,x):=\ee^{\ii\varphi_j(t,x)}\sqrt{-\frac{\displaystyle (u_{2j}^\mathrm{B}-u_0^\mathrm{B})\mathop{\prod_{k=1}^{\ell}}_{k\neq j}(u^\mathrm{B}_{2j-1}-u^\mathrm{B}_{2k-1})\cdot\mathop{\prod_{i=1}^\ell}_{i\neq j}(u_{2j}^\mathrm{B}-u_{2i}^\mathrm{B})}{\displaystyle(u_{2j-1}^\mathrm{B}-u_0^\mathrm{B})\prod_{k=1}^\ell(u_{2j}^\mathrm{B}-u_{2k-1}^\mathrm{B})\cdot\prod_{i=1}^\ell(u_{2j-1}^\mathrm{B}-u_{2i}^\mathrm{B})}},\quad j=1,\dots,\ell(t,x),
\end{equation}
where $\ell=\ell(t,x)$ and $u_j^\mathrm{B}=u_j^\mathrm{B}(t,x)$ (it follows from the strict inequalities $u_0^\mathrm{B}(t,x)<\cdots<u_{2\ell(t,x)}^\mathrm{B}(t,x)$ that the expression under the radical is positive).  Let $\mathbf{M}(t,x;\varepsilon)$ be the $\ell(t,x)\times \ell(t,x)$ matrix with elements
\begin{equation}
M_{jk}(t,x;\varepsilon):=\xi_j(t,x)\ee^{\ii\theta_j(t,x)/\varepsilon}\delta_{jk} +\frac{1}{u^\mathrm{B}_{2j-1}(t,x)-u^\mathrm{B}_{2k}(t,x)},\quad j,k=1,\dots,\ell(t,x).
\end{equation}
Finally, for $\ell(t,x)=1,2,3,\dots$, we set
\begin{equation}
u^\mathrm{ZD}(t,x;\varepsilon):=u_0^\mathrm{B}(t,x)+\sum_{j=1}^{\ell(t,x)}(u_{2j-1}^\mathrm{B}(t,x)-u_{2j}^\mathrm{B}(t,x))
-2\varepsilon\mathrm{Im}\left(\frac{\partial}{\partial x}\log(\det(\mathbf{M}(t,x;\varepsilon)))\right).
\label{eq:uZD}
\end{equation}
In the case $\ell(t,x)=1$, the matrix $\mathbf{M}(t,x;\varepsilon)$ is a scalar, and it is straightforward to check that \eqref{eq:uZD} reduces to the modulated single-phase wave given in \eqref{eq:uZDell1}--\eqref{eq:r-and-periodic-profile}.

\begin{remark}[Connection with multiphase modulation theory]
One can check that 
\begin{equation}
\frac{\partial\theta_j}{\partial x}(t,x)=u_{2j-1}^\mathrm{B}(t,x)-u_{2j}^\mathrm{B}(t,x)\;\text{and}\;
-\frac{\partial\theta_j}{\partial t}(t,x)=u_{2j-1}^\mathrm{B}(t,x)^2-u_{2j}^\mathrm{B}(t,x)^2,
\end{equation}
so the expressions $\partial_x\theta_j$ and $-\partial_t\theta_j$ representing the local wavenumber and frequency of an oscillation with phase $\theta_j$ are simple expressions in terms of the branches of the multi-valued solution of the inviscid Burgers equation.
Using these expressions we observe that if $u^\mathrm{B}_j(t,x)=U_j$, $j=0,\dots,2\ell$ are all constants, then the phases $\theta_j(t,x)$ are linear functions $(U_{2j-1}-U_{2j})x-(U_{2j-1}^2-U_{2j}^2)t+\theta_j^0$, and $u^\mathrm{ZD}(t,x;\varepsilon)$ is precisely the multiphase wave solution formula of Dobrokhotov-Krichever \cite{DobrokhotovK91}.
\end{remark}

\subsection{The phase correction in Theorem~\ref{thm:harmonic-edge}}
Here we give an alternate representation of the phase correction $\varphi(t,X_\mathrm{h})$ defined in \eqref{eq:harmonic-edge-phase-correction-rational} for rational data $u_0$ of the form \eqref{eq:PN-rational-IC}.  This alternate form makes sense just as well for non-rational data.  The proof that the two forms are equivalent can be found in \cite[Section 4.3]{BlackstoneMM26}.

When $x=x_\mathrm{h}(t)$, the function $g(y)=g(y;t,x)$ given by \eqref{eq:g-of-y} degenerates as
\begin{equation}
    g_\mathrm{h}(y,t):=\frac{y+2tu_0(y)-x_\mathrm{h}(t)}{(y+2tu_2^\mathrm{B}(t)-x_\mathrm{h}(t))(y+2tu_\mathrm{h}^\mathrm{B}(t)-x_\mathrm{h}(t))^2},
\label{eq:g-h}
\end{equation}
with the repeated factor in the denominator reflecting the double characteristic $y=y_\mathrm{h}$.  Still $\R\ni y\mapsto g(y,t)$ is positive and rational for $t-t_\mathrm{c}$ positive and small enough, so we may define a quantity
\begin{equation}
    \Lambda(t):=\frac{1}{2\pi}\int_0^{+\infty}\ln\left(\frac{g_\mathrm{h}(y_\mathrm{h}(t)+s,t)}{g_\mathrm{h}(y_\mathrm{h}(t)-s,t)}\right)\frac{\dd s}{s}-\frac{1}{2\pi}\int_0^{+\infty}\ln\left(\frac{g_\mathrm{h}(y_2(t)+s,t)}{g_\mathrm{h}(y_2(t)-s,t)}\right)\frac{\dd s}{s}.
\end{equation}
Then the alternate form of the phase correction is 
\begin{equation}
    \varphi(t,X_\mathrm{h}):=\arg\left(\mathrm{Ai}(\ee^{\ii\pi/3}X_\mathrm{h})\right)-\frac{5\pi}{12}+\Lambda(t).
\label{eq:h-phase-correction}
\end{equation}

\subsection{The time-dependent position shift in Theorem~\ref{thm:soliton-train}}
Finally, we give an alternate representation of the time-dependent position shift $\lambda(t)$ defined in \eqref{eq:soliton-edge-lambda-rational} for rational $u_0$.  Again, this formula allows for $\lambda(t)$ to be defined also for non-rational $u_0$.

The analogue of the degeneration from \eqref{eq:g-of-y} to \eqref{eq:g-h} in this situation is the function 
\begin{equation}
    g_\mathrm{s}(y,t):=\frac{y+2tu_0(y)-x_\mathrm{s}(t)}{(y+2tu_0^\mathrm{B}(t)-x_\mathrm{s}(t))(y+2tu_\mathrm{s}^\mathrm{B}(t)-x_\mathrm{s}(t))^2}
    \label{eq:g-s}
\end{equation}
which is again a positive rational function of $y\in\R$ for $t-t_\mathrm{c}>0$ small enough.  The alternate form of $\lambda(t)$ is then
\begin{equation}
    \lambda(t):=\frac{1}{2\pi}\int_0^{+\infty}\left[\frac{\partial_yg_\mathrm{s}(y_\mathrm{s}(t)+s,t)}{g_\mathrm{s}(y_\mathrm{s}+s,t)}-\frac{\partial_yg_\mathrm{s}(y_\mathrm{s}(t)-s,t)}{g_\mathrm{s}(y_\mathrm{s}(t)-s,t)}\right]\frac{\dd s}{s}.
\label{eq:lambda}
\end{equation}
The proof that \eqref{eq:lambda} and \eqref{eq:soliton-edge-lambda-rational} are equivalent for rational $u_0$ can be found in \cite[Section 3.3]{BlackstoneMM26}.

\section{Asymptotic expansion for large \texorpdfstring{$t>0$}{t>0} in positive velocity frames for rational initial data}
\label{app:soliton-limit}
This Appendix is devoted to the proof of~\eqref{eq:soliton-rational} and~\eqref{eq:soliton-rational2}. We consider a rational initial data $u_0$ of the form~\eqref{eq:PN-rational-IC} for which also $\diamond u_0(\diamond)\in L^2(\R)$.  We therefore impose the additional condition \eqref{eq:sum-of-residues-zero} on the coefficients $c_1,\dots,c_N$.

\subsection{Asymptotic formula for \texorpdfstring{$\Pi u$}{Pi u}}

Fix $x_0\in\mathbb{R}$ and $v>0$,  
and set $x=vt+x_0$. In this section, we expand all the coefficients in the determinants $\det(\mathbf{A}(t,x))$ and $\det(\mathbf{B}(t,x))$ for large $t>0$ setting $\varepsilon=1$, to obtain an asymptotic formula for $\Pi u(t,x)$ depending on the speed $v$. 

Note that because \eqref{eq:sum-of-residues-zero} holds, the product of factors in \eqref{eq:PN-exph} has a branch that tends to $1$ as $z\to\infty$, and we take this branch as its definition as $z\to\infty$ with $\arg(z)=3\pi/4$.  Then continuation of the product in the distant left half-plane to the real line and then toward $z=+\infty$ one has a different asymptotic value, namely $\ee^{2\pi\ii\mathrm{Im}(c)}$, where $c=c_1+\cdots + c_N$.  We similarly define
\begin{equation}
    \ee^{-\ii h^\infty(z;v)}:=\ee^{\ii vz/2}\prod_{n=1}^N(z-p_n)^{-\ii c_n}(z-p_n^*)^{-\ii c_n^*},
\label{eq:PN-exph-infty}
\end{equation}
where the product has the same interpretation on the contours $C_0,\dots,C_N$ as in \eqref{eq:PN-exph}.
Define a matrix $\mathbf{E}$ with coefficients
\begin{equation}
    E_{jk}(v):=\int_{C_j}\frac{\ee^{-\ii h^\infty(z;v)}\,\dd z}{z-p_k},\quad j,k=1,\dots,N.
\label{eq:PN-D0-define}
\end{equation}
We also introduce the vectors $\mathbf{f}$, $\mathbf{g}$ and $\mathbf{h}$ with coefficients
\begin{equation}
    f_j(v):=\int_{C_j}\ee^{-\ii h^\infty(z;v)}\,\dd z,\quad j=1,\dots,N,
\label{eq:PN-f-define}
\end{equation}
and similarly, for $j=1,\dots,N$,
\begin{equation}
    \begin{split}
g_j(v)&:=\sum_{n=1}^Nc_n^*p_n^*\int_{C_j}\frac{\ee^{-\ii h^\infty(z;v)}\,\dd z}{z-p_n^*}  \\
h_j(v)&:=\sum_{n=1}^Nc_n^*\int_{C_j}\frac{\ee^{-\ii h^\infty(z;v)}\,\dd z}{z-p_n^*}.
    \end{split}
\label{eq:PN-g-and-h}
\end{equation}
We first prove the following proposition.
\begin{proposition}[Asymptotic formula for $\Pi u(t,vt+x_0)$]
Define $E(v):=\det(\mathbf{E}(v))$ for $v>0$. If $E(v)\neq0$, then as $t\to+\infty$,
\begin{equation}
    \Pi u(t,vt+x_0)=O(t^{-1}).
\label{eq:vanish-eigenvalue}\end{equation}
If $E(v_\mathrm{s})=0$, then as $t\to+\infty$,
\begin{equation}
    \Pi u(t,v_\mathrm{s}t+x_0)=\frac{\ii\Sigma_0(v_\mathrm{s}) + O(t^{-1})}{(x_0+\ii v_\mathrm{s}^{-1})\Sigma_0(v_\mathrm{s})-\Sigma_1(v_\mathrm{s})+O(t^{-1})},
    \label{eq:PN-Piu-soliton}
\end{equation}
with (the notation $\displaystyle \mathbf{M}\mathop{\longleftarrow}^k\mathbf{v}$ means the matrix $\mathbf{M}$ with column $k$ replaced by the vector $\mathbf{v}$)
\begin{equation}
\begin{split}
\Sigma_0(v)&:=\sum_{k=1}^N\det\left(\mathbf{E}(v)\mathop{\longleftarrow}^k\mathbf{f}(v)\right),\\
\Sigma_1(v)&:=\frac{1}{v}\sum_{k=1}^N\det\left(\mathbf{E}(v)\mathop{\longleftarrow}^k\mathbf{g}(v)\right) +\frac{1}{v}\sum_{k=1}^Np_k\det\left(\mathbf{E}(v)\mathop{\longleftarrow}^k\mathbf{h}(v)\right).
\end{split}
\end{equation}
The $O(t^{-1})$ error terms are uniform for bounded $x_0$.
\label{prop:Piu-soliton-asymp}
\end{proposition}

The rest of this section is devoted to the proof of Proposition~\ref{prop:Piu-soliton-asymp}.  Since $v>0$, when $t>0$ is sufficiently large the line $D_+$ through $z=x=vt+x_0$ with $(z-x)\ee^{\ii\pi/4}\in\mathbb{R}$ lies to the right (greater real part) of the contours $C_0,\dots,C_N$.  Therefore, by elementary row operations $\det(\mathbf{A}(t,x))=\det(\widetilde{\mathbf{A}}(t,x))$ and $\det(\mathbf{B}(t,x))=\det(\widetilde{\mathbf{B}}(t,x))$, where $\widetilde{\mathbf{A}}(t,x)$ and $\widetilde{\mathbf{B}}(t,x))$ are the same as $\mathbf{A}(t,x)$ and $\mathbf{B}(t,x)$ from Theorem~\ref{thm:inversion-formula-rational} respectively except that in each case the integration in the first row is taken over $D_+$ instead of $C_0$.  This follows because if $\Delta C:=C_J$ where $J=1,2,\dots,N$ is the largest non-exceptional index (and $\Delta C=\emptyset$ if there are no non-exceptional indices), then $\int_{C_0} f(z)\,\dd z = \int_{\Delta C}f(z)\,\dd z + \int_{D_+} f(z)\,\dd z$ holds for the common integrand in each column.  

\subsubsection{Expansion of first-row entries of \texorpdfstring{$\widetilde{\mathbf{A}}(t,x)$}{A-tilde} and \texorpdfstring{$\widetilde{\mathbf{B}}(t,x)$}{B-tilde}}

Define a constant $L\in\mathbb{R}$ by
\begin{equation}
    L:=\sum_{n=1}^N(c_np_n+c_n^*p_n^*).
\end{equation}  
Under the assumption~\eqref{eq:sum-of-residues-zero},
we see that since $|z|\ge x/\sqrt{2}\gtrsim t$ for $z\in D_+$ and $t>0$ large, 
\begin{equation}
\begin{split}
    \ee^{-2\pi \ii\mathrm{Im}(c)}\prod_{n=1}^N(z-p_n)^{-\ii c_n}(z-p_n^*)^{-\ii c_n^*} &= \prod_{n=1}^N\left(1-\frac{p_n}{z}\right)^{-\ii c_n}\left(1-\frac{p_n^*}{z}\right)^{-\ii c_n^*}\\ &=1+\frac{\ii L}{z}+O(z^{-2})\\
    &=1+\frac{\ii L}{z} + O(t^{-2})\\
    &=1+\frac{\ii L}{x} +O\left(\frac{z-x}{xz}\right) + O(t^{-2})\\
    &=1+\frac{\ii L}{x} + O(t^{-2}|z-x|) + O(t^{-2})
    \end{split}
\end{equation}
holds as $t\to+\infty$, where $c:=c_1+\cdots+c_N$ and the factor $\ee^{-2\pi\ii\mathrm{Im}(c)}$ is present due to analytic continuation along $\R$ from the distant left half-plane where the product tends to $1$.
Hence parametrizing $D_+$ by $\mathbb{R}\ni\zeta\mapsto z=x+2t^\frac{1}{2}\ee^{-\ii\pi/4}\zeta$, the exponential factor \eqref{eq:PN-exph} has the expansion
\begin{equation}
    \ee^{-\ii h(z)}=\ee^{2\pi\ii\mathrm{Im}(c)}\ee^{-\zeta^2}\left(1+\frac{\ii L}{x} + O(t^{-\frac{3}{2}}|\zeta|)+ O(t^{-2})\right),\quad t\to+\infty
    \label{eq:PN-exponential-on-D}
\end{equation}
for $\zeta\in\mathbb{R}$.  Then
\begin{equation}
\begin{split}
    \widetilde{A}_{11}(t,x)&=-\frac{1}{2t}\int_{D_+} \ee^{-\ii h(z)}(z-x)\,\dd z \\ &= 2\ii\ee^{2\pi\ii\mathrm{Im}(c)}\int_{-\infty}^{+\infty}\ee^{-\zeta^2}\left(1+\frac{\ii L}{x} + O(t^{-\frac{3}{2}}|\zeta|)+O(t^{-2})\right)\zeta\,\dd\zeta\\
    &=O(t^{-\frac{3}{2}}),\quad x=vt+x_0,\quad t\to+\infty,
\end{split}
\end{equation}
and by evaluating a Gaussian integral and using $x^{-1}=O(t^{-1})$,
\begin{equation}
\begin{split}
    \widetilde{B}_{11}(t,x)&=\int_{D_+}\ee^{-\ii h(z)}\,\dd z\\
    &=2\ee^{2\pi\ii\mathrm{Im}(c)}\ee^{-\ii\pi/4}t^{\frac{1}{2}}\int_{-\infty}^{+\infty}\ee^{-\zeta^2}\left(1+\frac{\ii L}{x}+O(t^{-\frac{3}{2}}|\zeta|) + O(t^{-2})\right)\dd\zeta\\
    &=2\sqrt{\pi}\ee^{2\pi\ii\mathrm{Im}(c)}\ee^{-\ii\pi/4}t^\frac{1}{2} + O(t^{-\frac{1}{2}}),\quad x=vt+x_0,\quad t\to+\infty.
\end{split}
\end{equation}
Next, since for each $q\in\mathbb{C}_+$ fixed and $z\in D$, 
\begin{equation}
    \frac{1}{z-q}= 
    \frac{1}{x-q}\left(1-\frac{z-x}{z-q}\right)=
    \frac{1}{vt}\left(1+O(t^{-1}) + O(t^{-1}|z-x|)\right),\quad t\to+\infty,
\end{equation}
for $k=2,\dots,N+1$ we have (using $x^{-1}=O(t^{-1})$ in \eqref{eq:PN-exponential-on-D})
\begin{multline}
    \widetilde{A}_{1k}(t,x)=\widetilde{B}_{1k}(t,x)=\int_{D_+}\frac{\ee^{-\ii h(z)}\,\dd z}{z-p_{k-1}}\\
\begin{aligned}
    &=2\ee^{2\pi\ii\mathrm{Im}(c)}\ee^{-\ii\pi/4}t^\frac{1}{2}\frac{1}{vt}\int_{-\infty}^{+\infty}\ee^{-\zeta^2}\left(1+O(t^{-1}) + O(t^{-\frac{3}{2}}|\zeta|)\right)\left(1+O(t^{-1})+O(t^{-\frac{3}{2}}|\zeta|)\right)\,\dd\zeta\\
    &=\frac{2\sqrt{\pi}\ee^{2\pi\ii\mathrm{Im}(c)}\ee^{-\ii\pi/4}}{vt^\frac{1}{2}} + O(t^{-1}),\quad x=vt+x_0,\quad t\to+\infty.
\end{aligned}
\end{multline}
In other words, the first row of $\widetilde{\mathbf{A}}(t,vt+x_0)$ has the expansion
\begin{multline}
    \begin{bmatrix}
        \widetilde{A}_{11}(t,vt+x_0) & \widetilde{A}_{12}(t,vt+x_0) & \cdots & \widetilde{A}_{1,N+1}(t,vt+x_0)
    \end{bmatrix} \\= \frac{2\sqrt{\pi}\ee^{2\pi\ii\mathrm{Im}(c)}\ee^{-\ii\pi/4}}{vt^\frac{1}{2}}\begin{bmatrix}
    O(t^{-1}) & 1+O(t^{-1}) & \cdots & 1+O(t^{-1})\end{bmatrix},\quad t\to+\infty,
\end{multline}
and the first row of $\widetilde{\mathbf{B}}(t,vt+x_0)$ has the expansion
\begin{multline}
    \begin{bmatrix}
        \widetilde{B}_{11}(t,vt+x_0) & \widetilde{B}_{12}(t,vt+x_0)& \cdots & 
        \widetilde{B}_{1,N+1}(t,vt+x_0)
    \end{bmatrix}\\=\frac{2\sqrt{\pi}\ee^{2\pi\ii\mathrm{Im}(c)}\ee^{-\ii\pi/4}}{vt^\frac{1}{2}}\begin{bmatrix}vt+O(1) & 1 +O(t^{-1}) & \cdots & 1+O(t^{-1})\end{bmatrix},\quad t\to+\infty.
\end{multline}

\subsubsection{Expansion of remaining rows of \texorpdfstring{$\widetilde{\mathbf{A}}(t,x)$}{A-tilde} and \texorpdfstring{$\widetilde{\mathbf{B}}(t,x)$}{B-tilde}}
For the remaining $N$ rows of $\widetilde{\mathbf{A}}(t,vt+x_0)$ and $\widetilde{\mathbf{B}}(t,vt+x_0)$ (which agree with the corresponding rows of $\mathbf{A}(t,vt+x_0)$ and $\mathbf{B}(t,vt+x_0)$ respectively), we maintain the original contours $C_1,\dots,C_N$ which are independent of $t$, and instead write $\ee^{-\ii h(z)}$ in the form
\begin{equation}
\begin{split}
    \ee^{-\ii h(z)}&=\ee^{-\ii(z-x_0-vt)^2/(4t)}\prod_{n=1}^N(z-p_n)^{-\ii c_n}(z-p_n^*)^{-\ii c_n^*}\\
    &=\ee^{-\ii (tv^2/4+vx_0/2+x_0^2/(4t))}\ee^{-\ii (z^2-2x_0z)/(4t)}\ee^{-\ii h^\infty(z;v)}, 
\end{split}
\end{equation}
where $ \ee^{-\ii h^\infty(z;v)}$ is defined in~\eqref{eq:PN-exph-infty}.
Then, since the factor $\ee^{-\ii (z^2-2x_0z)/(4t)}$ decays exponentially as $z\to\infty$ along $C_{j-1}$, we have 
\begin{equation}
    \ee^{-\ii (z^2-2x_0z)/(4t)}=1-\frac{\ii}{4t}(z^2-2x_0z)+O\left(\frac{1+|z|^4}{t^2}\right),\quad t\to+\infty,\quad z\in C_{j-1}
\end{equation}
for $j=2,\dots,N+1$.  Because $\ee^{-\ii h^\infty(z;v)}$ is independent of $t$ and exponentially decaying as $z\to\infty$ along $C_{j-1}$, we then get
\begin{equation}
\begin{split}
    \ee^{\ii (tv^2/4+vx_0/2+x_0^2/(4t))}\widetilde{A}_{j1}(t,vt+x_0)&=\frac{v}{2}\int_{C_{j-1}}\ee^{-\ii h^\infty(z;v)}\ee^{-\ii (z^2-2x_0z)/(4t)}\left(1-\frac{z-x_0}{vt}\right)\,\dd z \\ &= \frac{v}{2}f_{j-1}(v) + O(t^{-1}),\quad t\to+\infty,\quad j=2,\dots,N+1,
\end{split}
\end{equation}
where $\mathbf{f}$ is defined in~\eqref{eq:PN-f-define}. 
Likewise, 
\begin{equation}
    \begin{split}
        \ee^{\ii (tv^2/4+vx_0/2+x_0^2/(4t))}\widetilde{B}_{j1}(t,vt+x_0)&=\int_{C_{j-1}}\ee^{-\ii h^\infty(z;v)}\ee^{-\ii (z^2-2x_0z)/(4t)}\,\dd z\\
        &=f_{j-1}(v) + O(t^{-1}),\quad
        t\to+\infty,\quad j=2,\dots,N+1,
    \end{split}
\end{equation}
and also,
\begin{equation}
    \begin{split}
       \ee^{\ii (tv^2/4+vx_0/2+x_0^2/(4t))}\widetilde{A}_{jk}(t,vt+x_0)&=\int_{C_{j-1}}\frac{\ee^{-\ii h^\infty(z;v)}}{z-p_{k-1}}\ee^{-\ii (z^2-2x_0z)/(4t)}\,\dd z\\
        &=E_{j-1,k-1}(v)+O(t^{-1}),\quad t\to+\infty,\quad j,k=2,\dots,N+1,
    \end{split}
\end{equation}
where $\mathbf{E}(v)$ is defined by~\eqref{eq:PN-D0-define}.
Since $\widetilde{B}_{jk}(t,x)=\widetilde{A}_{jk}(t,x)$ for $j,k=2,\dots,N+1$, we also have (keeping one more term) as $t\to+\infty$ and for $j,k=2,\dots,N+1$,
\begin{multline}
    \ee^{\ii (tv^2/4+vx_0/2+x_0^2/(4t))}\widetilde{B}_{jk}(t,vt+x_0)\\
    =E_{j-1,k-1}(v)+\frac{\ii x_0}{2t}E_{j-1,k-1}^{[1]}(v)-\frac{\ii}{4t}E_{j-1,k-1}^{[2]}(v) + O(t^{-2}),
\end{multline}
where for $p\geq 0$, the matrix $\mathbf{E}^{[p]}(v)$ generalizes~\eqref{eq:PN-D0-define} with coefficients
\begin{equation}
    E^{[p]}_{jk}(v):=\int_{C_j}\frac{\ee^{-\ii h^\infty(z;v)}z^p\,\dd z}{z-p_k},\quad j,k=1,\dots,N.
\label{eq:PN-Dp-define}
\end{equation}

\subsubsection{Expansion of \texorpdfstring{$\Pi u$}{Pi u}}
Now we combine the results and expand both determinants along the first row:
\begin{multline}
    \frac{\displaystyle\det(\widetilde{\mathbf{A}}(t,vt+x_0))}{\displaystyle
    \frac{2\sqrt{\pi}\ee^{2\pi\ii\mathrm{Im}(c)}\ee^{-\ii\pi/4}}{vt^\frac{1}{2}}\ee^{-\ii N(tv^2/4+vx_0/2+x_0^2/(4t))}}
    =\begin{vmatrix}
        O(t^{-1}) & \mathbf{1}^\top+O(t^{-1}) \\
        \frac{1}{2}v\mathbf{f}(v)+O(t^{-1}) & \mathbf{E}(v)+O(t^{-1})
     \end{vmatrix}\\
         =-\frac{v}{2}\sum_{k=1}^{N}\det\left(\mathbf{E}(v)\mathop{\longleftarrow}^k
    \mathbf{f}(v)\right)
    + O(t^{-1}),\quad t\to+\infty,
    \label{eq:PN-det-Atilde}
\end{multline}
where $\mathbf{1}=(1,\dots,1)^\top\in\C^N$. 
Similarly, we have
\begin{multline}
    \frac{\displaystyle\det(\widetilde{\mathbf{B}}(t,vt+x_0))}{\displaystyle \frac{2\sqrt{\pi}\ee^{2\pi\ii\mathrm{Im}(c)}\ee^{-\ii\pi/4}}{vt^\frac{1}{2}}\ee^{-\ii N(tv^2/4+vx_0/2+x_0^2/(4t))}}
    \\=\begin{vmatrix}
    vt+O(1) & \mathbf{1}^\top+O(t^{-1}) \\
    \mathbf{f}(v)+O(t^{-1}) &\displaystyle  \mathbf{E}(v)+\frac{\ii x_0}{2t}\mathbf{E}^{[1]}(v)-\frac{\ii}{4t}\mathbf{E}^{[2]}(v)+O(t^{-2})
    \end{vmatrix}\\
    =(vt+O(1))\det\left(\mathbf{E}(v)+\frac{\ii x_0}{2t}\mathbf{E}^{[1]}(v)-\frac{\ii}{4t}\mathbf{E}^{[2]}(v)+O(t^{-2})\right) \\- \sum_{k=1}^N\det\left(\mathbf{E}(v)\mathop{\longleftarrow}^k
    \mathbf{f}(v)
    \right) + O(t^{-1}),\quad t\to+\infty.
    \label{eq:PN-det-Btilde}
\end{multline}
Now observe that
\begin{equation}
    E^{[1]}_{jk}(v)=\int_{C_j}\frac{\ee^{-\ii h^\infty(z;v)}z\,\dd z}{z-p_k}=\int_{C_j}\ee^{-\ii h^\infty(z;v)}\,\dd z+p_k\int_{C_j}\frac{\ee^{-\ii h^\infty(z;v)}\,\dd z}{z-p_k} = 
    f_j(v)
    + p_kE_{jk}(v).
\label{eq:PN-D1}
\end{equation}
Similarly,
\begin{equation}
    E^{[2]}_{jk}(v)=\int_{C_j}\ee^{-\ii h^\infty(z;v)}(z+p_k)\,\dd z + p_k^2E_{jk}(v),
    \label{eq:PN-D2}
\end{equation}
and integrating by parts using the definition \eqref{eq:PN-exph-infty} and $\ee^{\ii v z/2}=-(2\ii/v)\dd\ee^{\ii v z/2}/\dd z$,
\begin{multline}
    \int_{C_j}\ee^{-\ii h^\infty(z;v)}(z+p_k)\,\dd z \\= \frac{2\ii}{v}\int_{C_j}\ee^{-\ii h^\infty(z;v)}\left[1-\ii\sum_{n=1}^N\left(\frac{c_n(z+p_k)}{z-p_n} +\frac{c_n^*(z+p_k)}{z-p_n^*}\right)\right]\,\dd z\\
    =\frac{2\ii}{v}f_j(v) + \frac{2}{v}\sum_{n=1}^Nc_n(p_k+p_n)E_{jn}(v)+\frac{2}{v}\sum_{n=1}^Nc_n^*(p_k+p_n^*)\int_{C_j}\frac{\ee^{-\ii h^\infty(z;v)}\,\dd z}{z-p_n^*},
\end{multline}
where in the second equality we used the assumption \eqref{eq:sum-of-residues-zero}.  
Therefore, we may rewrite \eqref{eq:PN-D2} using the vectors~\eqref{eq:PN-g-and-h} as
\begin{equation}
    E^{[2]}_{jk}(v)=\frac{2\ii}{v}f_j(v)+\frac{2}{v}\sum_{n=1}^Nc_n(p_k+p_n)E_{jn}(v) +\frac{2p_k}{v}h_j(v) +\frac{2}{v}g_j(v) + p_k^2 E_{jk}(v).
\label{eq:PN-D2-final}
\end{equation}
With \eqref{eq:PN-D1} and \eqref{eq:PN-D2-final} in hand, we can use column multilinearity to expand
\begin{multline}
    \det\left(\mathbf{E}(v)+\frac{\ii x_0}{2t}\mathbf{E}^{[1]}(v)-\frac{\ii}{4t}\mathbf{E}^{[2]}(v)+O(t^{-2})\right)=\det(\mathbf{E}(v)) \\+\frac{\ii x_0}{2t}\sum_{k=1}^N\left[\det\left(\mathbf{E}(v)\mathop{\longleftarrow}^k\mathbf{f}(v)\right)+p_k\det(\mathbf{E}(v))\right] \\-\frac{\ii}{4t}\sum_{k=1}^N
    \left[\frac{2\ii}{v}\det\left(\mathbf{E}(v)\mathop{\longleftarrow}^k\mathbf{f}(v)\right) + \frac{2p_k}{v}\det\left(\mathbf{E}(v)\mathop{\longleftarrow}^k\mathbf{h}(v)\right) \right.\\
    \left.+\frac{2}{v}\det\left(\mathbf{E}(v)\mathop{\longleftarrow}^k\mathbf{g}(k)\right) + \left(\frac{4c_kp_k}{v}+p_k^2\right)\det(\mathbf{E}(v))\right] + O(t^{-2}),\quad t\to+\infty,
\end{multline}
since from the sum in \eqref{eq:PN-D2-final}, only the $n=k$ term yields a nonzero contribution.  

Now it is clear that $\Pi u(t,vt+x_0)$ takes two different asymptotic forms, depending upon whether or not $E(v)=\det(\mathbf{E}(v))$ vanishes for given $v>0$.  In the general case that $E(v)\neq 0$, we may simplify \eqref{eq:PN-det-Atilde} as
\begin{equation}
\frac{\displaystyle\det(\widetilde{\mathbf{A}}(t,vt+x_0))}{\displaystyle
    \frac{2\sqrt{\pi}\ee^{2\pi\ii\mathrm{Im}(c)}\ee^{-\ii\pi/4}}{vt^\frac{1}{2}}\ee^{-\ii N(tv^2/4+vx_0/2+x_0^2/(4t))}} = O(1),\quad t\to+\infty
\end{equation}
and \eqref{eq:PN-det-Btilde} as
\begin{equation}
\frac{\displaystyle\det(\widetilde{\mathbf{B}}(t,vt+x_0))}{\displaystyle \frac{2\sqrt{\pi}\ee^{2\pi\ii\mathrm{Im}(c)}\ee^{-\ii\pi/4}}{vt^\frac{1}{2}}\ee^{-\ii N(tv^2/4+vx_0/2+x_0^2/(4t))}}
=vtE(v) + O(1),\quad t\to+\infty,
\end{equation}
and hence $\Pi u(t,vt+x_0)=\det(\widetilde{\mathbf{A}}(t,vt+x_0))/\det(\widetilde{\mathbf{B}}(t,vt+x_0))$ satisfies~\eqref{eq:vanish-eigenvalue}, with the error term being uniform for bounded $x_0$.  On the other hand, if $v=v_\mathrm{s}>0$ is a zero of $E(v)$, then we retain the full form of \eqref{eq:PN-det-Atilde} (replacing $v$ with $v_\mathrm{s}$) but \eqref{eq:PN-det-Btilde} simplifies instead as
\begin{multline}    \frac{\displaystyle\det(\widetilde{\mathbf{B}}(t,v_\mathrm{s}t+x_0))}{\displaystyle \frac{2\sqrt{\pi}\ee^{2\pi\ii\mathrm{Im}(c)}\ee^{-\ii\pi/4}}{v_\mathrm{s}t^\frac{1}{2}}\ee^{-\ii N(tv_\mathrm{s}^2/4+v_\mathrm{s}x_0/2+x_0^2/(4t))}}
=\left(\frac{\ii v_\mathrm{s}x_0}{2}-\frac{1}{2}\right)\sum_{k=1}^N\det\left(\mathbf{E}(v_\mathrm{s})\mathop{\longleftarrow}^k\mathbf{f}(v_\mathrm{s})\right)\\
-\frac{\ii}{2}\sum_{k=1}^N\det\left(\mathbf{E}(v_\mathrm{s})\mathop{\longleftarrow}^k\mathbf{g}(v_\mathrm{s})\right)-\frac{\ii}{2}\sum_{k=1}^Np_k\det\left(\mathbf{E}(v_\mathrm{s})\mathop{\longleftarrow}^k\mathbf{h}(v_\mathrm{s})\right) + O(t^{-1}),\quad t\to+\infty.
\end{multline}
Then we obtain~\eqref{eq:PN-Piu-soliton}.  This completes the proof of Proposition~\ref{prop:Piu-soliton-asymp}.

The statement \eqref{eq:vanish-eigenvalue} in Proposition~\ref{prop:Piu-soliton-asymp} is equivalent to \eqref{eq:soliton-rational2}.  To prove \eqref{eq:soliton-rational} it remains to show that it is implied by \eqref{eq:PN-Piu-soliton}, which is the subject of the next section.

\subsection{Further simplification of \texorpdfstring{$\Pi u$}{Pi u} when \texorpdfstring{$v=v_\mathrm{s}$}{v=vs}}
The first step is to show that when $E(v_\mathrm{s})=0$, $\Sigma_0(v_\mathrm{s})\neq 0$.  
\subsubsection{The inequality $\Sigma_0(v_\mathrm{s})\neq 0$ and interpretation of $\lambda_\mathrm{s}=-\frac{1}{2}v_\mathrm{s}$ as a Lax eigenvalue}
Differentiation of $E(v)$ using column multi-linearity gives
\begin{equation}
\begin{split}
    E'(v)&=\frac{\dd}{\dd v}\det(\mathbf{E}(v))\\ &=\frac{\ii}{2}\sum_{k=1}^N\det\left(\mathbf{E}(v)\mathop{\longleftarrow}^k\mathbf{f}(v)\right) +\frac{\ii}{2}\det(\mathbf{E}(v))\sum_{k=1}^Np_k\\ &=\frac{\ii}{2}\Sigma_0(v)+\frac{\ii}{2}E(v)\sum_{k=1}^Np_k.
\end{split}
\end{equation}
Hence we have
\begin{equation}
E(v_\mathrm{s})=0 \implies \Sigma_0(v_\mathrm{s})=-2\ii E'(v_\mathrm{s}),
    \label{eq:PN-Sigma0-Deltaprime}
\end{equation}
so $\Sigma_0(v_\mathrm{s})\neq 0$ means that $v_\mathrm{s}>0$ is a simple root of $E(v)$.

We will show that all roots $v>0$ of $E(v)$ are simple, and at the same time show that the roots are proportional to the eigenvalues of $\mathcal{L}_{u_0}$.  Under the condition \eqref{eq:sum-of-residues-zero} on the rational data $u_0$ of the form \eqref{eq:PN-rational-IC}, it follows from \cite{Wu16a} that the discrete spectrum of $\mathcal{L}_{u_0}$ is a finite set of negative simple eigenvalues $\lambda<0$.

\begin{siamprop}[\protect{\cite[Section 3.3]{MillerW16a}}]
Let $u_0$ be a rational function of the form \eqref{eq:PN-rational-IC} subject to the condition \eqref{eq:sum-of-residues-zero} on $c_1,\dots,c_N$.
Then, all roots $v=v_\mathrm{s}>0$ of $v\mapsto E(v):=\det(\mathbf{E}(v))$ are simple, and $\lambda<0$ is an eigenvalue of the Lax operator $\mathcal{L}_{u_0}$ if and only if $E(-2\lambda)=0$.
\label{PN-prop-MW16}
\end{siamprop}

The notation in \cite{MillerW16a} is slightly different but equivalent to ours.  To make the connection between $E(v)$ and the Evans function defined in that paper, one only needs to observe that by Cauchy's theorem the contours $C_1,\dots,C_N$ 
can be rotated so that they tend to $z=\infty$ in the upward vertical direction of the complex plane.  The proof of simplicity of the roots given in \cite{MillerW16a} is not convincing, so below we give a new proof of that part of the statement of Proposition~\ref{PN-prop-MW16}.

But first, we point out that since \eqref{eq:PN-Sigma0-Deltaprime} and Proposition~\ref{PN-prop-MW16} give that $\Sigma_0(v_\mathrm{s})\neq 0$ whenever $E(v_\mathrm{s})=0$, \eqref{eq:PN-Piu-soliton} can be written equivalently in the form 
\begin{equation}
    \Pi u(t,v_\mathrm{s}t+x_0)=\frac{\ii + O(t^{-1})}{x_0+\ii v_\mathrm{s}^{-1}-\Sigma_1(v_\mathrm{s})/\Sigma_0(v_\mathrm{s})+O(t^{-1})},\quad t\to+\infty,\quad E(v_\mathrm{s})=0.
    \label{eq:PN-Piu-almost-done}
\end{equation}

\begin{proof}[Proof of simplicity of roots of $E$ in Proposition~\ref{PN-prop-MW16}]
According to \cite{Wu16a}, for every eigenvalue  $\lambda=\lambda_\mathrm{s}<0$ of $\mathcal{L}_{u_0}$ we have $\dim(\mathrm{Ker}(\mathcal{L}_{u_0}-\lambda_\mathrm{s}\mathrm{Id}))=1$, and by \cite[Section 3.3]{MillerW16a}, $\mathrm{Ker}(\mathbf{E}(v_\mathrm{s}))$ is isomorphic to $\mathrm{Ker}(\mathcal{L}_{u_0}-\lambda_\mathrm{s}\mathrm{Id})$ for $v_\mathrm{s}=-2\lambda_\mathrm{s}$, so there is a nonzero vector $\boldsymbol{\phi}\in\mathbb{C}^N$ such that $\mathrm{Ker}(\mathbf{E}(v_\mathrm{s}))=\mathrm{span}\{\boldsymbol{\phi}\}$.  It then follows from the theory of holomorphic matrix pencils (see \cite[Section 2.2]{KollarM14} for instance) that the two conditions $E(v_\mathrm{s})=E'(v_\mathrm{s})=0$ imply the existence of a vector $\boldsymbol{\psi}\in\mathbb{C}^N$ such that
\begin{equation}
    \mathbf{E}(v_\mathrm{s})\boldsymbol{\psi} + \mathbf{E}^{\prime}(v_\mathrm{s})\boldsymbol{\phi}=\mathbf{0}.
\end{equation}
Now differentiation with respect to $v$ gives $E_{jk}^{\prime}(v)=\frac{1}{2}\ii E_{jk}^{[1]}(v)$, so using \eqref{eq:PN-D1} we can write the equation on $\boldsymbol{\psi}$ in the form
\begin{equation}
    \mathbf{E}(v_\mathrm{s})\boldsymbol{\psi} +\frac{1}{2}\ii \mathbf{f}(v_\mathrm{s})\mathbf{1}^\top\boldsymbol{\phi} + \frac{1}{2}\ii\mathbf{E}(v_\mathrm{s})\mathrm{diag}(p_1,\dots,p_N)\boldsymbol{\phi}=\mathbf{0}.
    \label{eq:PN-psi-vector-equation}
\end{equation}
An eigenfunction $\Phi\in \mathrm{Ker}(\mathcal{L}_{u_0}-\lambda_\mathrm{s}\mathrm{Id})$ can be normalized so that $x\Phi(x)\to 1$ as $x\to\infty$ in $\mathbb{C}_+$, and the isomorphism with $\mathrm{Ker}(\mathbf{E}(v_\mathrm{s}))$ developed in \cite[Section 3.3]{MillerW16a} shows that the corresponding nonzero nullvector $\boldsymbol{\phi}$ satisfies
$\mathbf{1}^\top\boldsymbol{\phi}=\lambda_\mathrm{s}\neq 0$.  Therefore the existence of $\boldsymbol{\psi}$ solving \eqref{eq:PN-psi-vector-equation} is equivalent to the existence of $\mathbf{a}\in\mathbb{C}^N$ given explicitly by
\begin{equation}
    \mathbf{a}:= \frac{2\ii \boldsymbol{\psi}-\mathrm{diag}(p_1,\dots,p_N)\boldsymbol{\phi}}{\mathbf{1}^\top\boldsymbol{\phi}}
\end{equation}
that satisfies $\mathbf{E}(v_\mathrm{s})\mathbf{a}=\mathbf{f}(v_\mathrm{s})$.  We will show that existence of such $\mathbf{a}\in\mathbb{C}^N$ leads to a contradiction, implying that $E(v_\mathrm{s})$ and $E'(v_\mathrm{s})$ cannot both vanish.

To this end, consider the function $g(x)$ defined by the formula
\begin{equation}
    g(x):=\ee^{\ii h^\infty(x;v_\mathrm{s})}\int_{\ee^{3\pi\ii/4}\infty}^x\ee^{-\ii h^\infty(z;v_\mathrm{s})}\left(\sum_{k=1}^N\frac{a_k}{z-p_k}-1\right)\,\dd z.
\end{equation}
The hypothesis that $\mathbf{E}(v_\mathrm{s})\mathbf{a}=\mathbf{f}(v_\mathrm{s})$ implies that $g$ is the boundary value of a function analytic in the upper half $x$-plane (see \cite[Section 3.2]{MillerW16a}).  In fact, $g\in L^\infty_+(\mathbb{R})$, i.e., it is bounded on the closed upper half-plane, and one can verify that $g(x)\to -\ii\lambda_\mathrm{s}^{-1}$ as $x\to\infty$ in $\mathbb{C}_+$.  Therefore also $u_0g\in L^2(\mathbb{R})$.  We then compute
\begin{equation}
    (\mathcal{L}_{u_0}-\lambda_\mathrm{s}\mathrm{Id})g(x)=-\ii g'(x)-\Pi(u_0g)(x)-\lambda_\mathrm{s} g(x).
\end{equation}
Since $g\in L_+^\infty(\mathbb{R})$, we have
\begin{equation}
    \Pi(u_0g)(x)=u_0(x)g(x)-\sum_{k=1}^N\frac{c_kg(p_k)}{x-p_k}.
\end{equation}
The second term depends on $g$ through its analytic continuation to the points $\{p_k\}\subset\mathbb{C}_+$, but it is a rational function of $x$.  Therefore,
\begin{equation}
\begin{split}
    (\mathcal{L}_{u_0}-\lambda_\mathrm{s}\mathrm{Id})g(x)&=\ee^{\ii h^\infty(x;v_\mathrm{s})}\left[\ee^{-\ii h^\infty(x;v_\mathrm{s})}\left(-\ii g'(x)-u_0(x)g(x)-\lambda_\mathrm{s}g(x)\right)\right]+\sum_{k=1}^N\frac{c_kg(p_k)}{x-p_k}\\
    &=-\ii\ee^{\ii h^\infty(x;v_\mathrm{s})}\frac{\dd}{\dd x}\left[\ee^{-\ii h^\infty(x;v_\mathrm{s})}g(x)\right] + \sum_{k=1}^N\frac{c_kg(p_k)}{x-p_k}\\
    &=-\ii\left(\sum_{k=1}^N\frac{a_k}{x-p_k}-1\right) + \sum_{k=1}^N\frac{c_kg(p_k)}{x-p_k}.
\end{split}
\end{equation}
On the other hand, since by hypothesis $g\in L_+^\infty(\mathbb{R})$, the left-hand side has to be the boundary value for $x\in\mathbb{R}$ of a function analytic in the upper half $x$-plane, which implies that 
$c_kg(p_k)=\ii a_k$ for all $k$, and hence $g(x)$ satisfies
\begin{equation}
    (\mathcal{L}_{u_0}-\lambda_\mathrm{s}\mathrm{Id})g(x)=\ii,
\label{eq:PN-g-equation}
\end{equation}
a constant nonzero function.  

We can write $g(x)=-\ii\lambda_\mathrm{s}^{-1} + h(x)$, where $h\in L_+^2(\mathbb{R})$.  Now for $\epsilon>0$, let $\chi_\epsilon\in L_+^2(\mathbb{R})$ be defined by
\begin{equation}
    \chi_\epsilon(x):=\frac{1}{1-\ii\epsilon x},\quad\epsilon>0.
    \label{eq:PN-chi-epsilon-def}
\end{equation}
Note that $\chi_\epsilon(x)\to 1$ as $\epsilon\downarrow 0$ for every $x\in\mathbb{R}$ and that $|\chi_\epsilon(x)|\le 1$.
Recalling the eigenfunction $\Phi\in \mathrm{Ker}(\mathcal{L}_{u_0}-\lambda_\mathrm{s}\mathrm{Id})$ normalized so that 
$\Phi(x)=x^{-1}+O(x^{-2})$ as $x\to\infty$ in $\mathbb{C}_+$, multiplying \eqref{eq:PN-g-equation} by $\Phi(x)^*\chi_\epsilon(x)$ and integrating over $\mathbb{R}$ implies that
\begin{equation}
    \langle\chi_\epsilon (\mathcal{L}_{u_0}-\lambda_\mathrm{s}\mathrm{Id})g,\Phi\rangle = \ii\langle\chi_\epsilon,\Phi\rangle,\quad\epsilon>0.
\end{equation}
Using $g(x)=-\ii\lambda_\mathrm{s}^{-1}+h(x)$, where the action of $\mathcal{L}_{u_0}-\lambda_\mathrm{s}\mathrm{Id}$ on the constant function $1$ is well defined:
\begin{equation}\label{eq:Lax-one}
    (\mathcal{L}_{u_0}-\lambda_\mathrm{s}\mathrm{Id})1(x)=-\Pi u_0(x)-\lambda_\mathrm{s},
\end{equation} 
we get that $h$ is in the domain of $\mathcal{L}_{u_0}$, and
\begin{equation}
\langle\chi_\epsilon (\mathcal{L}_{u_0}-\lambda_\mathrm{s}\mathrm{Id})h,\Phi\rangle -\ii \lambda_\mathrm{s}^{-1}\langle\chi_\epsilon(\mathcal{L}_{u_0}-\lambda_\mathrm{s}\mathrm{Id})1,\Phi\rangle=\ii\langle\chi_\epsilon,\Phi\rangle,\quad\epsilon>0.
\label{eq:PN-inner-product-equation}
\end{equation}
Using~\eqref{eq:Lax-one}, we arrive at the identity
\begin{equation}
    \langle\chi_\epsilon(\mathcal{L}_{u_0}-\lambda_\mathrm{s}\mathrm{Id})h,\Phi\rangle +\ii\lambda_\mathrm{s}^{-1}\langle\chi_\epsilon\Pi u_0,\Phi\rangle = 0,\quad\epsilon>0.
\end{equation}
Because $h$ is in the domain of $\mathcal{L}_{u_0}$ while $\Phi$ and $\Pi u_0$ are both in $L_+^2(\mathbb{R})$, a dominated convergence argument shows that the limit $\epsilon\downarrow 0$ can be taken in both terms, replacing $\chi_\epsilon$ by its pointwise limit $1$ and yielding
\begin{equation}
    \langle (\mathcal{L}_{u_0}-\lambda_\mathrm{s}\mathrm{Id})h,\Phi\rangle  +\ii\lambda_\mathrm{s}^{-1}\langle \Pi u_0,\Phi\rangle=0.
\label{eq:after-dominated-convergence}
\end{equation}
Now since $\mathcal{L}_{u_0}-\lambda_\mathrm{s}\mathrm{Id}$ is self adjoint we have
\begin{equation}
    \langle (\mathcal{L}_{u_0}-\lambda_\mathrm{s}\mathrm{Id})h,\Phi\rangle=\langle h,(\mathcal{L}_{u_0}-\lambda_\mathrm{s}\mathrm{Id})\Phi\rangle = \langle h,0\rangle = 0
\end{equation}
because $\mathcal{L}_{u_0}\Phi = \lambda_\mathrm{s}\Phi$, and since $-\ii\lambda_\mathrm{s}\neq 0$, \eqref{eq:after-dominated-convergence} gives
\begin{equation}
    \langle\Pi u_0,\Phi\rangle = 0.
    \label{eq:PN-penultimate-identity}
\end{equation}
Then, using
$\langle \Pi u_0,\Phi\rangle = \langle u_0,\Pi\Phi\rangle = \langle u_0,\Phi\rangle$, \eqref{eq:PN-penultimate-identity} becomes
\begin{equation}
    \int_\mathbb{R} \Phi(x)u_0(x)\,\dd x = 0
\end{equation}
because $u_0$ is real valued.  On the other hand using the normalization $\Phi(x)=x^{-1}+O(x^{-2})$ and the singular integral representation \eqref{eq:Pi-singular-integral} of $\Pi$ to expand the Lax equation $\mathcal{L}_{u_0}\Phi=\lambda_\mathrm{s}\Phi$ as $x\to\infty$ in $\C_+$ shows that (see also \cite[Lemma 2.5 and (2.30)]{Wu16a})
\begin{equation}
    \int_{\mathbb{R}}\Phi(x)u_0(x)\,\dd x = 2\pi\ii\lambda_\mathrm{s}\neq 0.
\end{equation}
This is the desired contradiction.  Hence the hypothesis that there exists $\mathbf{a}\in \mathbb{C}^N$ such that $\mathbf{E}(v_\mathrm{s})\mathbf{a}=\mathbf{f}(v_\mathrm{s})$ is false, and simplicity of the roots of the determinant $E(v)$ is proved.
\end{proof}

We have therefore proved \eqref{eq:PN-Piu-almost-done}.  Comparing \eqref{eq:PN-Piu-almost-done} with \eqref{eq:soliton-rational}, it remains to identify $\Sigma_1(v_\mathrm{s})/\Sigma_0(v_\mathrm{s})$ with $2\mathrm{Re}(\langle\mathcal{X}^*\varphi,\varphi\rangle)$, where $\varphi$ is a Lax eigenfunction  for the eigenvalue $\lambda_\mathrm{s}=-\frac{1}{2}v_\mathrm{s}$ normalized so that $\|\varphi\|_2^2=1$. 

\subsubsection{Simplification of $\Sigma_1(v_\mathrm{s})/\Sigma_0(v_\mathrm{s})$}
We will make the required identification in two steps, first relating $\Sigma_1(v_\mathrm{s})/\Sigma_0(v_\mathrm{s})$ to a complex phase constant $\gamma_\mathrm{s}$ associated with $\lambda_\mathrm{s}=-\frac{1}{2}v_\mathrm{s}$ via Proposition~\ref{prop:sigma1/sigma0}, and then providing an alternate formula for $\gamma_\mathrm{s}$ in Lemma~\ref{lem:PN-phase-constant}.

Associated with each Lax eigenvalue $\lambda_\mathrm{s}<0$ there is a complex \emph{phase constant} $\gamma_\mathrm{s}$ that we define as follows (see, for instance \cite[Equation (10)]{MillerW16a} or \cite[Corollary 2.7]{Wu16a}).  Consider the Jost solution $W(x;\lambda)$, which for given $\lambda\in\rho(\mathcal{L}_{u_0})$ is the unique solution of the equation $(\mathcal{L}_{u_0}-\lambda\,\mathrm{Id})W=-\lambda$, where $W=1+w$, $w\in L_+^2(\mathbb{R})\cap H^1(\mathbb{R})$.  Near an eigenvalue $\lambda=\lambda_\mathrm{s}<0$ one has a Laurent expansion of $\lambda\mapsto W(x;\lambda)$ exhibiting a simple pole at $\lambda=\lambda_\mathrm{s}$:
\begin{equation}
    W(x;\lambda)=-\frac{\ii\Phi(x)}{\lambda-\lambda_\mathrm{s}} + (x+\gamma_\mathrm{s})\Phi(x) + O(\lambda-\lambda_\mathrm{s}),\quad\lambda\to\lambda_\mathrm{s},
\label{eq:PN-W-Laurent}
\end{equation}
which serves to define $\gamma_\mathrm{s}$.
Here $\Phi\in L_+^2(\mathbb{R})\cap H^1(\mathbb{R})$ is the unique element of $\mathrm{Ker}(\mathcal{L}_{u_0}-\lambda_\mathrm{s}\,\mathrm{Id})$ that is normalized so that $\Phi(x)=x^{-1}+O(x^{-2})$ as $x\to\infty$ in $\mathbb{C}_+$.
\begin{siamprop}
Let $\gamma_\mathrm{s}$ denote the phase constant associated with the eigenvalue $\lambda_\mathrm{s}=-\frac{1}{2}v_\mathrm{s}$.  Then
\begin{equation}
    \frac{\Sigma_1(v_\mathrm{s})}{\Sigma_0(v_\mathrm{s})}=-\mathrm{Re}(\gamma_\mathrm{s}).
    \label{eq:PN-Sigma1-over-Sigma0-final}
\end{equation}
\label{prop:sigma1/sigma0}
\end{siamprop}

\begin{proof}
Recall $\mathbf{1}:=(1,1,\dots,1)^\top$ and set $\mathbf{p}:=(p_1,p_2,\dots,p_N)^\top$.
Then note that if $E(v)\neq 0$, by Cramer's rule,
\begin{equation}
    \frac{\Sigma_0(v)}{E(v)} = \mathbf{1}^\top\mathbf{w}(v)\quad\text{and}\quad\frac{v\Sigma_1(v)}{E(v)}=\mathbf{1}^\top\mathbf{x}(v)+\mathbf{p}^\top\mathbf{y}(v),
\end{equation}
where $\mathbf{w}(v)$, $\mathbf{x}(v)$, $\mathbf{y}(v)$ are the unique solutions of 
\begin{equation}
    \mathbf{E}(v)\mathbf{w}(v)=\mathbf{f}(v),\quad
    \mathbf{E}(v)\mathbf{x}(v)=\mathbf{g}(v),\quad
    \mathbf{E}(v)\mathbf{y}(v)=\mathbf{h}(v),\quad E(v)\neq 0.
\end{equation}
Therefore also
\begin{equation}
    \frac{\Sigma_1(v)}{\Sigma_0(v)} = \frac{\mathbf{1}^\top\mathbf{x}(v)+\mathbf{p}^\top\mathbf{y}(v)}{v\mathbf{1}^\top\mathbf{w}(v)},\quad E(v)\neq 0.
\end{equation}
Since $\Sigma_j(v)$ are analytic in $v$ at $v=v_\mathrm{s}>0$ and $\Sigma_0(v_\mathrm{s})\neq 0$, we can take the limit to obtain
\begin{equation}
    \frac{\Sigma_1(v_\mathrm{s})}{\Sigma_0(v_\mathrm{s})} = \frac{1}{v_\mathrm{s}}\lim_{v\to v_\mathrm{s}}\frac{\mathbf{1}^\top\mathbf{x}(v)+\mathbf{p}^\top\mathbf{y}(v)}{\mathbf{1}^\top\mathbf{w}(v)}.
\label{eq:PN-Sigma1Sigma0-limit}
\end{equation}

Since $E(v)=\det(\mathbf{E}(v))$ has a simple zero at $v=v_\mathrm{s}$, the vectors $\mathbf{w}(v)$, $\mathbf{x}(v)$, and $\mathbf{y}(v)$ have at worst simple poles at $v=v_\mathrm{s}$.  For instance $\mathbf{w}(v)$ has a Laurent expansion of the form
\begin{equation}
    \mathbf{w}(v)=\frac{1}{v-v_\mathrm{s}}\mathbf{w}^{[-1]} + \mathbf{w}^{[0]} + O(v-v_\mathrm{s}),
\end{equation}
and using the Taylor expansions
\begin{equation}
    \begin{split}
        \mathbf{E}(v)&=\mathbf{E}(v_\mathrm{s}) + (v-v_\mathrm{s})\mathbf{E}^{\prime}(v_\mathrm{s}) + O((v-v_\mathrm{s})^2)\\
        \mathbf{f}(v)&=\mathbf{f}(v_\mathrm{s}) + O(v-v_\mathrm{s}),
    \end{split}
\end{equation}
we substitute into $\mathbf{E}(v)\mathbf{w}(v)=\mathbf{f}(v)$ to obtain
\begin{equation}
    \mathbf{E}(v_\mathrm{s})\mathbf{w}^{[-1]}=\mathbf{0},
\label{eq:PN-w-leading}
\end{equation}
\begin{equation}
    \mathbf{E}(v_\mathrm{s})\mathbf{w}^{[0]}=\mathbf{f}(v_\mathrm{s}) -\mathbf{E}^{\prime}(v_\mathrm{s})\mathbf{w}^{[-1]},
\label{eq:PN-w-correction}
\end{equation}
and so on.  Clearly \eqref{eq:PN-w-leading} implies that we need $\mathbf{w}^{[-1]}\in \mathrm{Ker}(\mathbf{E}(v_\mathrm{s}))=\mathrm{span}\{\boldsymbol{\phi}\}$.  Normalizing $\boldsymbol{\phi}$ by assuming that $\mathbf{1}^\top\boldsymbol{\phi}=\lambda_\mathrm{s}=-\frac{1}{2}v_\mathrm{s}$, we may therefore write $\mathbf{w}^{[-1]}=\tau\boldsymbol{\phi}$ for some scalar $\tau$.  To determine $\tau$, let $\mathbf{m}$ be a nonzero left nullvector of $\mathbf{E}(v_\mathrm{s})$, i.e., $\mathbf{m}^\top\mathbf{E}(v_\mathrm{s})=\mathbf{0}^\top$.  Then the solvability condition for existence of $\mathbf{w}^{[0]}$ is obtained by multiplying \eqref{eq:PN-w-correction} on the left by $\mathbf{m}^\top$, leading to
\begin{equation}
\tau=\frac{\mathbf{m}^\top\mathbf{f}(v_\mathrm{s})}{\mathbf{m}^\top\mathbf{E}^{\prime}(v_\mathrm{s})\boldsymbol{\phi}}\implies \mathbf{w}(v)=\frac{1}{v-v_\mathrm{s}}\frac{\mathbf{m}^\top\mathbf{f}(v_\mathrm{s})}{\mathbf{m}^\top\mathbf{E}^{\prime}(v_\mathrm{s})\boldsymbol{\phi}}\boldsymbol{\phi}+O(1).
\end{equation}
Similarly, one finds that 
\begin{equation}
    \mathbf{x}(v)=\frac{1}{v-v_\mathrm{s}}\frac{\mathbf{m}^\top\mathbf{g}(v_\mathrm{s})}{\mathbf{m}^\top\mathbf{E}^{\prime}(v_\mathrm{s})\boldsymbol{\phi}}\boldsymbol{\phi}+O(1),\quad\mathbf{y}(v)=\frac{1}{v-v_\mathrm{s}}\frac{\mathbf{m}^\top\mathbf{h}(v_\mathrm{s})}{\mathbf{m}^\top\mathbf{E}^{\prime}(v_\mathrm{s})\boldsymbol{\phi}}\boldsymbol{\phi}+O(1).
\end{equation}
By $E_{jk}^{\prime}(v)=\frac{1}{2}\ii E_{jk}^{[1]}(v)$ and \eqref{eq:PN-D1} we have
\begin{equation}
    \mathbf{m}^\top\mathbf{E}^{\prime}(v_\mathrm{s})\boldsymbol{\phi} =\frac{1}{2}\ii\mathbf{m}^\top\mathbf{f}(v_\mathrm{s})\mathbf{1}^\top\boldsymbol{\phi} +\frac{1}{2}\ii\mathbf{m}^\top\mathbf{E}(v_\mathrm{s})\mathrm{diag}(p_1,\dots,p_N)\boldsymbol{\phi} = -\frac{1}{4}\ii v_\mathrm{s} \mathbf{m}^\top\mathbf{f}(v_\mathrm{s}),
\end{equation}
where we used $\mathbf{m}^\top\mathbf{E}(v_\mathrm{s})=\mathbf{0}^\top$ and $\mathbf{1}^\top\boldsymbol{\phi}=-\frac{1}{2}v_\mathrm{s}$.  The condition $\mathbf{m}^\top\mathbf{f}(v_\mathrm{s})=0$ is equivalent to the existence of $\mathbf{a}\in\mathbb{C}^N$ such that $\mathbf{E}(v_\mathrm{s})\mathbf{a}=\mathbf{f}(v_\mathrm{s})$, and it was shown in the proof of simplicity of zeros of $E(v)$ that no such $\mathbf{a}$ can exist.  
Therefore, cancelling nonzero common denominators, \eqref{eq:PN-Sigma1Sigma0-limit} becomes
\begin{equation}
    \frac{\Sigma_1(v_\mathrm{s})}{\Sigma_0(v_\mathrm{s})}=\frac{1}{v_\mathrm{s}}\frac{\mathbf{m}^\top\mathbf{g}(v_\mathrm{s})\mathbf{1}^\top\boldsymbol{\phi} + \mathbf{m}^\top\mathbf{h}(v_\mathrm{s})\mathbf{p}^\top\boldsymbol{\phi}}{\mathbf{m}^\top\mathbf{f}(v_\mathrm{s})\mathbf{1}^\top\boldsymbol{\phi}}=\frac{v_\mathrm{s}\mathbf{m}^\top\mathbf{g}(v_\mathrm{s})-2\mathbf{m}^\top\mathbf{h}(v_\mathrm{s})\mathbf{p}^\top\boldsymbol{\phi}}{v_\mathrm{s}^2\mathbf{m}^\top\mathbf{f}(v_\mathrm{s})}.
\label{eq:PN-Sigma1-Sigma0-no-limit}
\end{equation}

We focus first on the term $\mathbf{m}^\top\mathbf{h}(v_\mathrm{s})$.  Because $\mathbf{m}^\top\mathbf{E}(v_\mathrm{s})=\mathbf{0}^\top$, we can add some terms to $\mathbf{h}(v_\mathrm{s})$ without changing the value of $\mathbf{m}^\top\mathbf{h}(v_\mathrm{s})$:
\begin{equation}
    \mathbf{m}^\top\mathbf{h}(v_\mathrm{s}) = \mathbf{m}^\top\left(\mathbf{E}(v_\mathrm{s})\mathbf{c} +\mathbf{h}(v_\mathrm{s})\right),\quad \mathbf{c}:=(c_1,\dots,c_N)^\top.
\end{equation}
Using \eqref{eq:PN-Dp-define} and \eqref{eq:PN-g-and-h}, the components of $\mathbf{E}(v_\mathrm{s})\mathbf{c} +\mathbf{h}(v_\mathrm{s})$ are
\begin{equation}
\begin{split}
    \left(\mathbf{E}(v_\mathrm{s})\mathbf{c} +\mathbf{h}(v_\mathrm{s})\right)_j&=\int_{C_j}\ee^{-\ii h^\infty(z;v_\mathrm{s})}\left(\sum_{n=1}^N\frac{c_n}{z-p_n} + \sum_{n=1}^N\frac{c_n^*}{z-p_n^*}\right)\,\dd z \\ &= \int_{C_j}\ee^{-\ii h^\infty(z;v_\mathrm{s})}u_0(z)\,\dd z,\quad j=1,\dots,N,
\end{split}
\end{equation}
where on the second line we used \eqref{eq:PN-rational-IC}.  Since $(h^{\infty})^{\prime}(z;v_\mathrm{s})=u_0(z)-\frac{1}{2}v_\mathrm{s}$ and 
\begin{equation}
    \int_{C_j}\ee^{-\ii h^\infty(z;v_\mathrm{s})}(h^{\infty})^{\prime}(z;v_\mathrm{s})\,\dd z=0,\quad j=1,\dots,N,
\end{equation}
we can also write
\begin{equation}
    \left(\mathbf{E}(v_\mathrm{s})\mathbf{c} +\mathbf{h}(v_\mathrm{s})\right)_j=\frac{1}{2}v_\mathrm{s}\int_{C_j}\ee^{-\ii h^\infty(z;v_\mathrm{s})}\,\dd z = \frac{1}{2}v_\mathrm{s}f_j(v_\mathrm{s})
\end{equation}
according to \eqref{eq:PN-f-define}.  Thus, $\mathbf{m}^\top\mathbf{h}(v_\mathrm{s})=\frac{1}{2}v_\mathrm{s}\mathbf{m}^\top\mathbf{f}(v_\mathrm{s})$, so \eqref{eq:PN-Sigma1-Sigma0-no-limit} becomes
\begin{equation}
    \frac{\Sigma_1(v_\mathrm{s})}{\Sigma_0(v_\mathrm{s})}=\frac{\mathbf{m}^\top\mathbf{g}(v_\mathrm{s})}{v_\mathrm{s}\mathbf{m}^\top\mathbf{f}(v_\mathrm{s})} - \frac{\mathbf{p}^\top\boldsymbol{\phi}}{v_\mathrm{s}}.
\label{eq:PN-Sigma1-Sigma0-again}
\end{equation}

Next, for the term $\mathbf{m}^\top\mathbf{g}(v_\mathrm{s})$ we start by calculating from \eqref{eq:PN-f-define} that
\begin{equation}
\begin{split}
    v_\mathrm{s}f_j'(v_\mathrm{s}) &= \frac{1}{2}\ii v_\mathrm{s}\int_{C_j}z\ee^{-\ii h^\infty(z;v_\mathrm{s})}\,\dd z\\
    &=\int_{C_j}z\frac{\dd}{\dd z}\ee^{-\ii h^\infty(z;v_\mathrm{s})}\,\dd z+\ii\int_{C_j}zu_0(z)\ee^{-\ii h^\infty(z;v_\mathrm{s})}\,\dd z,\quad j=1,\dots,N.
\end{split}
\end{equation}
Integrating by parts and using \eqref{eq:PN-rational-IC} and \eqref{eq:PN-f-define} gives
\begin{equation}
v_\mathrm{s}f_j'(v_\mathrm{s})+f_j(v_\mathrm{s}) = \ii \int_{C_j}\ee^{-\ii h^\infty(z;v_\mathrm{s})}\sum_{n=1}^N\left(\frac{c_nz}{z-p_n}+\frac{c_n^*z}{z-p_n^*}\right)\,\dd z.
\end{equation}
Using the assumed condition \eqref{eq:sum-of-residues-zero} on $c_1,\dots,c_N$, 
\begin{equation}
    \sum_{n=1}^N\left(\frac{c_nz}{z-p_n}+\frac{c_n^*z}{z-p_n^*}\right)=\sum_{n=1}^N\left(\frac{c_np_n}{z-p_n}+\frac{c_n^*p_n^*}{z-p_n^*}\right).
\end{equation}
Therefore, again using that $\mathbf{m}^\top \mathbf{E}(v_\mathrm{s})=\mathbf{0}^\top$ and recalling \eqref{eq:PN-g-and-h} we deduce the identity
\begin{equation}
v_\mathrm{s}\mathbf{m}^\top\mathbf{f}'(v_\mathrm{s}) +\mathbf{m}^\top\mathbf{f}(v_\mathrm{s}) = \ii\mathbf{m}^\top\mathbf{g}(v_\mathrm{s}).
\end{equation}
Using this in \eqref{eq:PN-Sigma1-Sigma0-again} gives
\begin{equation}
    \frac{\Sigma_1(v_\mathrm{s})}{\Sigma_0(v_\mathrm{s})}=\frac{1}{\ii v_\mathrm{s}}-\ii\frac{\mathbf{m}^\top\mathbf{f}'(v_\mathrm{s})}{\mathbf{m}^\top\mathbf{f}(v_\mathrm{s})} -\frac{\mathbf{p}^\top\boldsymbol{\phi}}{v_\mathrm{s}}.
\end{equation}
In the notation of \cite{MillerW16a}, we have $\mathbf{b}^<(\lambda)=-\lambda\mathbf{f}(-2\lambda)$, and combining \cite[Proposition 3]{MillerW16a} with \cite[Corollary 5]{MillerW16a} gives~\eqref{eq:PN-Sigma1-over-Sigma0-final}.
\end{proof}

To complete the proof of \eqref{eq:soliton-rational} we only need to obtain an alternate representation of the phase constant $\gamma_\mathrm{s}$, which is the content of the following lemma.
\begin{lemma}
    Let $\varphi$ be a Lax eigenfunction for the eigenvalue $\lambda_\mathrm{s}<0$, where $v_\mathrm{s}=-2\lambda_\mathrm{s}$, normalized so that $\|\varphi\|_2^2=1$.  Then recalling the operator $\mathcal{X}^*$ from Definition~\ref{def:Xstar},
    \begin{equation}
        \gamma_\mathrm{s}=-\langle \mathcal{X}^*\varphi,\varphi\rangle.
    \end{equation}
    \label{lem:PN-phase-constant}
\end{lemma}
Recall the domain of $\mathcal{X}^*$ in the form \eqref{eq:domain-Xstar-alternate} and its action as given in \eqref{eq:Xstar-action-x}.  
Since $\varphi$ is proportional to $\Phi$ with $\Phi(x)=x^{-1}+O(x^{-2})$ because $\mathrm{dim}(\mathrm{Ker}(\mathcal{L}_{u_0}-\lambda_\mathrm{s}\,\mathrm{Id}))=1$ \cite{Wu16a}, we certainly have $\Phi,\varphi\in\mathrm{Dom}(\mathcal{X}^*)$.
\begin{proof}
    The difference of the boundary values taken by $\lambda\mapsto W(x;\lambda)$ as $\lambda$ approaches the positive real axis from above and below can be expressed in the form~\cite[Equation (7)]{MillerW16a}
    \begin{equation}
        \lim_{\delta\downarrow 0}\left(W(x;\lambda+\ii\delta)-W(x;\lambda-\ii\delta)\right) = \beta(\lambda)m_+(x;\lambda)
        ,\quad\lambda>0,
    \end{equation}
    where $\beta(\lambda)$ is the \emph{reflection coefficient} and $x\mapsto m_+(x;\lambda)$ is the generalized eigenfunction defined in \eqref{def:m}.
    According to \cite{Wu17,BlackstoneGGM24a}, there is a function $\Gamma(\lambda)$ (denoted $\ell_+(\lambda)$ in~\cite{BlackstoneGGM24a}) such that 
    \begin{equation}
    m_-(x;\lambda)=\Gamma(\lambda)m_+(x;\lambda),\quad |\Gamma(\lambda)|=1,\quad\lambda>0.    
    \label{eq:Gamma-of-lambda}
    \end{equation}
    Then
    \begin{equation}
        \lim_{\delta\downarrow 0}\left(W(x;\lambda+\ii\delta)-W(x;\lambda-\ii\delta)\right)  
    =\frac{\beta(\lambda)}{\Gamma(\lambda)}m_-(x;\lambda)
        ,\quad\lambda>0.
    \end{equation}
    It follows from the Plemelj formula and partial fraction expansion using \eqref{eq:PN-W-Laurent} that 
    $W(x;\lambda)$ can be written in the form (see also \cite[Equation (20)]{KaupM98}):
    \begin{equation}
        W(x;\lambda) = 1+\frac{1}{2\pi\ii}\int_0^\infty\frac{\beta(\zeta)m_-(x;\zeta)\,\dd\zeta}{\Gamma(\zeta)(\zeta-\lambda)}-\sum_{n=1}^{N_\mathrm{s}}\frac{\ii\Phi_n(x)}{\lambda-\lambda_n},
    \end{equation}
    where $\Phi_n(x)=x^{-1}+O(x^{-2})$ are the normalized eigenfunctions corresponding to the $N_\mathrm{s}<\infty$ (by \cite{Wu16a})  eigenvalues $\lambda_n<0$.  Subtracting off the term with $\lambda_n=\lambda_\mathrm{s}$, a given Lax eigenvalue with eigenfunction $\Phi$, and taking the limit $\lambda\to\lambda_\mathrm{s}$ gives the identity
    \begin{equation}
        (x+\gamma_\mathrm{s})\Phi(x) = 1+\frac{1}{2\pi\ii}\int_0^\infty\frac{\beta(\zeta)m_-(x;\zeta)\,\dd\zeta}{\Gamma(\zeta)(\zeta-\lambda_\mathrm{s})}-\mathop{\sum_{n=1}^{N_\mathrm{s}}}_{\lambda_n\neq\lambda_\mathrm{s}}\frac{\ii\Phi_n(x)}{\lambda_\mathrm{s}-\lambda_n}.
    \label{eq:PN-Phi-identity-1}
    \end{equation}
    Since $\Phi\in\mathrm{Dom}(\mathcal{X}^*)$, using \eqref{eq:Xstar-action-x} with constant $\mathfrak{a}[\Phi]=1$, \eqref{eq:PN-Phi-identity-1} can be rewritten as
    \begin{equation}
        \mathcal{X}^*\Phi(x)= -\gamma_\mathrm{s}\Phi(x)-\mathop{\sum_{n=1}^{N_\mathrm{s}}}_{\lambda_n\neq\lambda_\mathrm{s}}\frac{\ii\Phi_n(x)}{\lambda_\mathrm{s}-\lambda_n}+\frac{1}{2\pi\ii}\int_0^\infty\frac{\beta(\zeta)m_-(x;\zeta)\,\dd\zeta}{\Gamma(\zeta)(\zeta-\lambda_\mathrm{s})}.
        \label{eq:XstarPhi-expand}
    \end{equation}
    The right-hand side of this identity represents the spectral decomposition of $\mathcal{X}^*\Phi\in L_+^2(\mathbb{R})$ relative to the self-adjoint operator $\mathcal{L}_\mathrm{u_0}$.  Indeed, \cite[Theorem 7.2]{BlackstoneGGM24a} implies that for each $f\in L_+^2(\mathbb{R})$, 
    \begin{equation}        f(x)=\sum_{n=1}^{N_\mathrm{s}}\frac{\langle f,\Phi_n\rangle}{\|\Phi_n\|_2^2}\Phi_n(x) + \frac{1}{2\pi}\int_0^\infty \widetilde{f}^-(\lambda)m_-(x;\lambda)\,\dd\lambda,
    \end{equation}
    where $m_-$ is defined in~\eqref{def:m} and $\widetilde{f}^-$ is the distorted Fourier transform of $f$ given by~\eqref{distortedFT}.
Comparing with \eqref{eq:XstarPhi-expand}, we obtain in particular
    $-\gamma_\mathrm{s}\|\Phi\|_2^2=\langle \mathcal{X}^*\Phi,\Phi\rangle$.
    Taking $\varphi:=\Phi/\|\Phi\|_2$ completes the proof.
\end{proof}

Combining Proposition~\ref{prop:sigma1/sigma0} with Lemma~\ref{lem:PN-phase-constant} shows that if $E(v_\mathrm{s})=0$ for $v_\mathrm{s}>0$, then $\Sigma_1(v_\mathrm{s})/\Sigma_0(v_\mathrm{s})=-\mathrm{Re}(\gamma_\mathrm{s})=\mathrm{Re}(\langle\mathcal{X}^*\varphi,\varphi\rangle)$ for an $L^2$-normalized eigenfunction $\varphi$ of $\mathcal{L}_{u_0}$ for eigenvalue $\lambda_\mathrm{s}=-\frac{1}{2}v_\mathrm{s}$.  Using this in \eqref{eq:PN-Piu-almost-done} completes the proof of \eqref{eq:soliton-rational}.

\section{Asymptotic expansion for large \texorpdfstring{$t>0$}{t>0} in negative velocity frames for rational initial data}
\label{app:radiation}

Now fix $x_0\in\mathbb{R}$ and $\lambda>0$, and set $x=x_0-2\lambda t$, i.e., the fixed velocity of the frame of reference is $v=-2\lambda<0$.
The goal of this Appendix is to prove~\eqref{eq:radiation-rational} using Theorem~\ref{thm:inversion-formula-rational} for rational data $u_0$ of the form \eqref{eq:PN-rational-IC} subject to the condition \eqref{eq:sum-of-residues-zero}.  To shorten the proof, we assume that none of the indices $n=1,\dots,N$ are exceptional, leaving the modifications for exceptional indices to the reader.
 
\subsection{Strategy}
We define a contour $\Gamma_N$ beginning and ending at $z=\infty\ee^{-\ii\pi/4}$ encircling the only the poles $p_1^*,\dots,p_N^*$ once in the counterclockwise direction. Let $\Omega$ denote the domain to the right of $\Gamma_N$ by its orientation.  $\Omega$ has two distinct components near $z=\infty\ee^{-\ii\pi/4}$; $\Omega^-$ (resp., $\Omega^+$) has the arc of $\Gamma_N$ incoming from (resp., outgoing to) $z=\infty\ee^{-\ii\pi/4}$ on its boundary.  We define for $j=1,\dots,N$ two contours $U_j^\pm$ beginning and ending at $z=\infty\ee^{-\ii\pi/4}$ in $\Omega^\pm$ and both encircling only the poles $p_1,\dots,p_j$ once in the counterclockwise direction.
See
Figure~\ref{fig:Uj}. 
\begin{figure}[h]
\centering
\begin{tikzpicture}[>=Stealth, scale=0.75]

  \draw[gray, thin, ->] (-8,0) -- (8,0);
  \draw[gray, thin, ->] (0,-2.5) -- (0,2.5);
  \fill (0.5,1.8) circle (2.5pt); \node[below,right] at (0.5,1.8) {$p_1$};
  \fill (2,1) circle (2.5pt); \node[above,right] at (2,1) {$p_2$};
  \fill (5,1) circle (2.5pt); \node[above,right] at (5,1) {$p_3$};

  \fill (0.5,-1.8) circle (2.5pt); \node[below,right] at (0.5,-1.8) {$p_1^*$};
  \fill (2,-1) circle (2.5pt);
  \node[below] at (2,-1) {$p_2^*$};
  \fill (5,-1) circle (2.5pt); \node[below,left] at (5,-1) {$p_3^*$};

    \draw[red, thick, decoration={markings, mark=at position 0.75 with {\arrow[red,thick]{Stealth}}},
  postaction={decorate}
  ] plot [smooth] coordinates {(0,-3) (-3,0) (3,0.5) (2,2.5) (-1,2.5) (-4,0) (-1,-3)};
  
      \draw[olive, thick, decoration={markings, mark=at position 0.3 with {\arrow[olive,thick]{Stealth}}},
  postaction={decorate}
  ] plot [smooth] coordinates {(9,-3) (6,0) (2,2) (-1,2) (1,1) (5,0) (8,-3)};

        \draw[blue, thick, decoration={markings, mark=at position 0.55 with {\arrow[blue,thick]{Stealth}}},
  postaction={decorate}
  ] plot [smooth] coordinates {(7.6,-3) (5,-0.4) (1,-0.4) (0.2,-2) (1,-3)};

    \draw[gray] node at (7,2) {$\boxed{z}$};

  \node[above,right] at (-3.8,2) {$\color{red}U_2^+$};

    \node[above,right] at (3.8,1.7) {$\color{olive}U_2^-$};

    \node[below] at (2.8,-0.3) {$\color{blue}\Gamma_N$};

    \node[left] at (1,-3.2) {$\color{blue}\Omega^+$};

    \node[right] at (8,-3.2) {$\color{blue}\Omega^-$};
  
\end{tikzpicture}
\caption{The contours $U_j^+$, $U_j^-$, and $\Gamma_N$, and the asymptotic domains $\Omega^\pm$.
}\label{fig:Uj}
\end{figure}
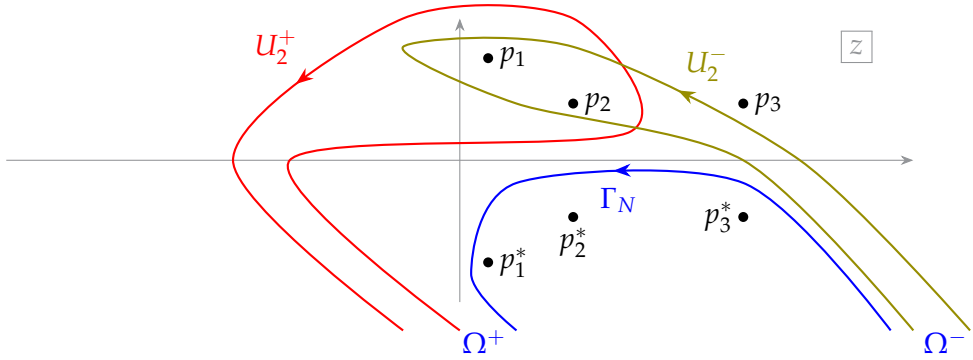

We recall from~\eqref{eq:PN-exph-infty} the function $\ee^{-\ii h^\infty(z;v)}$ which in this Appendix will be evaluated for $v=-2\lambda<0$.  To define it on the new contours, we take its value continued along $C_0$ to large $|z|$ with $\arg(z)=-\pi/4$, and then set that as the initial value in the asymptotic domain $\Omega^-$ to define it on $\{U_j^-\}_{j=1}^N$ and $\Gamma_N$ originating from $\Omega^-$ by continuation along each contour.  To define $\ee^{-\ii h^\infty(z;v)}$ on $\{U_j^+\}_{j=1}^N$, we first continue from $\Omega^-$ to $\Omega^+$ along $\Gamma_N$ and then continue along $U_j^+$, $j=1,\dots,N$.  We apply the same continuation procedure to define $\ee^{-\ii h(z;t,x)}$ on the new contours.

For $f\in L^2_+(\mathbb{R})$ admitting a meromorphic extension to $\mathbb{C}$ without poles on any of the contours and with subexponential growth as $z\to\infty$ with $\arg(z)=-\pi/4$, we denote by $\MA^\infty_f(\lambda)$, $\lambda>0$, the matrix with coefficients:
\begin{equation}\label{eq:A-f}
\MA^\infty_f(\lambda)
	=
\begin{bmatrix}\displaystyle 
	\int_{\Gamma_N}\ee^{-\ii h^{\infty}(z;-2\lambda)}f(z)\dd z 	
	& 	\displaystyle \int_{\Gamma_N}\frac{\ee^{-\ii h^{\infty}(z;-2\lambda)}\,\dd z}{z-p_1} 
	& \dots
	& 	\displaystyle \int_{\Gamma_N}\frac{\ee^{-\ii h^{\infty}(z;-2\lambda)}\,\dd z}{z-p_N} \\
	\displaystyle \int_{{U}_1^+}\ee^{-\ii h^{\infty}(z;-2\lambda)}f(z)\dd z 
	& 	\displaystyle \int_{{U}_1^+}\frac{\ee^{-\ii h^{\infty}(z;-2\lambda)}\,\dd z}{z-p_1}
	& \dots
	& 	\displaystyle \int_{{U}_{1}^+}\frac{\ee^{-\ii h^{\infty}(z;-2\lambda)}\,\dd z}{z-p_N} \\
	\vdots 
	&\vdots 
	&
	&\vdots
	\\  \displaystyle 
	\int_{{U}_N^+}\ee^{-\ii h^{\infty}(z;-2\lambda)}f(z)\dd z 
	& 	\displaystyle \int_{{U}_N^+}\frac{\ee^{-\ii h^{\infty}(z;-2\lambda)}\,\dd z}{z-p_1}
	& \dots
	& \displaystyle 	\int_{{U}_N^+}\frac{\ee^{-\ii h^{\infty}(z;-2\lambda)}\,\dd z}{z-p_N}\\
\end{bmatrix}.
\end{equation}
We also define
\begin{equation}\label{eq:D-negative}
\ME^\pm(\lambda)
	=\begin{bmatrix}\displaystyle 
		\int_{U_1^\pm}\frac{\ee^{-\ii h^{\infty}(z;-2\lambda)}\,\dd z}{z-p_1}
	& \dots
	& 	\displaystyle \int_{U_{1}^\pm}\frac{\ee^{-\ii h^{\infty}(z;-2\lambda)}\,\dd z}{z-p_N} \\
	\vdots 
	&
	&\vdots
	\\  
		\displaystyle \int_{U_N^\pm}\frac{\ee^{-\ii h^{\infty}(z;-2\lambda)}\,\dd z}{z-p_1}
	& \dots
	& 	\displaystyle \int_{U_N^\pm}\frac{\ee^{-\ii h^{\infty}(z;-2\lambda)}\,\dd z}{z-p_N}\\
\end{bmatrix},\quad\lambda>0.
\end{equation}

Our proof decomposes into two claims.  Let $c:=c_1+\cdots+c_N$, and note that \eqref{eq:sum-of-residues-zero} gives $\mathrm{Re}(c)=0$.
\begin{proposition}[Asymptotic expansion]\label{prop:pi-frac}
Fix $\lambda>0$.  Then  
\begin{equation}\label{eq:pi-frac}
\Pi u(t,x_0-2t\lambda)
	=- \ee^{\ii\pi/4}\frac{\ee^{-\ii t \lambda^2+\ii x_0\lambda}}{2\sqrt{\pi t}}
\frac{\det (\MA^\infty_{\Pi u_0}(\lambda))}{\det (\ME^-(\lambda))}+ O(t^{-1}),\quad t\to+\infty.
\end{equation}
The $O(t^{-1})$ error term is uniform for $x_0\in\R$ in any compact set.
\end{proposition}
We will prove that $\det (\ME^-(\lambda))\neq 0$ in Lemma~\ref{lem:detD-nonzero}.

\begin{proposition}[Formula for the distorted Fourier transform]\label{prop:distordu-frac}
Assume the condition \eqref{eq:sum-of-residues-zero}. If $f\in L^2_+(\mathbb{R})$ admits a meromorphic extension to $\mathbb{C}$ that is sub-exponentially bounded in the direction $z\to\infty\ee^{-\ii\pi/4}$, then
\begin{equation}\label{eq:distordu-frac}
\widetilde{f}^-(\lambda)
	=-
    \frac{\det(\MA^\infty_f(\lambda))}{\det(\ME^-(\lambda))},\quad\lambda>0.
\end{equation}
\end{proposition}
We prove  Proposition~\ref{prop:pi-frac} in Section~\ref{subsec:asymptotic-y<0} and Proposition~\ref{prop:distordu-frac} in Sections~\ref{subsec:distorted} and~\ref{ss:distorted}.
Specializing Proposition~\ref{prop:distordu-frac} to $f=\Pi u_0$, we conclude that~\eqref{eq:radiation-rational} holds.

\begin{remark}[Integral identities]
\label{rem:integral-identities}
The following identities hold for any meromorphic function $G$ without poles on any of the contours and for which all integrals converge:
\begin{equation}\label{eq:Uj+-}
\int_{U_j^-}\ee^{-\ii h^{\infty}(z;-2\lambda)}G(z)\dd z=\int_{U_j^+}\ee^{-\ii h^{\infty}(z;-2\lambda)}G(z)\dd z+(1-\ee^{2\pi(c_1+\dots+c_j)})\int_{\Gamma_N}\ee^{-\ii h^{\infty}(z;-2\lambda)}G(z)\dd z,
\end{equation}
for $\lambda>0$ and 
\begin{equation}
- (1-\ee^{2\pi(c_1+\dots+c_j)})
\int_{C_0}\ee^{-\ii h(z;t,x)}G(z)\dd z
	+\int_{\uC_j}\ee^{-\ii h(z;t,x)}G(z)\dd z	
	=\int_{U_j^-}\ee^{-\ii h(z;t,x)}G(z)\dd z
    \label{eq:C0CjUj-}
\end{equation}
for $x\in\R$ and $t>0$.  The identity \eqref{eq:Uj+-} implies that if $\lambda>0$, $\det(\mathbf{A}_f^\infty(\lambda))=\det(\widetilde{\mathbf{A}}_f^\infty(\lambda))$ where $\widetilde{\mathbf{A}}_f^\infty(\lambda)$ is the same as $\mathbf{A}_f^\infty(\lambda)$ except that all the contours of integration in rows $j=2,\dots,N+1$ are changed from $U_{j-1}^+$ to $U_{j-1}^-$.  Note that $u_0-\Pi u_0$ is a linear combination of $\{(z-p_n)^{-1}\}_{n=1}^N$.  This fact together with the identity \eqref{eq:C0CjUj-} implies that if $t>0$, $\det(\mathbf{A}(t,x))=\det(\widetilde{\mathbf{A}}(t,x))$ and $\det(\mathbf{B}(t,x))=\det(\widetilde{\mathbf{B}}(t,x))$, where 
\begin{itemize}
    \item $\widetilde{\mathbf{A}}(t,x)$ is the same as $\mathbf{A}(t,x)$ except that in rows $j=2,\dots,N+1$ the integration contour $C_{j-1}$ is replaced with $U_{j-1}^-$ and in the integrals in the first column the function $u_0$ is replaced with $\Pi u_0$, and
    \item $\widetilde{\mathbf{B}}(t,x)$ is the same as $\mathbf{B}(t,x)$ except that in rows $j=2,\dots,N+1$ the integration contour $C_{j-1}$ is replaced with $U_{j-1}^-$.
\end{itemize}
\end{remark}

\subsection{Long-time asymptotic expansion}\label{subsec:asymptotic-y<0}
We show that as $t\to+\infty$, $\widetilde{B}_{11}(t,x_0-2t\lambda)$ is proportional to $\sqrt{t}$ while all other elements of $\widetilde{\mathbf{A}}(t,x_0-2t\lambda)$ and of $\widetilde{\mathbf{B}}(t,x_0-2t\lambda)$ remain bounded; moreover the $(1,1)$ principal minor determinant of $\widetilde{\mathbf{B}}(t,x_0-2t\lambda)$ is bounded away from zero.  
This produces
the desired  $t^{-1/2}$ rate of decay in the long-time asymptotics.

\begin{proof}[Proof of Proposition~\ref{prop:pi-frac}]
Consider the first rows of $\widetilde{\mathbf{A}}(t,x_0-2t\lambda)$ and $\widetilde{\mathbf{B}}(t,x_0-2t\lambda)$.  Define the line $D_-$ through $z=x=x_0-2t\lambda$ with $(z-x)\ee^{\ii\pi/4}\in\mathbb{R}$, which lies to the left (lesser real part) of all of the contours for $t$ large enough, so that for either $G(z)=\Pi u_0(z)$ or $G(z)=(z-p_n)^{-1}$, $n=1,\dots,N$,
\begin{equation}
\int_{C_0}\ee^{-\ii h(z;t,x)}G(z)\dd z= \int_{D_-}\ee^{-\ii h(z;t,x)}G(z)\dd z
- \int_{\Gamma_N}\ee^{-\ii h(z;t,x)}G(z)\dd z.
\end{equation}
Using the same method as the positive-velocity case,
 we note that for $x=x_0-2\lambda t$ and $t\to+\infty$,
\begin{equation}
\begin{aligned}
\int_{D_-}\ee^{-\ii h(z;t,x)}\dd z
	&=2\sqrt{\pi}\ee^{-\ii\pi/4}t^{1/2}
    +O(t^{-1/2}),
\\
\int_{D_-}\ee^{-\ii h(z;t,x)}\Pi u_0(z)\dd z
	&=O(t^{-1/2}),
\\
\int_{D_-}\ee^{-\ii h(z;t,x)}\frac{\dd z}{z-p_n}
	&=O(t^{-1/2}),
	\quad 1\leq n\leq N,
\end{aligned}
\label{eq:D-integrals}
\end{equation}
because the product in \eqref{eq:PN-exph} converges to $1$ as $x\to-\infty$ uniformly on $D_-$.
Conversely, for $x=x_0-2\lambda t$ and $t\to+\infty$, all the integrals over $\Gamma_N$ remain bounded:
\begin{equation}
\int_{\Gamma_N}\ee^{-\ii h(z;t,x)}G(z)\dd z
	=\ee^{-\ii \lambda^2 t+\ii x_0\lambda}\int_{\Gamma_N}\ee^{-\ii h^{\infty}(z;-2\lambda)}G(z)\dd z+O(t^{-1}),
\end{equation}
and similarly, we have
\begin{equation}
\int_{U_j^-}\ee^{-\ii h(z;t,x)}G(z)\dd z
	=\ee^{-\ii \lambda^2 t+\ii x_0\lambda}\int_{U_j^-}\ee^{-\ii h^{\infty}(z;-2\lambda)}G(z)\dd z+O(t^{-1}).
\end{equation}
Hence one can neglect all the contributions of the integrals over the contour $D_-$ compared to those over $\Gamma_N$ or $U_j^-$, except in the element $\widetilde{B}_{11}(t,x_0-2\lambda t)$, for which the first equation in \eqref{eq:D-integrals} gives the leading contribution.

Putting the results together,
we get:
\begin{multline}
-\ee^{\ii(N+1)( \lambda^2 t- x_0\lambda)}\left(\det\left(\widetilde{\mathbf{A}}(t,x)\right)+O(t^{-1/2})\right)=\\
	 \begin{vmatrix}\displaystyle 
	\int_{\Gamma_N}\ee^{-\ii h^{\infty}(z;-2\lambda)}\Pi u_0(z)\dd z 	
	& 	\displaystyle \int_{\Gamma_N}\frac{\ee^{-\ii h^{\infty}(z;-2\lambda)}\,\dd z}{z-p_{1}}
	& \dots
	& 	\displaystyle \int_{\Gamma_N}\frac{\ee^{-\ii h^{\infty}(z;-2\lambda)}\,\dd z}{z-p_{N}}  \\ \displaystyle
	\int_{{U_1^-}}\ee^{-\ii h^{\infty}(z;-2\lambda)}\Pi u_0(z)\dd z 
	& 	\displaystyle \int_{{U_1^-}}\frac{\ee^{-\ii h^{\infty}(z;-2\lambda)}\,\dd z}{z-p_{1}}
	& \dots
	& 	\displaystyle \int_{{U_1^-}}\frac{\ee^{-\ii h^{\infty}(z;-2\lambda)}\,\dd z}{z-p_{N}} \\
	\vdots 
	&\vdots 
	&
	&\vdots
	\\  \displaystyle 
	\int_{{U_N^-}}\ee^{-\ii h^{\infty}(z;-2\lambda)}\Pi u_0(z)\dd z 
	& 	\displaystyle \int_{{U_N^-}}\frac{\ee^{-\ii h^{\infty}(z;-2\lambda)}\,\dd z}{z-p_{1}}
	& \dots
	& \displaystyle 	\int_{{U_N^-}}\frac{\ee^{-\ii h^{\infty}(z;-2\lambda)}\,\dd z}{z-p_{N}}\\
\end{vmatrix}.
\end{multline}
We recognize the right-hand side as $\det(\widetilde{\mathbf{A}}^\infty_{\Pi u_0}(\lambda))$, which by Remark~\ref{rem:integral-identities} coincides with $\det(\mathbf{A}^\infty_{\Pi u_0}(\lambda))$.
Similarly, we get
\begin{multline}\label{eq:B<0-2}
-\ee^{\ii (N+1)( \lambda^2 t- x_0\lambda)}\left(\det \left(\widetilde{\mathbf{B}}(t,x)\right)+O(t^{-1/2})\right)=\\
	\begin{vmatrix}
	-2\sqrt{\pi t}\ee^{-\ii\pi/4}\ee^{\ii (\lambda^2t-x_0\lambda)}
	& 	\displaystyle \int_{\Gamma_N}\frac{\ee^{-\ii h^{\infty}(z;-2\lambda)}\,\dd z}{z-p_{1}}
	& \dots
	& 	\displaystyle \int_{\Gamma_N}\frac{\ee^{-\ii h^{\infty}(z;-2\lambda)}\,\dd z}{z-p_{N}} \\ \displaystyle 
	\int_{{U_1^-}}\ee^{-\ii h^{\infty}(z;-2\lambda)}\dd z 
	& 	\displaystyle \int_{{U_1^-}}\frac{\ee^{-\ii h^{\infty}(z;-2\lambda)}\,\dd z}{z-p_{1}}
	& \dots
	& \displaystyle	\int_{{U_{1}^-}}\frac{\ee^{-\ii h^{\infty}(z;-2\lambda)}\,\dd z}{z-p_{N}} \\
	\vdots 
	&\vdots 
	&
	&\vdots
	\\  
	\displaystyle \int_{{U_N^-}}\ee^{-\ii h^{\infty}(z;-2\lambda)}\dd z 
	& 	\displaystyle \int_{{U_N^-}}\frac{\ee^{-\ii h^{\infty}(z;-2\lambda)}\,\dd z}{z-p_{1}}
	& \dots
	& 	\displaystyle \int_{{U_N^-}}\frac{\ee^{-\ii h^{\infty}(z;-2\lambda)}\,\dd z}{z-p_{N}}\\
\end{vmatrix}.
\end{multline}
Combining these and using $\det(\mathbf{A}(t,x))=\det(\widetilde{\mathbf{A}}(t,x))$ and $\det(\mathbf{B}(t,x))=\det(\widetilde{\mathbf{B}}(t,x))$ from Remark~\ref{rem:integral-identities} completes the proof of \eqref{eq:pi-frac}.
\end{proof}

\subsection{Spectral theory of the Lax operator}\label{subsec:distorted}
In this section, we note that the condition \eqref{eq:sum-of-residues-zero}, i.e., $\mathrm{Re}(c)=0$, implies that $\diamond\, u_0(\diamond)\in L^2(\mathbb{R})$, and we prove Proposition~\ref{prop:distordu-frac}.
We first justify that the determinant of $\mathbf{E}^\pm(\lambda)$ does not vanish for $\lambda>0$.
\begin{lemma}\label{lem:detD-nonzero}
For every $\lambda>0$, we have $\det(\mathbf{E}^\pm(\lambda))\neq 0$. 
\end{lemma}

\begin{proof}
We show that every $\lambda>0$ such that $\det(\mathbf{E}^\pm(\lambda))=0$ is an eigenvalue of $\mathcal{L}_{u_0}$, in contradiction with the fact that all eigenvalues are negative \cite{Wu16a}. The rest of the proof is very similar to the proof of~\cite[Proposition 1]{MillerW16a}, therefore we only recall the ideas.
The eigenvalue equation $\mathcal{L}_{u_0}g=\lambda g$ can be written as an ordinary differential equation on the complex $x$-plane:
\begin{equation}
-\ii g'(x)-u_0(x)g(x)-\lambda g(x)=-\sum_{k=1}^N\frac{c_ng(p_n)}{x-p_n}.
\label{eq:spectral-ODE}
\end{equation}
Integrating yields two particular solutions $g=g_\pm(x)$ that have to satisfy
\begin{equation}
g_\pm(x)
    =\ee^{\ii h^{\infty}(x;-2\lambda)}\int_{\ell_0^\pm(x)}\ee^{-\ii h^{\infty}(x;-2\lambda)}\sum_{n=1}^N\frac{v_n^\pm(\lambda)}{z-p_n}\dd z,\quad v_n^\pm(\lambda)=-\ii c_ng_\pm(p_n),\quad n=1,\dots,N,
\end{equation}
where $\ell_0^\pm(x)$ is a contour from $z=\infty\ee^{-\ii\pi/4}\in\Omega^\pm$  to $x\in\C_+$. As in the proof of~\cite[Corollary 1]{MillerW16a}, one can show that $g_\pm$ is analytic on the upper half-plane if and only if
\begin{equation}\label{eq:Uj-f}
\int_{U_j^\pm}\ee^{-\ii h^{\infty}(x;-2\lambda)}\sum_{n=1}^N\frac{v_n^\pm(\lambda)}{z-p_n}\dd z
    =0,
    \quad j=1,\dots,N.
\end{equation}
If this holds for some nonzero vector $\mathbf{v}^\pm(\lambda)=(v_1^\pm(\lambda),\dots,v_N^\pm(\lambda))^\top\in\C^N$, one can check that $g_\pm\in L_+^2(\R)$ defines an eigenfunction of $\mathcal{L}_{u_0}$ associated to the eigenvalue $\lambda>0$, which is impossible.
It remains to observe that if $\lambda>0$, the existence of $\mathbf{v}^\pm(\lambda)\neq\mathbf{0}$ satisfying \eqref{eq:Uj-f} is equivalent to the assertion that $\det(\mathbf{E}^\pm(\lambda))=0$.
\end{proof}

We recall the relationship \eqref{eq:Gamma-of-lambda} between the generalized eigenfunctions $m_\pm(x;\lambda)$ defined in~\eqref{def:m}.
We also recall~\cite[Theorem 26]{BlackstoneGGM24a} that the spectral measure $\dd\mu_f$ of $f$ for $\mathcal{L}_{u_0}$ satisfies
\begin{equation}\label{eq:dmuf}
\dd\mu_f(\lambda)
	=\frac{1}{2\ii\pi}\lim_{\epsilon\to 0^+}\langle \varphi(\cdot;\lambda+\ii\epsilon)-\varphi(\cdot;\lambda-\ii\epsilon), f \rangle\,\dd\lambda,
\end{equation}
where given $\epsilon>0$, $\varphi(\cdot;\lambda\pm \ii\epsilon)\in L^2(\mathbb{R})$ are given by
\begin{equation}
\varphi(\cdot;\lambda\pm \ii\epsilon):=(\mathcal{L}_{u_0}-(\lambda\pm\ii \epsilon)\,\mathrm{Id})^{-1}f.
\end{equation}
Moreover, for either of the two distorted Fourier transforms (which have the same modulus),
\begin{equation}\label{eq:dmuf2}
\dd \mu_f(\lambda)
	=\frac{1}{2\pi}|\widetilde{f}^\pm(\lambda)|^2\,\dd \lambda.
\end{equation}

For rational $u_0$ subject to the condition \eqref{eq:sum-of-residues-zero}, we observe that the characterizations and $m_\pm(\diamond;\lambda)$ and $\varphi(\diamond;\lambda\pm\ii\epsilon)$ are very similar to ~\cite[Proposition 1]{MillerW16a}.
In particular, we have the following lemmas.  
\begin{lemma}[Generalized eigenfunction]\label{lem:m}
Let $\lambda>0$, and assume that $u_0$ is rational of the form \eqref{eq:PN-rational-IC} with coefficients $c_1,\dots,c_N$ subject to the condition \eqref{eq:sum-of-residues-zero}.  Then with $\ee^{\ii h^\infty(x;-2\lambda)}$ being the branch of \eqref{eq:PN-exph-infty} for which the product tends to $1$ as $x\to-\infty$ analytically continued along $x\in\R$,
\begin{equation}
m_\pm(x;\lambda)= K_\pm\ee^{\ii  h^\infty(x;-2\lambda)}
\left(1	+\int_{{\ell_0^\mp(x)}} \ee^{-\ii h^\infty(z;-2\lambda)}\sum_{n=1}^N \frac{w_n^\pm(\lambda)}{z-p_n}\dd z\right) 
\label{eq:mpm-contours}
\end{equation}
with 
\begin{equation}
    K_\pm:=\lim_{x\to\pm\infty}\ee^{\ii\lambda x}\ee^{-\ii h^\infty(x;-2\lambda)},
    \label{eq:Kpm-consts}
\end{equation}
so that $K_-=1$ and $K_+=\ee^{2\pi\ii\mathrm{Im}(c)}$, $c=c_1+\dots+c_N$,
where $\ell_0^\pm(x)$ are the contours described in the proof of Lemma~\ref{lem:detD-nonzero}.  
The coefficients $w_n^\pm(\lambda)$ are given by
\begin{equation}\label{eq:wk}
w_n^\pm(\lambda)=-\frac{\det({\ME^\mp}(\lambda)\longleftarroww{n} \Md)}{\det ({\ME^\mp}(\lambda))},
\end{equation}
where $\Md$ was defined in~\eqref{def:d-vector} as
\(
\Md=(1-\ee^{2\pi c_1},1-\ee^{2\pi (c_1+c_2)},\dots, 1-\ee^{2\pi (c_1+\cdots+c_N)})^\top.
\)
\end{lemma}

\begin{proof}
For simplicity we only focus on $m_-(\diamond;\lambda)$. We seek solutions $m\in L^{\infty}(\mathbb{R})$ to $(\mathcal{L}_{u_0}-\lambda\,\mathrm{Id})m=0$, by following the approach in the proof of Lemma~\ref{lem:detD-nonzero}, but now solving the differential equation \eqref{eq:spectral-ODE} with an arbitrary integration constant denoted $\mu\in\C$:  
\begin{multline}
m(x)= \mu\ee^{\ii  h^\infty(x;-2\lambda)}\left( 1	+\int_{{\ell_0^+(x)}} \ee^{-\ii h^\infty(z;-2\lambda)}\sum_{n=1}^N \frac{w_n(\lambda)}{z-p_n}\dd z\right),\quad w_n(\lambda):=-\frac{\ii c_nm(p_n)}{\mu},\\ n=1,\dots,N.
\end{multline}
The second term in parentheses tends to zero as $x\to-\infty$, and since $\mathrm{Re}(c_1+\cdots+c_N)=0$ according to \eqref{eq:sum-of-residues-zero}, it is easy to check that the limit $K_-$ in \eqref{eq:Kpm-consts} exists, and that
\begin{equation}
    \lim_{x\to-\infty}m(x)\ee^{-\ii\lambda x}=\frac{\mu }{K_-}.
\end{equation}
Hence we take $\mu=K_-$ and then $m(x)=m_-(x;\lambda)$, so we have proved \eqref{eq:mpm-contours}.

We claim that the function $m_-(\diamond,\lambda)$ is analytic on $\mathbb{C}_+$ if and only if  the coefficients $w_n(\lambda)$ satisfy the system of equations  
\begin{equation}\label{eq:m-holo}
 1-\ee^{2\pi (c_1+\cdots+c_j)}+ \int_{{U_j^+}} \ee^{-\ii h^\infty(z;-2\lambda)}\sum_{n=1}^N \frac{w_n^-(\lambda)}{z-p_n}\dd z=0,
\quad j=1,\dots, N.
\end{equation}
Using this claim, it only remains to solve the system of equations~\eqref{eq:m-holo} thanks to Cramer's rule to get~\eqref{eq:wk}.

The proof of the claim follows that of~\cite[Proposition 1]{MillerW16a}, so we only recall the ideas. Fixing $j=1,\dots,N$, we require that $m_-(\cdot,\lambda)$ be analytic at $p_j$. We first assume that $\operatorname{Re}(\ii c_j)<0$. Then $\ee^{-\ii h^\infty(z;-2\lambda)}$ vanishes to some positive fractional order at $z=p_j$ but is not single-valued there (because $j$ is not an exceptional index); therefore 
for $x$ in the vicinity of $p_j$:
\begin{multline}
m_-(x,\lambda)
	=K_-\ee^{\ii  h^\infty(x;-2\lambda)}
    \left(1 	+\int_{\ell_0^+(p_j)} \ee^{-\ii h^\infty(z;-2\lambda)}\sum_{n=1}^N \frac{w_n^-(\lambda)}{z-p_n}\dd z\right)\\
	+K_- 
    \ee^{\ii  h^\infty(x;-2\lambda)}\int_{p_j}^x\ee^{-\ii h^\infty(z;-2\lambda)}\sum_{n=1}^N \frac{w_n^-(\lambda)}{z-p_n}\dd z.
\end{multline}
It turns out that the last term on the right hand side is analytic at $p_j$. However, the first term in the right hand side is analytic at $p_j$ if and only if 
\begin{equation}
1	+\int_{\ell_0^+(p_j)} \ee^{-\ii h^\infty(z;-2\lambda)}\sum_{n=1}^N \frac{w_n^-(\lambda)}{z-p_n}\dd z
 	=0.
    \label{eq:first-pass}
\end{equation}
Now we introduce contours $V_j^\pm$ analogous to $U_j^\pm$ except that $V_j^\pm$ encircles only the pole $z=p_j$, and we define $\ee^{-\ii h^\infty(z;-2\lambda)}$ on these contours similarly as for $U_j^\pm$.  Then for $G(z)=(z-p_n)^{-1}$, $n=1,\dots,N$, we have the identity
\begin{equation}
    \int_{V_j^\pm}\ee^{-\ii h^\infty(z;-2\lambda)}G(z)\,\dd z = (1-\ee^{2\pi c_j})\int_{\ell_0^\pm(p_j)}\ee^{-\ii h^\infty(z;-2\lambda)} G(z)\,\dd z,
\end{equation}
so \eqref{eq:first-pass} can be written equivalently as
\begin{equation}
    1-\ee^{2\pi c_j} + \int_{V_j^+}\ee^{-\ii h^\infty(z;-2\lambda)}\sum_{n=1}^N\frac{w_n^-(\lambda)}{z-p_n}\,\dd z = 0.
    \label{eq:second-pass}
\end{equation}
In the case $\operatorname{Re}(\ii c_j)\geq 0$, we get the same conclusion after first integrating by parts.  Now we relate integration over $U_j^\pm$ to integration over $V_k^\pm$, $k=1,\dots,j$:
\begin{equation}
\begin{aligned}
    \int_{U_j^\pm}\ee^{-\ii h^\infty(z;-2\lambda)}\sum_{n=1}^N\frac{w_n^-(\lambda)}{z-p_n}\,\dd z &= \sum_{k=1}^j\ee^{2\pi(c_j+c_{j-1}+\cdots+c_{k+1})}\int_{V_k^\pm}\ee^{-\ii h^\infty(z;-2\lambda)}\sum_{n=1}^N\frac{w_n^-(\lambda)}{z-p_n}\,\dd z \\
    &=-\sum_{k=1}^j\ee^{2\pi(c_j+c_{j-1}+\cdots+c_{k+1})}(1-\ee^{2\pi c_k})\\
    &=-(1-\ee^{2\pi(c_1+\cdots+c_j)}),
\end{aligned}
\end{equation}
where on the second line we used \eqref{eq:second-pass} and on the third line we noted a telescoping sum.  This establishes \eqref{eq:m-holo} and completes the proof.
\end{proof}

Similarly, the eigenfunctions $\varphi(\diamond,\lambda\pm\ii \epsilon)$ for $\epsilon>0$ have the following limiting structure.
\begin{lemma}[Boundary limits of complex inverses]\label{lem:phi}
Let $f$ admit continuation to $\C$ as a meromorphic function with a sub-exponential bound as $z\to\infty\ee^{-\ii\pi/4}$.
The functions $\varphi(\diamond;\lambda \pm \ii\epsilon)$ have a limit as $\epsilon\to 0^+$ given by
\begin{equation}\label{eq:W+}
\varphi^{\pm}(x;\lambda)
	=\lim_{\epsilon\to 0^+}\varphi(x;\lambda\pm\ii\epsilon)
	=-\ii \ee^{\ii {h}^\infty(x;-2\lambda)}\int_{\ell_0^{\pm}(x)} \ee^{-\ii {h}^\infty(z;-2\lambda)}\left(-f(z)+\sum_{k=1}^N\frac{y_n^{\pm}(\lambda)}{z-p_n}\right) \dd z,
\end{equation}
with $\ee^{\ii h^\infty(x;-2\lambda)}$ as in Lemma~\ref{lem:m} and 
\begin{equation}
y_n^{\pm}(\lambda )=\frac{\det(   \ME^\pm(\lambda)\longleftarroww{n} \Mb^{\pm}(\lambda))}{\det(\ME^\pm(\lambda))}, \quad n=1,\dots,N
\label{eq:ypm-n}
\end{equation}
where $\Mb^\pm(\lambda)=(b_1^\pm(\lambda),\dots,b_N^\pm(\lambda))^\top$ is given by
\begin{equation}
b^{\pm}_j(\lambda)=\int_{U_j^{\pm}}
 \ee^{-\ii h^\infty(z;-2\lambda)}f(z)\dd z.
 \label{eq:bpm-elements}
\end{equation}
\end{lemma}
\begin{remark}
Define $q_0(\lambda)$ and $\Mq(\lambda)=(q_1(\lambda),\dots,q_N(\lambda))^\top$ by
\begin{equation}\label{eq:qk}
q_0(\lambda):=\int_{\Gamma_N}\ee^{-\ii {h}^\infty(z;-2\lambda)}f(z)\,\dd z,\quad
q_n(\lambda):=\int_{\Gamma_N}\frac{\ee^{-\ii h^\infty(z;-2\lambda)}\,\dd z}{z-p_n},
\quad n=1,\dots, N.
\end{equation}
Then using the identity~\eqref{eq:Uj+-} gives an equivalent representation of $y_n^-(\lambda)$ as
\begin{equation}\label{eq:vk+-}
y_n^-(\lambda)=\frac{\det(\ME^+(\lambda)+\Md\Mq(\lambda)^\top\longleftarroww{n}\Mb^+(\lambda)+q_0 \Md)}{\det(\ME^+(\lambda)+\Md\Mq(\lambda)^\top)},\quad n=1,\dots,N.
\end{equation}
\end{remark}

\begin{proof}
Similarly, we only give an idea of the proof. The functions $\varphi$ are not Jost solutions  but they satisfy~\cite[Equation (14)]{MillerW16a} with $\lambda w_0$ replaced by $-f$. 
We therefore obtain the equation satisfied by $\varphi(\cdot;\lambda\pm\ii\epsilon)$ on $\mathbb{C}_+$:
\begin{equation}\label{eq:Lax-W}
-\ii \varphi'(x)- u_0(x) \varphi(x)-(\lambda\pm \ii\epsilon) \varphi(x)=f(x)-\sum_{k=1}^N\frac{c_n \varphi(p_n)}{x-p_n},
\end{equation}
with the boundary condition $\varphi(x;\lambda\pm\ii\epsilon)\to 0$ as $|x|\to \infty$.

We can integrate equation~\eqref{eq:Lax-W} and obtain the convergent improper integral representation corresponding to~\cite[Equation (15)]{MillerW16a}: defining constants 
$y_n(\lambda\pm\ii\epsilon)= c_n\varphi(p_n;\lambda\pm\ii\epsilon)$ we obtain
\begin{equation}\label{eq:W}
\varphi(x;\lambda\pm\ii\epsilon)
	=-\ii \ee^{\ii {h}^\infty(x;-2(\lambda\pm \ii\epsilon))}\int_{\ell_0^\pm(x)} \ee^{-\ii {h}^\infty(z;-2(\lambda\pm \ii\epsilon))}\left(-f(z)+\sum_{k=1}^N\frac{y_n(\lambda\pm\ii\epsilon)}{z-p_n}\right) \dd z.
\end{equation}
Then we follow the same steps as in the proof of Lemma~\ref{lem:m}, using the idea from~\cite[Corollary 1]{MillerW16a}, and find that
the function $\varphi(\cdot;\lambda\pm\ii\epsilon)$ is analytic in $\mathbb{C}_+$ if and only if the elements of $\My=(y_1(\lambda\pm\ii\epsilon), \dots, y_N(\lambda\pm\ii\epsilon))^\top$ satisfy the system of equations
\begin{equation}\label{eq:W-analytic}
\int_{U_j^\pm}\ee^{-\ii {h}^\infty(z;-2(\lambda\pm\ii\epsilon))}
	\left(-f(z)+\sum_{k=1}^N\frac{y_n(\lambda\pm\ii\epsilon)}{z-p_n}\right) \dd z
	=0,
\quad j=1,\dots, N.
\end{equation}
Using the definitions \eqref{eq:D-negative} and \eqref{eq:bpm-elements} extended from $\lambda>0$ slightly into the upper/lower half-planes, 
the system of equations~\eqref{eq:W-analytic} becomes as in \cite[equation (22)]{MillerW16a}
\begin{equation}
\ME^{\pm}(\lambda\pm\ii\epsilon) \My(\lambda\pm\ii\epsilon)=\Mb^{\pm}(\lambda\pm\ii\epsilon).
\end{equation}
Cramer's rule implies that for every $n=1,\dots,N$, the coefficient $y_n(\lambda\pm\ii\epsilon)$ has the expression
\begin{equation}
y_n(\lambda\pm\ii\epsilon )
	=\frac{\det(  \ME^\pm(\lambda\pm\ii\epsilon )\longleftarroww{n} \Mb^\pm(\lambda\pm\ii\epsilon ))}{\det (\ME\pm(\lambda\pm\ii\epsilon ))}.
\end{equation}
Letting $\epsilon\to 0^+$, the function $\varphi(\diamond;\lambda\pm\ii\epsilon)$ tends to the pointwise limit given by~\eqref{eq:W+}--\eqref{eq:bpm-elements}.
\end{proof}

\subsection{Distorted Fourier transform}\label{ss:distorted}
In this subsection, we show the following identity between $\varphi^\pm(\diamond,\lambda)$ and $m_-(\diamond,\lambda)$.

\begin{proposition}\label{prop:phi-m} For every $\lambda>0$ and for every $x\in\mathbb{R}$,
\begin{equation}\label{eq:phi-m-}
\varphi^+(x;\lambda)-\varphi^-(x;\lambda)
	=-\ii\frac{\det(\MA^\infty_f(\lambda))}{\det({\ME^-}(\lambda))} {m_-}(x;\lambda).
\end{equation}
\end{proposition}

Let us first show that this Proposition implies Proposition~\ref{prop:distordu-frac}.

\begin{proof}[Proof of Proposition~\ref{prop:distordu-frac}]
Combining~\eqref{eq:dmuf} and~\eqref{eq:dmuf2} with \eqref{eq:phi-m-} yields
\begin{equation}\begin{aligned}
|\widetilde{f}^-(\lambda)|^2
	&=-\ii \lim_{\epsilon\to 0^+}\langle \varphi(\diamond;\lambda+\ii\epsilon)-\varphi(\diamond;\lambda-\ii\epsilon), f \rangle\\
	&=-
    \frac{\det (\MA^\infty_f(\lambda))}{\det (\ME^-(\lambda))}\langle m_-(\diamond,\lambda),f\rangle\\
	&=-
    \frac{\det (\MA^\infty_f(\lambda))}{\det (\ME^-(\lambda))}{\widetilde{f}^-(\lambda)}^*
\end{aligned}
\end{equation}
so that
\begin{equation}\label{eq:distordu-frac2}
\widetilde{f}^-(\lambda)
	=-
    \frac{\det (\MA^\infty_f(\lambda))}{\det( \ME^-(\lambda))}.
\end{equation}
\end{proof}

We now turn to the proof of Proposition~\ref{prop:phi-m}. From Lemma~\ref{lem:phi} we have
\begin{multline}\label{eq:phi+phi-}
\varphi^+(x;\lambda)-\varphi^-(x;\lambda)
	=
	-\ii \ee^{\ii{h}^\infty(x;-2\lambda)}
		\int_{{\ell_0^+(x)}}\ee^{-\ii {h}^\infty(z;-2\lambda)}
	\sum_{n=1}^N\frac{y_n^+(\lambda)-y_n^-(\lambda)}{z-p_n}\, \dd z	\\
		-\ii\ee^{\ii h^\infty(x;-2\lambda)}\int_{\Gamma_N}\ee^{-\ii {h}^\infty(z;-2\lambda)}
	\left(f(z)-\sum_{n=1}^N\frac{{y_n^-}(\lambda)}{z-p_n}\right)\, \dd z.
\end{multline}
It is enough to show the following lemma.

\begin{lemma}[Identities between $y_n^\pm$ and $w_n^-$]\label{lem:m-to-phi}
Let $\lambda>0$. Then
\begin{equation}\label{eq:vk-constant}
\int_{\Gamma_N}\ee^{-\ii  {h}^\infty(z;-2\lambda)}
	\left(f(z)-\sum_{n=1}^N\frac{{y_n^-}(\lambda)}{z-p_n}\right)\, \dd z
		=\frac{\det (\MA^\infty_f(\lambda))}{\det ({\ME^-}(\lambda))},
\end{equation}
\begin{equation}\label{eq:vk-diff}
y_n^+(\lambda)-y_n^-(\lambda)
	=\frac{\det (\MA^\infty_f(\lambda))}{\det ({\ME^-}(\lambda))}w_n^-(\lambda) ,
	\quad n=1,\dots,N.
\end{equation}
\end{lemma}

\begin{proof}[Proof of~\eqref{eq:vk-constant}]
Let $L$ denote the left-hand side of \eqref{eq:vk-constant}.  Using \eqref{eq:ypm-n} and the definition \eqref{eq:qk} of $q_n(\lambda)$, we then have
\begin{equation}
    L=q_0(\lambda)-\mathbf{q}(\lambda)^\top\mathbf{y}^-(\lambda) = q_0(\lambda)-\mathbf{q}(\lambda)^\top \mathbf{E}^-(\lambda)^{-1}\mathbf{b}^-(\lambda),
\end{equation}
where in the second equality we used \eqref{eq:ypm-n}.  By Schur's formula and Lemma~\ref{lem:detD-nonzero},
\begin{equation}
    L=\frac{1}{\det(\mathbf{E}^-(\lambda))}\begin{vmatrix}q_0(\lambda) & \mathbf{q}(\lambda)^\top\\\mathbf{b}^-(\lambda) & \mathbf{E}^-(\lambda)\end{vmatrix}=\frac{\det(\mathbf{A}_f^\infty(\lambda))}{\det(\mathbf{E}^-(\lambda))},
\end{equation}
where we used the definition \eqref{eq:A-f} in the second equality.
\end{proof}

\begin{proof}[Proof of~\eqref{eq:vk-diff}]
Taking into account that $\mathbf{E}^-(\lambda)=\mathbf{E}^+(\lambda)+\mathbf{dq}(\lambda)^\top$ which follows from \eqref{eq:Uj+-} in Remark~\ref{rem:integral-identities}, showing \eqref{eq:vk-diff} is equivalent to proving that $S_2=-S_1$, where $S_1,S_2$ are defined as follows.
Using the expression~\eqref{eq:ypm-n} for $y_n^+(\lambda)$ and \eqref{eq:vk+-} for $y_n^-(\lambda)$, and using~\eqref{eq:wk} for $w_n^-(\lambda)$,
\begin{equation}S_1:=\det(\ME^+\longleftarroww{n}\Mb^+)\det(\ME^++\Md\Mq^\top)
	-\det(\ME^++\Md\Mq^\top \longleftarroww{n} \Mb^+ +q_0 \Md)\det(\ME^+)
\end{equation} and 
\begin{equation}S_2:=\det(\MA^\infty_f)\det(\ME^+\longleftarroww{n} \Md),
\end{equation}
where we suppress the dependence on $\lambda$ for simplicity of notation.

Further simplifying the notation by writing $\mathbf{E}$ and $\mathbf{b}$ for $\mathbf{E}^+$ and $\mathbf{b}^+$ respectively, we have from Cramer's rule
\begin{equation}
\det(\ME \longleftarroww{n} \Mb)
	=\det(\ME)(\ME^{-1}\Mb)_n,
\end{equation}
where $(\ME^{-1}\Mb)_n$ denotes element $n$ of the vector $\ME^{-1}\Mb$. 
Moreover, applying the matrix determinant lemma:
\begin{equation}\label{eq:detD}
\det(\ME+\Md\Mq^T)
	=\det(\ME)(1+\Mq^T \ME^{-1}\Md),
\end{equation}
the first term in the definition of $S_1$ therefore can be written as
\begin{equation}
\det(\ME \longleftarroww{n} \Mb)\det(\ME+\Md\Mq^T)
	=\det(\ME)^2 (\ME^{-1}\Mb)_n (1+\Mq^T \ME^{-1}\Md).
\end{equation}

Now looking at the second term in $S_1$, another application of Cramer's rule gives
\begin{equation}
\det(\ME+\Md \Mq^\top \longleftarroww{n} \Mb+q_0 \Md)
	=\det(\ME+\Md\Mq^\top)((\ME+\Md\Mq^\top)^{-1}(\Mb+q_0 \Md))_n.
\end{equation}
Moreover, by the Sherman–Morrison formula, 
\begin{equation}
(\ME+\Md\Mq^\top)^{-1}
	=\ME^{-1}-\frac{\ME^{-1}\Md\Mq^\top\ME^{-1}}{1+\Mq^\top\ME^{-1}\Md}.
\end{equation}
Also using again~\eqref{eq:detD}, we find
\begin{equation}
\begin{aligned}
\det(\ME+\Md \Mq^\top \longleftarroww{n} \Mb+q_0 \Md)
	&=\det(\ME)((1+\Mq^\top\ME^{-1}\Md)\ME^{-1}(\Mb+q_0\Md)-\ME^{-1}\Md\Mq^\top\ME^{-1}(\Mb+q_0\Md))_n\\
	&=\det(\ME)\left(
	(1+\Mq^\top \ME^{-1}\Md)(\ME^{-1}\Mb)_n+(q_0
	- \Mq^\top \ME^{-1}\Mb)(\ME^{-1} \Md)_n
	\right).
\end{aligned}
\end{equation}
Finally, we obtain a simplied form of $S_1$:
\begin{equation}
S_1
	=-\det(\ME)^2
	(q_0-\Mq^\top \ME^{-1}\Mb) (\ME^{-1}\Md)_n.
\end{equation}

Now we turn to $S_2$.
Applying Cramer's rule to the second factor gives
\begin{equation}
\det(\ME \longleftarroww{n} \Md )=\det(\ME)(\ME^{-1}\Md)_n.
\end{equation}
Then we expand the determinant in the first factor of $S_2$ along the first row and apply Cramer's rule to get
\begin{equation}
\begin{aligned}
\det(\MA^\infty_f)
	&=q_0\det(\ME)-\sum_{n=1}^Nq_n\det(\ME \longleftarroww{n}\Mb)\\
	&=\det(\ME)\left(q_0-\sum_{n=1}^N q_n(\ME^{-1}\Mb)_n\right)\\
	&=\det(\ME)(q_0-\Mq^\top \ME^{-1}\Mb).
\end{aligned}
\end{equation}
We conclude that 
\begin{equation}
S_2
	=\det(\ME)^2(\ME^{-1} \Md)_n(q_0-\Mq^\top \ME^{-1} \Mb)
	=-S_1.
\end{equation}
\end{proof}

Returning now to \eqref{eq:phi+phi-}, we use Lemma~\ref{lem:m-to-phi} to obtain
\begin{equation}
    \varphi^+(x;\lambda)-\varphi^-(x;\lambda)=-\ii\frac{\det(\mathbf{A}_f^\infty(\lambda))}{\det(\mathbf{E}^-(\lambda))}\ee^{\ii h^\infty(x;-2\lambda)}\left(1+\int_{\ell_0^+(x)}\ee^{-\ii h^\infty(z;-2\lambda)}\sum_{n=1}^N\frac{w_n^-(\lambda)}{z-p_n}\,\dd z\right).
\end{equation}
Then using Lemma~\ref{lem:m} completes the proof of Proposition~\ref{prop:distordu-frac}.

\begin{remark}
In \cite[Appendix A]{GassotGM26b}, the reflection coefficient $\beta(\lambda)$ was calculated in a similar fashion for the same initial data with the result written in the present notation as
\begin{equation}
    \beta(\lambda)=\ii\ee^{-2\pi\ii\mathrm{Im}(c)}\frac{\det(\mathbf{A}^\infty_{\Pi u_0}(\lambda))}{\det(\mathbf{E}^+(\lambda))},\quad\lambda>0.
\end{equation}
Therefore, as mentioned in Remark~\ref{rem:beta-of-lambda}, the other distorted Fourier transform of $\Pi u_0$ is
\begin{equation}
    \widetilde{\Pi u_0}^+(\lambda)=-\ii\beta(\lambda)=\ee^{-2\pi\ii\mathrm{Im}(c)}\frac{\det(\mathbf{A}^\infty_{\Pi u_0}(\lambda))}{\det(\mathbf{E}^+(\lambda))},\quad\lambda>0.
\end{equation}
Recalling now \eqref{eq:Gamma-of-lambda}, the unit-modulus function $\Gamma(\lambda)$ relating $m_\pm(\diamond;\lambda)$ and hence also the distorted Fourier transforms $\widetilde{\Pi u_0}^\pm(\lambda)$ is given explicitly by
\begin{equation}
    \Gamma(\lambda)=\frac{\widetilde{\Pi u_0}^-(\lambda)}{\widetilde{\Pi u_0}^+(\lambda)}=-\ee^{2\pi\ii\mathrm{Im}(c)}\frac{\det(\mathbf{E}^+(\lambda))}{\det(\mathbf{E}^-(\lambda))},\quad\lambda>0,
\end{equation}
where we recall that $c:=c_1+\dots+c_N$.

\begin{remark}[Reflectionless potentials]
The $N$-soliton initial data $u_0$ considered in Section~\ref{sec:examples} and Appendix~\ref{app:Nsoliton} has $c_1=c_2=\cdots = c_N=-\ii$.  In this case, $\ee^{-\ii h^\infty(z;-2\lambda)}$ has simple zeros at $p_1^*,\dots,p_N^*$, so the integrand in each of the first row entries of $\mathbf{A}^\infty_{\Pi u_0}(\lambda)$ is analytic in $\C_-$ and exponentially decaying as $z\to\infty\ee^{-\ii\pi/4}$ for $\lambda>0$.  It follows that the first row entries of $\mathbf{A}^\infty_{\Pi u_0}(\lambda)$ all vanish identically for $\lambda>0$, so $\widetilde{\Pi u_0}^\pm(\lambda)=0$ and $\beta(\lambda)=0$, consistent with the terminology of $u_0$ being a reflectionless potential.
\end{remark}
\end{remark}

\bibliographystyle{siamplain}
\bibliography{references}

@article{Benjamin67,
    title = {Internal waves of permanent form in fluids of great depth},
    author = {T. Brooke Benjamin},
    fjournal = {Journal of Fluid Mechanics},
    journal = {J. Fluid Mech.},
    volume = {29},
    number = {3},
    pages = {559--592},
    year = {1967}
}

@article{Ono75,
    title = {Algebraic solitary waves in stratified fluids},
    author = {H. Ono},
    fjournal = {Journal of the Physical Society of Japan},
    journal = {J. Phys. Soc. Japan},
    volume = {39},
    number = {4},
    pages = {1082--1091},
    doi = {10.1143/JPSJ.39.1082},
    mrnumber = {398275},
    year={1975}
}

@article{GassotG26,
    title = {Infinite-order multisoliton solutions to the {B}enjamin-{O}no equation and soliton resolution},
    author={Louise Gassot and Patrick G\'erard},
    note={preprint, arXiv:2603.15419},
    year={2026}
}

@article{GassotGM26,
    title = {A proof of the soliton resolution conjecture for the {B}enjamin-{O}no equation},
    author={Louise Gassot and Patrick G\'erard and Peter D. Miller},
    note={preprint, arXiv:2601.10488},
    year={2026}
}

@article{GassotGM26b,
    title={The {B}enjamin-{O}no equation in the long-time limit: linearized self-similar universality},
    author={Louise Gassot and Patrick G\'erard and Peter D. Miller},
    note = {preprint, arXiv:2606.02989},
    year = {2026}
}

@article {GerardG15,
    AUTHOR = {G\'erard, Patrick and Grellier, Sandrine},
     TITLE = {An explicit formula for the cubic {S}zeg{\H{o}} equation},
   JOURNAL = {Trans. Amer. Math. Soc.},
  FJOURNAL = {Transactions of the American Mathematical Society},
    VOLUME = {367},
      YEAR = {2015},
    NUMBER = {4},
     PAGES = {2979--2995},
      ISSN = {0002-9947,1088-6850},
   MRCLASS = {37K15 (47B35)},
  MRNUMBER = {3301889},
MRREVIEWER = {Julia\ B.\ Prada},
       DOI = {10.1090/S0002-9947-2014-06310-1},
       URL = {https://doi-org.proxy.lib.umich.edu/10.1090/S0002-9947-2014-06310-1},
}

@article {GerardP24,
    AUTHOR = {G\'erard, Patrick and Pushnitski, Alexander},
     TITLE = {The cubic {S}zeg{\H{o}} equation on the real line: explicit
              formula and well-posedness on the {H}ardy class},
   JOURNAL = {Comm. Math. Phys.},
  FJOURNAL = {Communications in Mathematical Physics},
    VOLUME = {405},
      YEAR = {2024},
    NUMBER = {7},
     PAGES = {Paper No. 167, 31},
      ISSN = {0010-3616,1432-0916},
   MRCLASS = {47B35 (30E25 35Q55)},
  MRNUMBER = {4768537},
       DOI = {10.1007/s00220-024-05040-4},
       URL = {https://doi-org.proxy.lib.umich.edu/10.1007/s00220-024-05040-4},
}

@article {GerardL25,
    AUTHOR = {G\'erard, Patrick and Lenzmann, Enno},
     TITLE = {Global well-posedness and soliton resolution for the half-wave
              maps equation with rational data},
   JOURNAL = {Forum Math. Sigma},
  FJOURNAL = {Forum of Mathematics. Sigma},
    VOLUME = {13},
      YEAR = {2025},
     PAGES = {Paper No. e190, 71},
      ISSN = {2050-5094},
   MRCLASS = {35R11 (37K10 47B53)},
  MRNUMBER = {4990790},
MRREVIEWER = {Junyi\ Zhu},
       DOI = {10.1017/fms.2025.10136},
       URL = {https://doi-org.proxy.lib.umich.edu/10.1017/fms.2025.10136},
}

@misc{GerardL25HWMtorus,
      title={The half-wave maps equation on $\mathbb{T}$: Global well-posedness in {$H^{1/2}$} and almost periodicity}, 
      author={P. Gérard and E. Lenzmann},
      year={2026},
      eprint={2602.20895},
      archivePrefix={arXiv},
      primaryClass={math.AP},
      url={https://arxiv.org/abs/2602.20895}, 
}

@article {AbdelouhabBFS89,
    AUTHOR = {Abdelouhab, L. and Bona, J. L. and Felland, M. and Saut,
              J.-C.},
     TITLE = {Nonlocal models for nonlinear, dispersive waves},
   JOURNAL = {Phys. D},
  FJOURNAL = {Physica D. Nonlinear Phenomena},
    VOLUME = {40},
      YEAR = {1989},
    NUMBER = {3},
     PAGES = {360--392},
      ISSN = {0167-2789,1872-8022},
   MRCLASS = {58D25 (35Q53)},
  MRNUMBER = {1044731},
MRREVIEWER = {Peter\ L.\ Christiansen},
       DOI = {10.1016/0167-2789(89)90050-X},
       URL = {https://doi-org.proxy.lib.umich.edu/10.1016/0167-2789(89)90050-X},
}

@article{BronsardCD26,
    author = {Alama Bronsard, Yvonne and Chen, Xi and Dobeault, Matthieu},
    journal = {Found.\@ Comp.\@ Math.\@},
    fjournal = {Foundations of Computational Mathematics},
    title = {Spectrally accurate fully discrete schemes for some nonlocal and nonlinear integrable {PDE}s via explicit formulas},
    doi = {10.1007/s10208-026-09740-2},
    year = {2026},
}

@article {AmickT91,
    AUTHOR = {Amick, C. J. and Toland, J. F.},
     TITLE = {Uniqueness and related analytic properties for the
              {B}enjamin-{O}no equation---a nonlinear {N}eumann problem in
              the plane},
   JOURNAL = {Acta Math.},
  FJOURNAL = {Acta Mathematica},
    VOLUME = {167},
      YEAR = {1991},
    NUMBER = {1-2},
     PAGES = {107--126},
      ISSN = {0001-5962,1871-2509},
   MRCLASS = {35Q35 (35C05 76C10)},
  MRNUMBER = {1111746},
MRREVIEWER = {Paul\ G.\ Schmidt},
       DOI = {10.1007/BF02392447},
       URL = {https://doi-org.proxy.lib.umich.edu/10.1007/BF02392447},
}

@article {AndersonT85,
    AUTHOR = {Anderson, R. L. and Taflin, E.},
     TITLE = {The {B}enjamin-{O}no equation---recursivity of linearization
              maps---{L}ax pairs},
   JOURNAL = {Lett. Math. Phys.},
  FJOURNAL = {Letters in Mathematical Physics. A Journal for the Rapid
              Dissemination of Short Contributions in the Field of
              Mathematical Physics},
    VOLUME = {9},
      YEAR = {1985},
    NUMBER = {4},
     PAGES = {299--311},
      ISSN = {0377-9017},
   MRCLASS = {35Q20 (45K05 76V05)},
  MRNUMBER = {796631},
MRREVIEWER = {F.\ Pempinelli},
       DOI = {10.1007/BF00397756},
       URL = {https://doi-org.proxy.lib.umich.edu/10.1007/BF00397756},
}

@Article{BlackstoneGGM24a,
    title= {The {B}enjamin-{O}no Initial-Value Problem for Rational Data with Application to Long-Time Asymptotics and Scattering},
    author={Elliot Blackstone and Louise Gassot and Patrick G\'erard and Peter D. Miller},
    journal={Ann. Inst. H. Poincar\'e C Anal. Non Lin\'eare},
    year=2025,
    DOI = {10.4171/AIHPC/169},
    note={Published online first.  arXiv:2410.14870}
}

@article{BlackstoneGGM26,
    title={The {B}enjamin-{O}no equation in the zero-dispersion limit for rational initial data: generation of dispersive shock waves},
    author={Elliot Blackstone and Louise Gassot and Patrick G\'erard and Peter D. Miller},
    journal={Comm. Pure Appl. Math.},
    year=2026,
    DOI = {10.1002/cpa.70044},
    note = {Published online first. arXiv:2410.17405}
}

@incollection {BlackstoneGM25,
    AUTHOR = {Blackstone, Elliot and Gassot, Louise and Miller, Peter D.},
     TITLE = {On strong zero-dispersion asymptotics for {B}enjamin-{O}no
              soliton ensembles},
 BOOKTITLE = {The versatility of integrability---in memory of {I}gor
              {K}richever},
    SERIES = {Contemp. Math.},
    VOLUME = {823},
     PAGES = {213--270},
 PUBLISHER = {Amer. Math. Soc., [Providence], RI},
      YEAR = {[2025] \copyright 2025},
      ISBN = {978-1-4704-7526-0; [9781470480363]},
   MRCLASS = {35R09 (35C07 35C08 35C20 35Q53)},
  MRNUMBER = {4945851},
MRREVIEWER = {Theodoros\ P.\ Horikis},
       DOI = {10.1090/conm/823/16498},
       URL = {https://doi-org.proxy.lib.umich.edu/10.1090/conm/823/16498},
}

@article {BlackstoneMM26,
    AUTHOR = {Blackstone, Elliot and Miller, Peter D. and Mitchell, Matthew
              D.},
     TITLE = {Universality in the small-dispersion limit of the
              {B}enjamin-{O}no equation},
   JOURNAL = {Comm.\@ Math.\@ Phys.\@},
  FJOURNAL = {Communications in Mathematical Physics},
    VOLUME = {407},
      YEAR = {2026},
    NUMBER = {3},
    PAGES = {1--69},
    Note = {Paper No. 60},
      ISSN = {0010-3616,1432-0916},
   MRCLASS = {76B55 (35Q53)},
  MRNUMBER = {5035577},
       DOI = {10.1007/s00220-025-05506-z},
    
}

@article {BockK79,
    AUTHOR = {Bock, T. L. and Kruskal, M. D.},
     TITLE = {A two-parameter {M}iura transformation of the {B}enjamin-{O}no
              equation},
   JOURNAL = {Phys. Lett. A},
  FJOURNAL = {Physics Letters. A},
    VOLUME = {74},
      YEAR = {1979},
    NUMBER = {3-4},
     PAGES = {173--176},
      ISSN = {0375-9601,1873-2429},
   MRCLASS = {35Q20 (58F07 76B15)},
  MRNUMBER = {591320},
       DOI = {10.1016/0375-9601(79)90762-X},
       URL = {https://doi-org.proxy.lib.umich.edu/10.1016/0375-9601(79)90762-X},
}

@article {Case79,
    AUTHOR = {Case, K. M.},
     TITLE = {Meromorphic solutions of the {B}enjamin-{O}no equation},
   JOURNAL = {Phys. A},
  FJOURNAL = {Physica A},
    VOLUME = {96},
      YEAR = {1979},
    NUMBER = {1-2},
     PAGES = {173--182},
      ISSN = {0378-4371},
   MRCLASS = {35C05 (35Q20)},
  MRNUMBER = {534587},
MRREVIEWER = {Manfred\ W.\ Kracht},
       DOI = {10.1016/0378-4371(79)90204-8},
       URL = {https://doi-org.proxy.lib.umich.edu/10.1016/0378-4371(79)90204-8},
}

@article {Cheng26+,
author = {Cheng, Zijian},
title = {Degenerations in {B}enjamin-{O}no theory},
note = {In preparation},
year = {2026},
}

@article {Cheng26+b,
author = {Cheng, Zijian},
title = {On the solution of the {B}enjamin-{O}no equation in the small-dispersion limit near the gradient catastrophe point},
note = {In preparation},
year = {2026},
}

@article {ClaeysG09,
    AUTHOR = {Claeys, T. and Grava, T.},
     TITLE = {Universality of the break-up profile for the {K}d{V} equation
              in the small dispersion limit using the {R}iemann-{H}ilbert
              approach},
   JOURNAL = {Comm. Math. Phys.},
  FJOURNAL = {Communications in Mathematical Physics},
    VOLUME = {286},
      YEAR = {2009},
    NUMBER = {3},
     PAGES = {979--1009},
      ISSN = {0010-3616,1432-0916},
   MRCLASS = {35Q53 (35C20 35P25 35Q15 35R30 37K15 37K20)},
  MRNUMBER = {2472023},
MRREVIEWER = {Woodford\ W.\ Zachary},
       DOI = {10.1007/s00220-008-0680-5},
       URL = {https://doi-org.proxy.lib.umich.edu/10.1007/s00220-008-0680-5},
}

@article {ClaeysG10a,
    AUTHOR = {Claeys, Tom and Grava, Tamara},
     TITLE = {Painlev\'e{} {II} asymptotics near the leading edge of the
              oscillatory zone for the {K}orteweg-de {V}ries equation in the
              small-dispersion limit},
   JOURNAL = {Comm. Pure Appl. Math.},
  FJOURNAL = {Communications on Pure and Applied Mathematics},
    VOLUME = {63},
      YEAR = {2010},
    NUMBER = {2},
     PAGES = {203--232},
      ISSN = {0010-3640,1097-0312},
   MRCLASS = {35Q53 (35C20 35Q15 37K10 37K15)},
  MRNUMBER = {2588460},
       DOI = {10.1002/cpa.20277},
       URL = {https://doi-org.proxy.lib.umich.edu/10.1002/cpa.20277},
}

@article {ClaeysG10b,
    AUTHOR = {Claeys, Tom and Grava, Tamara},
     TITLE = {Solitonic asymptotics for the {K}orteweg-de {V}ries equation
              in the small dispersion limit},
   JOURNAL = {SIAM J. Math. Anal.},
  FJOURNAL = {SIAM Journal on Mathematical Analysis},
    VOLUME = {42},
      YEAR = {2010},
    NUMBER = {5},
     PAGES = {2132--2154},
      ISSN = {0036-1410,1095-7154},
   MRCLASS = {35Q53 (35B36 35C08 35P25 35Q15 37K15)},
  MRNUMBER = {2729434},
MRREVIEWER = {Sarah\ G.\ Raynor},
       DOI = {10.1137/090779103},
       URL = {https://doi-org.proxy.lib.umich.edu/10.1137/090779103},
}

@article {CoifmanW90,
    AUTHOR = {Coifman, Ronald R. and Wickerhauser, Mladen Victor},
     TITLE = {The scattering transform for the {B}enjamin-{O}no equation},
   JOURNAL = {Inverse Problems},
  FJOURNAL = {Inverse Problems. An International Journal on the Theory and
              Practice of Inverse Problems, Inverse Methods and Computerized
              Inversion of Data},
    VOLUME = {6},
      YEAR = {1990},
    NUMBER = {5},
     PAGES = {825--861},
      ISSN = {0266-5611,1361-6420},
   MRCLASS = {34L25 (34A55 35Q53)},
  MRNUMBER = {1073870},
MRREVIEWER = {C.\ Tomei},
       DOI = {10.1088/0266-5611/6/5/011},
       URL = {https://doi-org.proxy.lib.umich.edu/10.1088/0266-5611/6/5/011},
}

@article {DeiftVZ97,
    AUTHOR = {Deift, P. and Venakides, S. and Zhou, X.},
     TITLE = {New results in small dispersion {K}d{V} by an extension of the
              steepest descent method for {R}iemann-{H}ilbert problems},
   JOURNAL = {Internat. Math. Res. Notices},
  FJOURNAL = {International Mathematics Research Notices},
      YEAR = {1997},
    NUMBER = {6},
     PAGES = {286--299},
      ISSN = {1073-7928,1687-0247},
   MRCLASS = {35Q53},
  MRNUMBER = {1440305},
MRREVIEWER = {Linghai\ Zhang},
       DOI = {10.1155/S1073792897000214},
       URL = {https://doi-org.proxy.lib.umich.edu/10.1155/S1073792897000214},
}

@article {DeiftVZ98,
    AUTHOR = {Deift, P. and Venakides, S. and Zhou, X.},
     TITLE = {An extension of the steepest descent method for
              {R}iemann-{H}ilbert problems: the small dispersion limit of
              the {K}orteweg-de {V}ries ({K}d{V}) equation},
   JOURNAL = {Proc. Natl. Acad. Sci. USA},
  FJOURNAL = {Proceedings of the National Academy of Sciences of the United
              States of America},
    VOLUME = {95},
      YEAR = {1998},
    NUMBER = {2},
     PAGES = {450--454},
      ISSN = {0027-8424,1091-6490},
   MRCLASS = {65N15},
  MRNUMBER = {1613151},
       DOI = {10.1073/pnas.95.2.450},
       URL = {https://doi-org.proxy.lib.umich.edu/10.1073/pnas.95.2.450},
}

@article {DobrokhotovK91,
    AUTHOR = {Dobrokhotov, S. Yu. and Krichever, Igor Moiseevich},
     TITLE = {Multiphase solutions of the {B}enjamin-{O}no equation and
              their averaging},
   JOURNAL = {Mat. Zametki},
  FJOURNAL = {Akademiya Nauk SSSR. Matematicheskie Zametki},
    VOLUME = {49},
      YEAR = {1991},
    NUMBER = {6},
     PAGES = {42--58, 158},
      ISSN = {0025-567X},
   MRCLASS = {35Q35 (45K05)},
  MRNUMBER = {1135514},
       DOI = {10.1007/BF01156581},
       URL = {https://doi-org.proxy.lib.umich.edu/10.1007/BF01156581},
}

@incollection {Dubrovin08,
    AUTHOR = {Dubrovin, Boris},
     TITLE = {On universality of critical behaviour in {H}amiltonian {PDE}s},
 BOOKTITLE = {Geometry, topology, and mathematical physics},
    SERIES = {Amer. Math. Soc. Transl. Ser. 2},
    VOLUME = {224},
     PAGES = {59--109},
 PUBLISHER = {Amer. Math. Soc., Providence, RI},
      YEAR = {2008},
      ISBN = {978-0-8218-4674-2},
   MRCLASS = {37K55 (34M55 35F20 37K05)},
  MRNUMBER = {2462355},
MRREVIEWER = {Ian\ A. B. Strachan},
       DOI = {10.1090/trans2/224/03},
       URL = {https://doi-org.proxy.lib.umich.edu/10.1090/trans2/224/03},
}

@article {FokasA83,
    AUTHOR = {Fokas, A. S. and Ablowitz, M. J.},
     TITLE = {The inverse scattering transform for the {B}enjamin-{O}no
              equation---a pivot to multidimensional problems},
   JOURNAL = {Stud. Appl. Math.},
  FJOURNAL = {Studies in Applied Mathematics},
    VOLUME = {68},
      YEAR = {1983},
    NUMBER = {1},
     PAGES = {1--10},
      ISSN = {0022-2526},
   MRCLASS = {35Q20 (44A15 76V05)},
  MRNUMBER = {686248},
MRREVIEWER = {Amy\ Cohen},
       DOI = {10.1002/sapm19836811},
       URL = {https://doi-org.proxy.lib.umich.edu/10.1002/sapm19836811},
}

@article {FlaschkaFM80,
    AUTHOR = {Flaschka, H. and Forest, M. G. and McLaughlin, D. W.},
     TITLE = {Multiphase averaging and the inverse spectral solution of the
              {K}orteweg-de {V}ries equation},
   JOURNAL = {Comm. Pure Appl. Math.},
  FJOURNAL = {Communications on Pure and Applied Mathematics},
    VOLUME = {33},
      YEAR = {1980},
    NUMBER = {6},
     PAGES = {739--784},
      ISSN = {0010-3640,1097-0312},
   MRCLASS = {35Q20 (58F07 76B15)},
  MRNUMBER = {596433},
MRREVIEWER = {H.\ P.\ McKean, Jr.},
       DOI = {10.1002/cpa.3160330605},
       URL = {https://doi-org.proxy.lib.umich.edu/10.1002/cpa.3160330605},
}

@article{Gerard22,
	author = {Gérard, P.},
	doi = {10.2140/tunis.2023.5.593},
	fjournal = {Tunisian Journal of Mathematics},
	issn = {2576-7658,2576-7666},
	journal = {Tunis. J. Math.},
	mrclass = {35Q53 (35C05 37K15 47B35)},
	mrnumber = {4662323},
	number = {3},
	pages = {593--603},
	title = {An explicit formula for the {B}enjamin-{O}no equation},
	volume = {5},
	year = {2023}
}

@article {Gerard23,
    AUTHOR = {G\'erard, Patrick},
     TITLE = {The zero dispersion limit for the {B}enjamin-{O}no equation on
              the line},
   JOURNAL = {C. R. Math. Acad. Sci. Paris},
  FJOURNAL = {Comptes Rendus Math\'ematique. Acad\'emie des Sciences. Paris},
    VOLUME = {362},
      YEAR = {2024},
     PAGES = {619--634},
      ISSN = {1631-073X,1778-3569},
   MRCLASS = {35Q53 (35B30 37K15 47B35)},
  MRNUMBER = {4773745},
       DOI = {10.5802/crmath.575},
}

@article{Gassot21zero,
  title={{Zero-dispersion limit for the Benjamin-Ono equation on the torus with single well initial data}},
  author={Gassot, L.},
 fjournal = {Communications in Mathematical Physics},
 journal = {Comm. Math. Phys.},
 issn = {0010-3616},
 volume = {401},
 number = {3},
 pages = {2793--2843},
 year = {2023},
 language = {English},
 doi = {10.1007/s00220-023-04701-0},
 zbMATH = {7719634},
 Zbl = {1529.35381}
}

@article{Gassot23zero,
 author = {Gassot, L.},
 title = {Lax eigenvalues in the zero-dispersion limit for the {Benjamin}-{Ono} equation on the torus},
 fjournal = {SIAM Journal on Mathematical Analysis},
 journal = {SIAM J. Math. Anal.},
 issn = {0036-1410},
 volume = {55},
 number = {5},
 pages = {5782--5822},
 year = {2023},
 language = {English},
 doi = {10.1137/23M154635X},
 zbMATH = {7757958},
 Zbl = {1527.35352}
}

@article{Maehlen26zero,
	author = {M{\ae}hlen, O.},
	journal = {Comm. Math. Phys.},
	number = {6},
	pages = {120},
	title = {The Zero-Dispersion Limit for the {B}enjamin-{O}no Equation on the Circle},
	volume = {407},
	year = {2026},
	doi = {10.1007/s00220-026-05628-y}}

@article {MasoeroRA15,
    AUTHOR = {Masoero, Davide and Raimondo, Andrea and Antunes, Pedro R. S.},
     TITLE = {Critical behavior for scalar nonlinear waves},
   JOURNAL = {Phys. D},
  FJOURNAL = {Physica D. Nonlinear Phenomena},
    VOLUME = {292/293},
      YEAR = {2015},
     PAGES = {1--7},
      ISSN = {0167-2789,1872-8022},
   MRCLASS = {35Q53 (35L67)},
  MRNUMBER = {3286607},
       DOI = {10.1016/j.physd.2014.09.007},
       URL = {https://doi-org.proxy.lib.umich.edu/10.1016/j.physd.2014.09.007},
}

@article{Badreddine24zero,
author = {Badreddine, Rana},
title = {Zero-Dispersion Limit of the {C}alogero–{M}oser Derivative {NLS} Equation},
journal = {SIAM Journal on Mathematical Analysis},
volume = {56},
number = {6},
pages = {7228-7249},
year = {2024},
doi = {10.1137/24M1646935},
URL = { https://doi.org/10.1137/24M1646935
},
eprint = { https://doi.org/10.1137/24M1646935}
}

@article{GurevichP74,
    author = {Gurevich, A. V. and Pitaevski\v{i}, L. P.},
    title = {Nonstationary structure of a collisionless shock wave},
    journal = {Zh. Eksp. Teor. Fiz.},
    volume = {65}, 
    pages = {590--604}, 
    year = {1973}, 
    language = {Russian},
    note = {English trans. \textit{Sov. Phys. JETP} \textbf{38}, 291--297, 1974},
}

@article{Chen25explicitBO,
 author = {Chen, Xi},
 title = {Explicit formula for the {B}enjamin-{O}no equation with square integrable and real valued initial data and applications to the zero dispersion limit},
 fjournal = {Pure and Applied Analysis},
 journal = {Pure Appl. Anal.},
 issn = {2578-5893},
 volume = {7},
 number = {1},
 pages = {101--126},
 year = {2025},
 language = {English},
 doi = {10.2140/paa.2025.7.101},
 zbMATH = {8014452},
 Zbl = {1560.35058}
}

@article{Chen2026explicitCCM,
  title={The defocusing {C}alogero-{M}oser derivative nonlinear {S}chr{\"o}dinger equation with a nonvanishing condition at infinity},
  author={Chen, X.},
  journal={SIAM J. Math. Anal.},
  volume={58},
  number={1},
  pages={182--205},
  year={2026},
  publisher={SIAM}
}

@misc{GerardH26,
      title={An Explicit Formula for the {B}enjamin-{O}no Hierarchy with Applications to Traveling Waves and Zero-Dispersion Limits}, 
      author={Patrick G\'erard and Jiao He},
      year={2026},
      eprint={2604.20464},
      archivePrefix={arXiv},
      primaryClass={math.AP},
      url={https://arxiv.org/abs/2604.20464}, 
}

@article {IonescuK07,
    AUTHOR = {Ionescu, Alexandru D. and Kenig, Carlos E.},
     TITLE = {Global well-posedness of the {B}enjamin-{O}no equation in
              low-regularity spaces},
   JOURNAL = {J. Amer. Math. Soc.},
  FJOURNAL = {Journal of the American Mathematical Society},
    VOLUME = {20},
      YEAR = {2007},
    NUMBER = {3},
     PAGES = {753--798},
      ISSN = {0894-0347,1088-6834},
   MRCLASS = {35Q53 (35B30 76B55 76B70)},
  MRNUMBER = {2291918},
MRREVIEWER = {Grzegorz\ Karch},
       DOI = {10.1090/S0894-0347-06-00551-0},
       URL = {https://doi-org.proxy.lib.umich.edu/10.1090/S0894-0347-06-00551-0},
}

@article {Iorio86,
    AUTHOR = {I\'orio, Jr., Rafael Jos\'e},
     TITLE = {On the {C}auchy problem for the {B}enjamin-{O}no equation},
   JOURNAL = {Comm. Partial Differential Equations},
  FJOURNAL = {Communications in Partial Differential Equations},
    VOLUME = {11},
      YEAR = {1986},
    NUMBER = {10},
     PAGES = {1031--1081},
      ISSN = {0360-5302,1532-4133},
   MRCLASS = {35B40 (35K55 35Q20)},
  MRNUMBER = {847994},
MRREVIEWER = {Bjorn\ Birnir},
       DOI = {10.1080/03605308608820456},
       URL = {https://doi-org.proxy.lib.umich.edu/10.1080/03605308608820456},
}

@article{JorgeMS99,
title = {Modulation solutions for the {B}enjamin–{O}no equation},
journal = {Phys. D},
fjournal = {Physica D: Nonlinear Phenomena},
volume = {132},
number = {1},
pages = {1-18},
year = {1999},
issn = {0167-2789},
doi = {https://doi.org/10.1016/S0167-2789(99)00039-1},
url = {https://www.sciencedirect.com/science/article/pii/S0167278999000391},
author = {M.C. Jorge and A.A. Minzoni and Noel {F. Smyth}}
}

@article {KenigK03,
    AUTHOR = {Kenig, Carlos E. and Koenig, Kenneth D.},
     TITLE = {On the local well-posedness of the {B}enjamin-{O}no and
              modified {B}enjamin-{O}no equations},
   JOURNAL = {Math. Res. Lett.},
  FJOURNAL = {Mathematical Research Letters},
    VOLUME = {10},
      YEAR = {2003},
    NUMBER = {5-6},
     PAGES = {879--895},
      ISSN = {1073-2780},
   MRCLASS = {35Q53 (35B30 76B03)},
  MRNUMBER = {2025062},
MRREVIEWER = {Pierre\ A.\ Vuillermot},
       DOI = {10.4310/MRL.2003.v10.n6.a13},
       URL = {https://doi-org.proxy.lib.umich.edu/10.4310/MRL.2003.v10.n6.a13},
}

@article {KillipLV24,
    AUTHOR = {Killip, Rowan and Laurens, Thierry and Vi{\c{s}}an, Monica},
     TITLE = {Sharp well-posedness for the {B}enjamin-{O}no equation},
   JOURNAL = {Invent. Math.},
  FJOURNAL = {Inventiones Mathematicae},
    VOLUME = {236},
      YEAR = {2024},
    NUMBER = {3},
     PAGES = {999--1054},
      ISSN = {0020-9910,1432-1297},
   MRCLASS = {35C08},
  MRNUMBER = {4743514},
MRREVIEWER = {Francisco\ J.\ Vielma Leal},
       DOI = {10.1007/s00222-024-01250-8},
       URL = {https://doi-org.proxy.lib.umich.edu/10.1007/s00222-024-01250-8},
}

@article {KillipLV25,
    AUTHOR = {Killip, Rowan and Laurens, Thierry and Vi{\c{s}}an, Monica},
     TITLE = {Scaling-critical well-posedness for continuum
              {C}alogero-{M}oser models on the line},
   JOURNAL = {Commun. Am. Math. Soc.},
  FJOURNAL = {Communications of the American Mathematical Society},
    VOLUME = {5},
      YEAR = {2025},
     PAGES = {284--320},
      ISSN = {2692-3688},
   MRCLASS = {35Q51 (37K10)},
  MRNUMBER = {4922705},
MRREVIEWER = {Junyi\ Zhu},
       DOI = {10.1090/cams/48},
       URL = {https://doi-org.proxy.lib.umich.edu/10.1090/cams/48},
}

@article{Badreddine24gwp,
  author  = {Badreddine, R.},
  title   = {On the global well-posedness of the {C}alogero--{S}utherland derivative nonlinear {S}chr\"odinger equation},
  journal = {Pure Appl. Anal.},
  volume  = {6},
  number  = {2},
  pages   = {379--414},
  year    = {2024},
  doi     = {10.2140/paa.2024.6.379},
  eprint  = {2303.01087},
  archivePrefix = {arXiv}
}

@article{BadreddineKV26orbital,
      title={Orbital stability of {B}enjamin-{O}no multisolitons}, 
      author={R. Badreddine and R. Killip and M. Visan},
      year={2026},
      eprint={2509.14153},
      archivePrefix={arXiv},
      journal={To appear in Comm. Math. Phys.},
      primaryClass={math.AP},
      url={https://arxiv.org/abs/2509.14153}, 
}

@misc{BadreddineKV26asymptotic,
      title={Asymptotic stability of {B}enjamin-{O}no multisolitons in $L^2(\mathbb R)$}, 
      author={R. Badreddine and R. Killip and M. Visan},
      year={2026},
      eprint={2606.12951},
      archivePrefix={arXiv},
      primaryClass={math.AP},
      url={https://arxiv.org/abs/2606.12951}, 
}

@book{KleinS21nonlinear,
  author    = {Klein, C. and Saut, J.-C.},
  title     = {Nonlinear Dispersive Equations: Inverse Scattering and {PDE} Methods},
  series    = {Applied Mathematical Sciences},
  volume    = {209},
  publisher = {Springer},
  year      = {2021},
  doi       = {10.1007/978-3-030-91427-1}
}

@article{KochT03,
    AUTHOR = {Koch, H. and Tzvetkov, N.},
     TITLE = {On the local well-posedness of the {B}enjamin-{O}no equation
              in {$H^s({\Bbb R})$}},
   JOURNAL = {Int. Math. Res. Not.},
  FJOURNAL = {International Mathematics Research Notices},
      YEAR = {2003},
    NUMBER = {26},
     PAGES = {1449--1464},
      ISSN = {1073-7928,1687-0247},
   MRCLASS = {35Q53 (35B30)},
  MRNUMBER = {1976047},
MRREVIEWER = {W.\ Watzlawek},
       DOI = {10.1155/S1073792803211260},
       URL = {https://doi-org.proxy.lib.umich.edu/10.1155/S1073792803211260},
}

@article{KoopB81,
    author = {Koop, C. Gary and Butler, Gerald},
    title = {An investigation of internal solitary waves in a two-fluid system},
    journal = {J. Fluid Mech.},
    fjournal = {Journal of Fluid Mechanics}, 
    volume = {112},
    year = 1981,
    pages = {225--251},
}

@article{CLP1979,
author = {Chen, H.H. and Lee, Y.C. and Pereira, Nino},
title = {Algebraic internal wave solitons and the integrable {C}alogero–{M}oser–{S}utherland $N$‐body problem},
journal = {Phys. Fluids},
fjournal = {{The Physics of Fluids}},
volume = {22},
number = {187},
pages = {619--621},
year={1979},
doi = {10.1063/1.862457},
url = {},
  publisher={}
}

@book{Matsuno84,
    AUTHOR = {Matsuno, Yoshimasa},
     TITLE = {{B}ilinear {T}ransformation {M}ethod},
    SERIES = {Mathematics in Science and Engineering},
    VOLUME = {174},
 PUBLISHER = {Academic Press, Inc., Orlando, FL},
      YEAR = {1984},
     PAGES = {viii+223},
      ISBN = {0-12-480480-2},
   MRCLASS = {35Q20 (34-02 35Axx 58F07)},
  MRNUMBER = {759718},
MRREVIEWER = {Francesco\ Calogero},
}

@article {Matsuno98,
    AUTHOR = {Matsuno, Yoshimasa},
     TITLE = {The small dispersion limit of the {B}enjamin-{O}no equation
              and the evolution of a step initial condition},
   JOURNAL = {J. Phys. Soc. Japan},
  FJOURNAL = {Journal of the Physical Society of Japan},
    VOLUME = {67},
      YEAR = {1998},
    NUMBER = {6},
     PAGES = {1814--1817},
      ISSN = {0031-9015,1347-4073},
   MRCLASS = {76C10 (76V05)},
  MRNUMBER = {1632413},
       DOI = {10.1143/JPSJ.67.1814},
       URL = {https://doi-org.proxy.lib.umich.edu/10.1143/JPSJ.67.1814},
}

@Article{MillerX10,
  title =   {On the zero-dispersion limit of the {B}enjamin-{O}no {C}auchy problem for positive initial data},
  author =  {Peter D. Miller and Zhengjie Xu},
  journal = {Communications on Pure and Applied Mathematics},
  volume = 64,
  year = 2011,
  pages = {205--270}
  }

@Article{MillerW16a,
    title = {Direct scattering for the {B}enjamin-{O}no equation with rational initial data},
    author = {Peter D. Miller and Alfredo N. Wetzel},
    journal={Stud. Appl. Math.},
    fjournal = {Studies in Applied Mathematics},
    volume = 137,
    year = 2015,
    doi={10.1111/sapm.12101},
    pages = {53--69}
}

@article {MolinetP12,
    AUTHOR = {Molinet, Luc and Pilod, Didier},
     TITLE = {The {C}auchy problem for the {B}enjamin-{O}no equation in
              {$L^2$} revisited},
   JOURNAL = {Anal. PDE},
  FJOURNAL = {Analysis \& PDE},
    VOLUME = {5},
      YEAR = {2012},
    NUMBER = {2},
     PAGES = {365--395},
      ISSN = {2157-5045,1948-206X},
   MRCLASS = {35Q53 (35B30)},
  MRNUMBER = {2970711},
MRREVIEWER = {John\ Albert},
       DOI = {10.2140/apde.2012.5.365},
       URL = {https://doi-org.proxy.lib.umich.edu/10.2140/apde.2012.5.365},
}

@article {Nakamura79,
    AUTHOR = {Nakamura, Akira},
     TITLE = {B\"acklund transform and conservation laws of the
              {B}enjamin-{O}no equation},
   JOURNAL = {J. Phys. Soc. Japan},
  FJOURNAL = {Journal of the Physical Society of Japan},
    VOLUME = {47},
      YEAR = {1979},
    NUMBER = {4},
     PAGES = {1335--1340},
      ISSN = {0031-9015,1347-4073},
   MRCLASS = {35Q20 (35A30 58F07)},
  MRNUMBER = {550203},
MRREVIEWER = {A.\ Degasperis},
       DOI = {10.1143/JPSJ.47.1335},
       URL = {https://doi-org.proxy.lib.umich.edu/10.1143/JPSJ.47.1335},
}

@article {Ponce91,
    AUTHOR = {Ponce, Gustavo},
     TITLE = {On the global well-posedness of the {B}enjamin-{O}no equation},
   JOURNAL = {Differential Integral Equations},
  FJOURNAL = {Differential and Integral Equations. An International Journal
              for Theory and Applications},
    VOLUME = {4},
      YEAR = {1991},
    NUMBER = {3},
     PAGES = {527--542},
      ISSN = {0893-4983},
   MRCLASS = {35Q35 (76C05 76V05)},
  MRNUMBER = {1097916},
MRREVIEWER = {Piotr\ Biler},
}

@article{PorterS02,
    title = {Modelling the morning glory of the {G}ulf of {C}arpentaria},
    author = {Porter, Anne and Smyth, Noel F.},
    journal = {J. Fluid Mech.},
    fjournal = {Journal of Fluid Mechanics},
    volume = {454}, 
    pages = {1--20},
    doi = {10.1017/S0022112001007455},
    year={2002}
}

@article{ChoiC96,
	title={Weakly nonlinear internal waves in a two-fluid system},
    fjournal={Journal of Fluid Mechanics}, 
    journal = {J. Fluid Mech.},
    volume = {313}, 
    year = {1996}, 
    pages = {83--103},
    author = {Choi, Wooyong and Camassa, Roberto},
    DOI = {10.1017/S0022112096002133}
}

@article{ChristieMC81, title={Solitary waves in the lower atmosphere}, volume={293}, DOI={10.1038/293046a0}, number={5827}, journal={Nature (London)}, publisher={London: Nature Publishing Group UK}, author={Christie, D. R. and Muirhead, K. J. and Clarke, R. H.}, year={1981}, pages={46–49} 
}

@article{Paulsen24,
    author = {Paulsen, M.O.},
    title = {Justification of the {B}enjamin–{O}no equation as an internal water waves model},
    journal = {Ann. PDE},
    fjournal = {Annals of Partial Differential Equations},
    volume = {10}, 
    note = {article number 25},
    year = {2024}, 
    doi = {10.1007/s40818-024-00190-z}
}

@article{SatsumaI79,
author = {Satsuma ,Junkichi and Ishimori ,Y{\=u}ji},
title = {Periodic Wave and Rational Soliton Solutions of the {B}enjamin-{O}no Equation},
journal = {Journal of the Physical Society of Japan},
volume = {46},
number = {2},
pages = {681-687},
year = {1979},
doi = {10.1143/JPSJ.46.681},

URL = { 
    
        https://doi.org/10.1143/JPSJ.46.681
    
    

},
eprint = { 
    
        https://doi.org/10.1143/JPSJ.46.681
    
    

}
}

@article {Saut79,
    AUTHOR = {Saut, J.-C.},
     TITLE = {Sur quelques g\'en\'eralisations de l'\'equation de
              {K}orteweg-de {V}ries},
   JOURNAL = {J. Math. Pures Appl. (9)},
  FJOURNAL = {Journal de Math\'ematiques Pures et Appliqu\'ees. Neuvi\`eme
              S\'erie},
    VOLUME = {58},
      YEAR = {1979},
    NUMBER = {1},
     PAGES = {21--61},
      ISSN = {0021-7824,1776-3371},
   MRCLASS = {35Q20 (76B99)},
  MRNUMBER = {533234},
MRREVIEWER = {Amy\ Cohen},
}

@article {Tao04,
    AUTHOR = {Tao, Terence},
     TITLE = {Global well-posedness of the {B}enjamin-{O}no equation in
              {$H^1(\mathbb{R})$}},
   JOURNAL = {J. Hyperbolic Differ. Equ.},
  FJOURNAL = {Journal of Hyperbolic Differential Equations},
    VOLUME = {1},
      YEAR = {2004},
    NUMBER = {1},
     PAGES = {27--49},
      ISSN = {0219-8916,1793-6993},
   MRCLASS = {35Q53 (35B30 76B03)},
  MRNUMBER = {2052470},
MRREVIEWER = {Grzegorz\ Karch},
       DOI = {10.1142/S0219891604000032},
       URL = {https://doi-org.proxy.lib.umich.edu/10.1142/S0219891604000032},
}

@book {Talbut21,
    AUTHOR = {Talbut, Blaine Aaron Foster},
     TITLE = {{B}enjamin-{O}no at Low Regularity: An Integrability
              Approach},
      NOTE = {Thesis (Ph.D.)--University of California, Los Angeles},
 PUBLISHER = {ProQuest LLC, Ann Arbor, MI},
      YEAR = {2021},
     PAGES = {59},
      ISBN = {979-8759-96416-2},
   MRCLASS = {99-05},
  MRNUMBER = {4380591},
       URL =
              {http://gateway.proquest.com.proxy.lib.umich.edu/openurl?url_ver=Z39.88-2004&rft_val_fmt=info:ofi/fmt:kev:mtx:dissertation&res_dat=xri:pqm&rft_dat=xri:pqdiss:28863669},
}

@article {Wu16a,
    AUTHOR = {Wu, Yilun},
     TITLE = {Simplicity and finiteness of discrete spectrum of the
              {B}enjamin-{O}no scattering operator},
   JOURNAL = {SIAM J. Math. Anal.},
  FJOURNAL = {SIAM Journal on Mathematical Analysis},
    VOLUME = {48},
      YEAR = {2016},
    NUMBER = {2},
     PAGES = {1348--1367},
      ISSN = {0036-1410,1095-7154},
   MRCLASS = {35P25 (35Q53 35R30 47A40)},
  MRNUMBER = {3484397},
       DOI = {10.1137/15M1030649},
       URL = {https://doi.org/10.1137/15M1030649},
}

@article {Wu17,
author = {Wu, Yilun},
title = {{J}ost solutions the direct scattering problem of the {B}enjamin-{O}no equation},
journal = {SIAM J. Math. Anal.},
fjournal= {SIAM Journal on Mathematical Analysis},
volume = {49},
year = {2017},
number = {6},
doi={10.1137/17M1124528},
pages = {5158--5206}
}

@article {KollarM14,
AUTHOR = {Koll\'ar, Richard and Miller, Peter D},
TITLE = {Graphical {K}rein signature theory and {E}vans-{K}rein functions},
JOURNAL = {SIAM Rev.},
FJOURNAL = {SIAM Review},
VOLUME = {56}, 
YEAR = {2014},
NUMBER = {1},
PAGES = {73--123}
}

@article {KaupM98,
author = {Kaup, David J and Matsuno, Yoshimasa},
title = {The inverse scattering transform for the {B}enjamin-{O}no equation},
journal = {Stud. Appl. Math.},
fjournal = {Studies in Applied Mathematics},
volume = {101},
pages = {73--98},
year = {1998},
DOI = {10.1111/1467-9590.00086}
}

@misc{DLMF,
         key = "{\relax DLMF}",
       title = "{\it NIST Digital Library of Mathematical Functions}",
howpublished = "\url{https://dlmf.nist.gov/}, Release 1.2.4 of 2025-03-15",
         url = "https://dlmf.nist.gov/",
        note = "F.~W.~J. Olver, A.~B. {Olde Daalhuis}, D.~W. Lozier, B.~I. Schneider,
                R.~F. Boisvert, C.~W. Clark, B.~R. Miller, B.~V. Saunders,
                H.~S. Cohl, and M.~A. McClain, eds."}

@article {Whitham65a,
    AUTHOR = {Whitham, G. B.},
     TITLE = {Non-linear dispersive waves},
   JOURNAL = {Proc. Roy. Soc. London Ser. A},
  FJOURNAL = {Proceedings of the Royal Society. London. Series A.
              Mathematical, Physical and Engineering Sciences},
    VOLUME = {283},
      YEAR = {1965},
     PAGES = {238--261},
      ISSN = {0962-8444,2053-9169},
   MRCLASS = {76.35},
  MRNUMBER = {176724},
MRREVIEWER = {F.\ Ursell},
       DOI = {10.1098/rspa.1965.0019},
       URL = {https://doi-org.proxy.lib.umich.edu/10.1098/rspa.1965.0019},
}

@article {Whitham65b,
    AUTHOR = {Whitham, G. B.},
     TITLE = {A general approach to linear and non-linear dispersive waves
              using a {L}agrangian},
   JOURNAL = {J. Fluid Mech.},
  FJOURNAL = {Journal of Fluid Mechanics},
    VOLUME = {22},
      YEAR = {1965},
     PAGES = {273--283},
      ISSN = {0022-1120,1469-7645},
   MRCLASS = {76.35},
  MRNUMBER = {182236},
MRREVIEWER = {H.\ P.\ Greenspan},
       DOI = {10.1017/S0022112065000745},
       URL = {https://doi-org.proxy.lib.umich.edu/10.1017/S0022112065000745},
}

@article {AblowitzS77,
    AUTHOR = {Ablowitz, M. J. and Segur, H.},
     TITLE = {Asymptotic solutions of the {K}orteweg-de {V}ries equation},
   JOURNAL = {Stud. Appl. Math.},
  FJOURNAL = {Studies in Applied Mathematics},
    VOLUME = {57},
      YEAR = {1976/77},
    NUMBER = {1},
     PAGES = {13--44},
      ISSN = {0022-2526,1467-9590},
   MRCLASS = {35Q99 (35B40)},
  MRNUMBER = {481656},
MRREVIEWER = {Alan\ Jeffrey},
       DOI = {10.1002/sapm197757113},
}

@article{Lax59,
 author = {Lax, Peter D.},
 title = {Translation invariant spaces},
 fjournal = {Acta Mathematica},
 journal = {Acta Math.},
 issn = {0001-5962},
 volume = {101},
 pages = {163--178},
 year = {1959},
 language = {English},
 doi = {10.1007/BF02559553},
 zbMATH = {3139141},
 Zbl = {0085.09102}
}

@article {LL,
    AUTHOR = {Lax, Peter D. and Levermore, C. David},
     TITLE = {The small dispersion limit of the {K}orteweg-de {V}ries
              equation. {I,II,III}},
   JOURNAL = {Comm. Pure Appl. Math.},
  FJOURNAL = {Communications on Pure and Applied Mathematics},
    VOLUME = {36},
      YEAR = {1983},
     PAGES = {253--290, 571--593, 809--829},
      ISSN = {0010-3640},
   MRCLASS = {35Q20},
}

@article{GKT23,
 author = {Gérard, P. and Kappeler, T. and Topalov, P.},
 title = {Sharp well-posedness results of the {Benjamin}-{Ono} equation in {{\(H^s(\mathbb{T},\mathbb{R})\)}} and qualitative properties of its solutions},
 fjournal = {Acta Mathematica},
 journal = {Acta Math.},
 issn = {0001-5962},
 volume = {231},
 number = {1},
 pages = {31--88},
 year = {2023},
 language = {English},
 doi = {10.4310/ACTA.2023.v231.n1.a2},
 zbMATH = {7779089},
 Zbl = {1533.35061}
}

@article{Sun2020,
 title={{Complete integrability of the Benjamin-Ono equation on the multi-soliton manifolds}},
  author={Sun, Ruoci},
DOI={10.1007/s00220-021-03996-1},
volume={383},
journal={Comm. Math. Phys.},
fjournal={Communications in Mathematical Physics},
publisher={Springer},
year={2021},
month={4},
pages={1051--1092}
}

\end{document}